\documentclass{article}

\usepackage{arxiv}

\usepackage[utf8]{inputenc} 
\usepackage[T1]{fontenc}    
\usepackage{hyperref}       
\usepackage{url}            
\usepackage{booktabs}       
\usepackage{amsfonts}       
\usepackage{nicefrac}       
\usepackage{microtype}      
\usepackage{graphicx}

\usepackage{enumerate}
\usepackage{amsopn, amstext, amscd, amsfonts, amssymb, amsthm}
\usepackage{url, hyperref, verbatim}
\usepackage{etex}

\usepackage{afterpage}
\usepackage{pdflscape}
\usepackage{multirow}
\usepackage[short]{optidef}
\usepackage{pifont}

\usepackage{amssymb, amsmath}

\newtheorem{theorem}{Theorem}[section]

\newtheorem{lemma}[theorem]{Lemma}

\newtheorem{rule-thm}[theorem]{Rule}

\title{A revisited branch-and-cut algorithm for large-scale orienteering problems}

\author{ \href{https://orcid.org/0000-0002-8669-4482}{\includegraphics[scale=0.06]{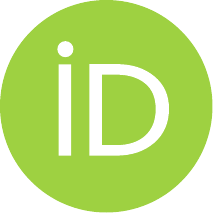}\hspace{1mm}Gorka Kobeaga} \\
  Basque Center for Applied Mathematics BCAM\\
  \texttt{gkobeaga@bcamath.org} \\
  \And
  \href{https://orcid.org/0000-0002-4947-2784}{\includegraphics[scale=0.06]{orcid.pdf}\hspace{1mm}Mar\'ia Merino} \\
  University of the Basque Country UPV/EHU\\
  \texttt{maria.merino@ehu.eus} \\
  \And
  \href{https://orcid.org/0000-0002-4683-8111}{\includegraphics[scale=0.06]{orcid.pdf}\hspace{1mm}Jose A. Lozano} \\
  Basque Center for Applied Mathematics BCAM\\
  University of the Basque Country UPV/EHU\\
  \texttt{jlozano@bcamath.org} \\
}

\hypersetup{
  pdftitle={A revisited branch-and-cut algorithm for large-scale orienteering problems},
  pdfsubject={math.OC},
  pdfauthor={G. Kobeaga, M. Merino, J.A. Lozano},
  pdfkeywords={orienteering problem, branch-and-cut, large problems}
}

\begin{document}
\maketitle

\begin{abstract}
  The orienteering problem is a route optimization problem which consists in finding a simple cycle that maximizes the total collected profit subject to a maximum distance limitation. In the last few decades, the occurrence of this problem in real-life applications has boosted the development of many heuristic algorithms to solve it. However, during the same period, not much research has been devoted to the field of exact algorithms for the orienteering problem. The aim of this work is to develop an exact method which is able to obtain optimality certification in a wider set of instances than with previous methods, or to improve the lower and upper bounds in its disability.

  We propose a revisited version of the branch-and-cut algorithm for the orienteering problem which includes new contributions in the separation algorithms of inequalities stemming from the cycle problem, in the separation loop, in the variables pricing, and in the calculation of the lower and upper bounds of the problem. Our proposal is compared to three state-of-the-art algorithms on 258 benchmark instances with up to 7397 nodes. The computational experiments show the relevance of the designed components where 18 new optima, 76 new best-known solutions and 85 new upper-bound values were obtained.
\end{abstract}


\begin{section}{Introduction}
  The Orienteering Problem (OP) is a well-known routing problem proposed in the 80s, see~\cite{tsiligirides1984} and~\cite{GoLeVo87}. Given a weighted complete graph with vertex profits
  and a constant $d_0$, the goal is to find the cycle which, with a length lower than or equal to $d_0$,  maximizes the sum of profit of the visited vertices. In addition to the length constraint, every feasible cycle solution must visit a given  depot vertex.

  The OP can be seen as a combination of the Knapsack Problem (KP) and the Travelling Salesperson Problem (TSP). Given a set of items with an assigned weight and profit and a constant $w_0$, the goal in KP is to find the subset of items which, with a total weight  lower than or equal to $w_0$, maximizes the sum of the profit of subset items. In the KP, the feasibility of a subset is checked in linear time. In the OP, however, the feasibility of a solution is checked by solving a TSP-decision problem. A subset of vertices is feasible if there exists a cycle (Hamiltonian in the subgraph obtained by the vertices) whose length does not exceed $d_0$, finding such a cycle is an NP-complete problem. This simple but non-trivial combination of two NP-hard problems makes the OP an interesting problem to study.

  The occurrence of the OP in many real-life applications, such as logistics and tourism, has boosted the emergence of many variants and algorithms to solve the problem over the last decades. A survey of OP variants, approaches, and applications can be seen in~\cite{Vansteenwegen2019}. In this paper, we have focused on solving the classical problem through an exact algorithm.

  Some exact algorithms have been proposed for the OP, see \cite{LAPORTE1990193}, \cite{ramesh1992}, \cite{leifer1994}, and \cite{Gendreau1998ABA}. The most competitive approach thus far was proposed by~\cite{fischetti5} two decades ago. To our knowledge, no exact algorithm for the classical OP has been published after this work. The first exact algorithm, a Branch-and-Bound (B\&B) algorithm, was published in~\cite{LAPORTE1990193} where bounds for the problem were provided based on the Knapsack relaxation of the OP\@. In~\cite{ramesh1992}, new bounds for the B\&B were obtained by Lagrangian relaxation. In~\cite{leifer1994} a Branch-and-Cut (B\&C) algorithm was proposed, which included logical, connectivity, and cover cuts for the first time. In~\cite{Gendreau1998ABA}~a B\&C was proposed for a variant of the OP which considers multiple depot nodes.
  The B\&C approach in~\cite{fischetti5} outperformed the previous ones in middle-sized OP instances by considering column generation, new valid inequalities (cycle cover and path inequalities), problem-specific separation algorithms, and an efficient primal heuristic.
  In the last two decades, authors dealing with exact approaches have focused on solving variants of the problem, such as Team OP~(\cite{boussier2007}, \cite{poggi2010}, \cite{dang2013}, \cite{keshtkaran2015}, \cite{bianchessi2018}), Arc OP~(\cite{archetti2016}), Team Arc OP~(\cite{archetti2014}, \cite{rieraledesma2017}) and Probabilistic OP~(\cite{ANGELELLI2017269}).

  Recent results for large-sized instances of the OP, presented in~\cite{kobeaga2018} and~\cite{santini2019}, have shown that the state-of-the-art B\&C algorithm in~\cite{fischetti5} does not obtain satisfactory results when the number of nodes is larger than 400. In half of the large-sized instances, the named B\&C algorithm does not produce any output. In many of the other half of instances, the returned solution value is far from the values obtained by the heuristic algorithms.
  The motivation of this work is to develop a B\&C algorithm which is able to improve the values of the best-known lower and upper bounds in the literature, and if possible, to obtain optimality certifications in a wider set of instances than the state-of-the-art B\&C algorithm.
  In our view, there is room for improvement for B\&C algorithms in the case of OP, mainly if we consider that some of the successful techniques developed for the TSP, such as shrinking and efficient pricing, have not yet been adapted for the OP\@. This paper is an attempt to combine the adaptation of some of those techniques with our OP specific contributions.

  In this work, we have developed and adapted techniques to scale the B\&C algorithm to large OP problems. Our main contributions in this paper are the following:

  \begin{itemize}
    \item Develop a new joint separation algorithm for Subcycle Elimination Constraints and Connectivity Constraints, which efficiently uses the shrinking techniques in \cite{kobeaga2020} to speed up the algorithm by reducing the adverse effects of the shrinking for Connectivity Constraints (Section ~\ref{sec:seccc}).
    \item Design blossom separation algorithms for Cycle Problems, which generalize the \cite{Padberg1980} and \cite{Grotschel1991} heuristics, providing a considerable improvement in the solution quality and running time (Section \ref{sec:blossom}).
    \item Design an efficient variable pricing procedure for the OP inspired by the one developed in \cite{concorde} for the TSP. It enables repetitive calculations to be avoided and the exact calculation of the reduced cost of some variables to be skipped (Section \ref{sec:pricing}).

    \item Conceive a three-level separation loop for the OP, which treats the cuts according to their relevance
      while giving the same chance to the cuts with complementary incidence. This significantly reduces the running time of the B\&C algorithm (Section \ref{sec:sep}).

    \item Devise a combination of two alternative primal heuristics, a greedy one (\cite{fischetti5}) in the separation loop and a metaheuristic based one (\cite{kobeaga2018}) at the beginning of branch nodes. The new combination boosts the quality of the obtained solutions in large-sized problems (Section \ref{sec:heur}).

    \item Formulate the computation of the global upper bound in the branching phase for the OP, which enables the upper bound obtained in the branching root node to be updated (Section \ref{sec:branching}).
  \end{itemize}

  The computational experiments presented in this paper show the importance of the proposed techniques and components for the B\&C algorithm for the OP. We compare our B\&C algorithm with three state-of-the-art algorithms, two heuristics (\cite{kobeaga2018} and~\cite{santini2019}) and one exact (\cite{fischetti5}), on 258 benchmark instances with up to 7397 nodes.
  In comparison to the literature, the revisited B\&C obtains:

  \begin{itemize}
    \item 180 optimum values, from which 18 are new.
    \item 245 best-known solution values, from which 76 are new.
    \item 249 best-known upper-bound values, from which 85 are new.
  \end{itemize}

  The rest of the paper is organized as follows. In Section~\ref{sec:modelling} the 0{-}1 Integer Linear model of the OP is introduced. In Section~\ref{sec:valid} we present the valid inequalities for the OP\@. In Section~\ref{sec:bac} we detail the proposed B\&C algorithm for the problem. In Section~\ref{sec:exp} the results of the computational experiments are shown. The detailed experimental results can be found in the appendices.

  \section{OP Modelling and Polyhedral Considerations}\label{sec:modelling}
  The OP can be defined by a 5-tuple $\langle G, d, s, 1, d_0 \rangle$, where $G=K_n= ( V, E )$ is a complete graph with vertex set $V$ and edge set $E$; $d=(d_e)$ where $d_e$ is the positive distance value (time or weight) associated to each $e \in E$; $s=(s_v)$, where $s_v$ is a positive value that represents the score (profit) of vertex $v \in V$; $1 \in V$ is a vertex selected as the depot; and $d_0$ is a positive value that limits the cycle length.

  Let us define the following sets:
  \begin{subequations}
    \begin{align}
      (Q:W) & : = \{[u, v] \in E:  u \in Q, v \in W\} & Q, W \subset V \\
      \delta(Q) & : = (Q:V-Q) & Q \subset V \\
      E(Q) & : = (Q:Q) & Q \subset V  \\
      V(T) & : = \{v \in V:  T \cap (v : V) \neq \emptyset\} & T \subset E
    \end{align}
  \end{subequations}
  \noindent where  $(Q:W)$ are the edges connecting $Q$ and $W$, $\delta(Q)$ is the set of edges in the coboundary of $Q$ also known as the star-set of $Q$, $E(Q)$ is the set of edges between the vertices of $Q$, and $V(T)$ is the set of vertices incident with an edge set $T$. For simplicity, we sometimes denote $\{e\}$ and $\{v\}$ by $e$ and $v$, respectively, e.g., $\delta(v)$ and $V(e)$.

  We denote by $\mathbb{R}^{V}$, and $\mathbb{R}^{E}$, the space of real vectors whose components are indexed by elements of $V$, and $E$ respectively.
  In the model of OP, two types of decision variables are used, $y=(y_v) \in \mathbb{R}^{V}$ and $x=(x_e) \in \mathbb{R}^{E}$, associated with the nodes and edges of G, respectively, where:

  \begin{equation}
    y_v = \left\{ \begin{aligned}
        1 & \quad \; \mbox{if node $v$ is visited} \\
        0 & \quad \; \mbox{otherwise}
    \end{aligned} \right. \notag
    \quad
    x_{e}= \left\{ \begin{aligned}
        1 & \quad \;\mbox{if edge $e$ is traversed } \\
        0 & \quad \; \mbox{otherwise}
    \end{aligned} \right.
  \end{equation}

  \noindent For $(y,x) \in \mathbb{R}^{V \times E} $, $S\subset V$ and $T\subset E$, we define $y(S) = \sum_{v \in S} y_v$ and $x(T) = \sum_{e \in T} x_e$.

  The OP goal is to determine a simple cycle that maximizes the sum of the scores of the visited vertices, such that it contains the depot node $1 \in V$ and whose length is equal to or lower than the distance limitation, $d_0$. Then, the OP can be formulated as the following 0-1 Integer Linear model:

  \begin{maxi!}
    {}{\sum_{v \in V} s_v y_v\label{op:f}}
    {\label{op}}{}
    \addConstraint{\sum_{e \in E} d_e x_e}{\leq d_0 \label{op:d0}}
    \addConstraint{x(\delta(v)) - 2y_v}{=0\label{op:deg},}{ \qquad v \in V}
    \addConstraint{x(\delta(H)) - 2y_l-2y_r}{\geq -2\label{op:sec},}{ \qquad l \in H \subset V, r \in V-H }
    \addConstraint{}{\notag}{ \qquad  3 \leq |H| \leq |V|-3\; }
    \addConstraint{y_v - x_e}{\geq 0\label{op:log},}{ \qquad v \in V, e \in \delta(v)}
    \addConstraint{ 0 \leq y_v}{ \leq 1,\label{op:vbnd}}{ \qquad v\in V}
    \addConstraint{ 0 \leq x_e}{ \leq 1,\label{op:ebnd}}{ \qquad e\in E}
    \addConstraint{y_1}{= 1\label{op:depot}}{}
    \addConstraint{  x_e}{ \in \mathbb{Z}\label{op:eint}}{ \qquad e\in E}
  \end{maxi!}

  \noindent where the objective function~\eqref{op:f} is to maximize the total collected profit. The constraint~\eqref{op:d0} limits the total cycle length.
  The degree equations~\eqref{op:deg}, together with the logical constraints~\eqref{op:log} and the integrality constraints~\eqref{op:eint}, ensure that the visited vertices have exactly two incident edges and the unvisited vertices none.
  The Subcycle Elimination Constraints (SEC)~\eqref{op:sec} ensure that only one connected cycle exists. Throughout the paper, we use the notation $\langle H, l, r\rangle$ for the SEC defined by the set $H \subset V$ and the vertices $l\in H$ and $r\notin H$.  The constraints \eqref{op:ebnd} and~\eqref{op:eint} impose that the edge variables are 0-1, consequently, considering these together with the Logical Constraints  \eqref{op:log} and the bounds \eqref{op:ebnd}, the vertex variables are also 0-1. The constraint~\eqref{op:depot} defines the depot condition.

  As mentioned in the introduction, the OP can be seen as a combination of the TSP-decision and the KP problems. Particularly, the OP is a Cycle Problem (CP) where the solutions, which are cycles, need to satisfy a certain length constraint. This relation with the two classical optimization problems is useful when identifying the valid inequalities and their respective separation algorithms for OP. Let us show how the solution space of OP is related to those well-known problems. The OP Polytope ($P_{OP}$) of the complete graph $K_n$ is defined by:
  \begin{equation}
    P_{{OP}} := conv\{(y,x) \in \mathbb{R}^{V\times E}: (y,x)\text{~satisfies~\eqref{op:d0},~\eqref{op:deg},~\eqref{op:sec},~\eqref{op:log},~\eqref{op:vbnd},~\eqref{op:ebnd},~\eqref{op:depot},~\eqref{op:eint}}\}
  \end{equation}

  The Knapsack Polytope ($P_{{KP}}$), see \cite{balas1975}, is a well-studied polytope closely related to the $P_{OP}$:
  \begin{equation}
    P_{KP} := conv\{x \in \mathbb{R}^{E}: x\text{~satisfies~\eqref{op:d0},~\eqref{op:ebnd},~\eqref{op:eint}}\}
  \end{equation}

  Since the solutions of the OP are cycles, the Cycle Polytope ($P_{CP}$), plays a crucial role when solving the OP with B\&C. Based on~\cite{bauer1997}, the $P_{CP}$ can be characterized as:
  \begin{equation}
    P_{{CP}}  :=  conv \{(y,x) \in \mathbb{R}^{V \times E}: (y,x) \text{~satisfies~\eqref{op:deg},~\eqref{op:sec}, $x(E) \geq 3$,~\eqref{op:vbnd},~\eqref{op:ebnd},~\eqref{op:eint}}\}
  \end{equation}

  We have the following relationship:
  \begin{equation}
    P_{OP} \subset P_{{CP}} \cap ( \mathbb{R}^{V} \times P_{{KP}}) \cap \{(y,x) \in \mathbb{R}^{V \times E}: y_1 = 1 \}  \label{cp:incl}
  \end{equation}

  Consequently, the potential valid inequalities for the OP are those which are valid for $P_{{CP}}$ and the $P_{{KP}}$.
  However, the $P_{OP}$ and the intersected polytopes in the relationship~\eqref{cp:incl} are not equal and alternative valid inequalities
  are needed to deal with the OP. An example of a point in $P_{CP} \cap ( \mathbb{R}^{V} \times P_{{KP}} ) \cap \{(y,x): y_1 = 1 \}$ but not in $P_{OP}$ is given in Figure 2 of~\cite{fischetti5}.

\end{section}

\begin{section}{Valid Inequalities}\label{sec:valid}
  In this section, we present valid inequalities for the $OP$. The straightforward inequalities, as motivated in Section $\ref{sec:modelling}$, are based on the $P_{KP}$ (Edge Cover inequalities) and $P_{CP}$ (Comb inequalities) relaxations of the $P_{OP}$ and they were mainly proposed in~\cite{fischetti5} and~\cite{Gendreau1998ABA}.
  Additional valid inequalities to those based on $P_{KP}$ and $P_{CP}$ have also been proposed in the literature: the Connectivity Constraints in~\cite{leifer1994}, the Vertex Cover inequalities in~\cite{Gendreau1998ABA}, and the Cycle Cover and the Path inequalities in~\cite{fischetti5}. The novelty of this section is an alternative representation of comb inequalities, which is then used for the efficient pricing in Section~\ref{sec:pricing}.

  \begin{subsection}{Connectivity Constraints}
    The Connectivity Constraints (CC) are well-known inequalities for the OP, e.g.~\cite{Gendreau1998ABA} and \cite{leifer1994}, and are a particular case of the conditional cuts proposed in~\cite{fischetti5}. The CCs exploit the depot constraint~\eqref{op:depot}.
    Given a lower bound, LB, of the OP, let $T$ be a subset of nodes such that $1 \in T$, $|T| \geq 2$ and  $\sum_{v \in T} s_v \leq LB$. The inequality defined by $T$
    \begin{equation}
      x(\delta(T)) \geq 2 \label{valid:cc}
    \end{equation}
    \noindent is valid for the OP\@. Since $x(\delta(T)) = x(\delta(V-T))$, the inequality can also be defined for $T\subset V$ such that $1 \notin T$ and  $\sum_{v \notin T} s_v \leq LB$. So, it is always possible to assume that $|T| \leq |V|/2$.

    \begin{subsection}{Comb Inequalities}
      The comb inequalities
      were generalized from the TSP to cycle problems in~\cite{bauer1997}.
      A comb is a tuple $\langle H, \{ T_1, \ldots, T_t \}, L, R \rangle$ of three vertex subsets and a family $\mathcal{T}=\{ T_1, \ldots, T_t \}$ of vertex subsets such that satisfies the following properties:
      \begin{enumerate}[i)]
        \item $t\geq 3$ and an odd integer
        \item $T_i \cap T_j = \emptyset$ for $1 \leq i < j \leq t$
        \item $T_i \cap H \neq \emptyset$ and $T_i  - H \neq \emptyset$ for $i =1 ,\ldots, t$
        \item  $L=\{l_i\}$ such that $l_i \in T_i \cap H$ for $i =1 ,\ldots, t$
        \item  $R=\{r_i\}$ such that $r_i \in T_i - H$ for $i =1 ,\ldots, t$
      \end{enumerate}

      The set $H$ is called the handle, the sets in $\mathcal{T}$ are called the teeth, the set $R$ is called the {\em Root\/} set, and $L$ is called the {\em Link\/} set. Then, the inequality

      \begin{equation}
        x(\delta(H)) + \sum_{j=1}^t x(\delta(T_j)) - 2y(R) - 2y(L) \geq 1 - t \label{valid:comb}
      \end{equation}

      \noindent is facet-defining for $P_{CP}$, as was shown in~\cite{bauer1997}, and therefore, a valid inequality for $OP$.
      When all the teeth consist of exactly two vertices, the comb inequalities are known as blossom inequalities.

    \end{subsection}

    \begin{subsection}{Edge Cover Inequalities}
      The maximum length constraint \eqref{op:d0}, which is a capacity constraint for the edge variables, defines a $KP$ polytope, as explained in Section~\ref{sec:modelling}. For every feasible $(y, x)$, the edge variable, $x$, belongs to $P_{KP}$. For the OP, the Edge Cover inequalities are the cover inequalities of the associated $P_{KP}$ (\cite{balas1975}). These inequalities were first introduced for the OP in~\cite{leifer1994} and also used in~\cite{fischetti5} and~\cite{Gendreau1998ABA}. Let $F \subset E$ be a subset with $\sum_{e \in F} d_e > d_0$, then:
      \begin{equation}
        x(F) \leq |F|-1 \label{valid:ecover}
      \end{equation}
      defines an Edge Cover inequality for the OP\@. We assume that $F$ is a minimal cover, i.e.~for every $F_0 \subsetneq F$, we have $\sum_{e \in F_0} d_e \leq d_0$.

    \end{subsection}

    \begin{subsection}{Cycle Cover Inequalities}
      Every feasible cycle $F \subset E$ satisfies the equation $x(F) = y(V(F))$. Let $F \subset E$ be a subset that defines a cycle with $\sum_{e \in F} d_e > d_0$, then the inequality
      \begin{equation}
        x(F) \leq y(V(F))-1 \label{valid:ccover}
      \end{equation}
      is valid for the OP\@. These cuts were used in~\cite{fischetti5} and~\cite{Gendreau1998ABA}.

    \end{subsection}

    \begin{subsection}{Vertex Cover Inequalities}
      Let UB be an upper bound of the OP and $Q \subset V$ be a subset with $\sum_{v \in Q} s_v > UB$, then:
      \begin{equation}
        y(Q) \leq |Q|-1 \label{valid:vcover}
      \end{equation}
      defines a Vertex Cover inequality for the OP\@. We assume that $S$ is a minimal cover. These inequalities were first used for the OP in~\cite{Gendreau1998ABA}.

    \end{subsection}

    \begin{subsection}{Path Inequalities}
      The goal of these cuts is to exclude the paths that due to the length constraint~\eqref{op:d0} cannot be part of a feasible solution.
      Let $P=\{[i_1, i_2], [i_2, i_3], \ldots, [i_{k-1}, i_k]\}$ be any simple path through $V(P)= \{i_1, \ldots, i_k\} \subset V - \{1\}$, and define the vertex set:

      \begin{equation}
        W(P) := \{v \in V  - V(P) : d_{1, i_1} +\sum_{e \in P} d_e + d_{i_k, v} + d_{v,1} \leq d_0\}
      \end{equation}
      Then the following Path inequality
      \begin{equation}
        x(P) - y(V(P)) + y_1 + y_k  - \sum_{v \in W(P)} x_{i_k, v} \leq 0 \label{valid:path}
      \end{equation}

      \noindent is valid for the OP, see~\cite{fischetti5}.
    \end{subsection}

  \end{subsection}

\end{section}

\section{Branch-and-Cut Algorithm}\label{sec:bac}
  In this section, we present the principal contributions of this paper.  These contributions deal with the separation algorithms of inequalities stemming from the cycle problem (SECs and comb inequalities), the design of the separation loop, the pricing of variables for the column generation and the calculation of the lower and upper bounds of the problem.
  In Figure~\ref{fig:scheme} a flowchart representing a simplified B\&C algorithm can be consulted.
\begin{figure}[htb!]
  \centering
  \includegraphics[width=\columnwidth]{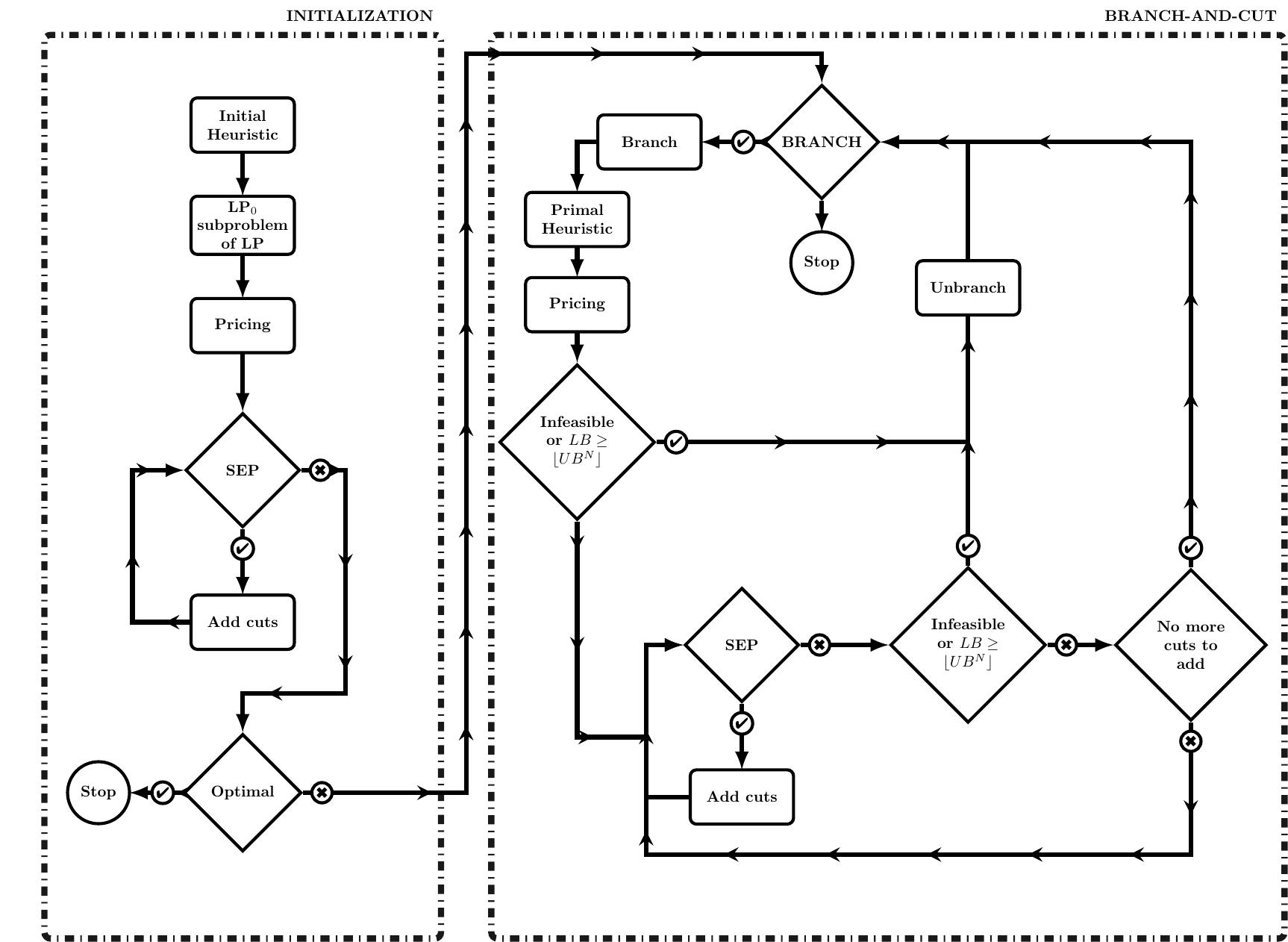}
  \caption{Flowchart of the Branch-and-Cut algorithm considered in this work. BRANCH is an oracle which returns an unevaluated node in the branching tree. SEP refers to the separation algorithms. At each action box of the flowchart the subproblem $LP_0$ is updated and solved.}\label{fig:scheme}
\end{figure}

  \begin{subsection}{Initialization}\label{bac:init}

    First of all, we obtain an initial heuristic solution. To that aim, we make use of the EA4OP metaheuristic in \cite{kobeaga2018} considering a small size population.

    Next, we build the initial subproblem, LP$_0$. Given the computational requirements of considering all the variables and constraints that define the OP, an initial subproblem LP$_0$ is built. The LP$_0$
    is initialized considering the following subset of constraints and variables:
    \begin{enumerate}[i)]
      \item All the vertex variables.
      \item Edges in the 10 nearest neighborhood graph.
      \item Maximum length constraint~\eqref{op:d0}, degree constraints~\eqref{op:deg}, and depot constraint~\eqref{op:depot}.
      \item Variable bounds,~\eqref{op:vbnd} and~\eqref{op:ebnd}.
    \end{enumerate}

    Immediately after the initialization, the edge variables are priced, see Section~\ref{sec:pricing}.
    In the rest of the paper, we use the LP$_0$ symbol to refer to any subproblem of the OP, regardless of whether it is the initial one or not.

  \end{subsection}

  \begin{subsection}{Separation Algorithms}\label{separation}

    In this section, we present the heuristic and exact separation algorithms used to find the violated inequalities.
    Our contributions are concentrated in the separation algorithms for SECs, CCs and blossom inequalities. Hence, we only give details of these separation algorithms in the section. The details of separation algorithms for the rest of the inequalities (Logical Constraints, Edge Cover, Vertex Cover, Cycle Cover, and Path inequalities) can be found in~\cite{fischetti5}.

    Let  $(y^{*}, x^{*})$ be a solution of a particular LP$_0$ problem and define $V^{*}=\{v \in V : y^{*}_v >0\}$ and  $E^{*}=\{e \in E : x^{*}_e >0\}$. Then, $G^{*}=(V^{*}, E^{*})$ is called the support graph associated with the solution $(y^{*}, x^{*})$.

    \begin{subsubsection}{SECs and CCs}\label{sec:seccc}
      Violated SECs \eqref{op:sec} and CCs \eqref{valid:cc} are found using a common separation algorithm. This is natural since, in both constraint families, the incidence vector of the arcs, $x$ in the inequality can be written as the star-set value, $x(\delta(Q))$ of a subset $Q$ of vertices. Since $\delta(Q)$ is the cut associated with $Q$, the separations of both inequalities are closely related to the minimum cut problem.
      In \cite{kobeaga2020} it was shown that the shrinking techniques substantially speed up the SEC separation algorithms. However, as explained below, the shrinking might also have a negative impact on the finding of violated CCs. In this section, we study how to efficiently use the shrinking to speed up the joint separation algorithm by reducing the adverse effects for CCs.

      Given a solution $(y^{*},x^{*})$ and a subset $Q$, the subset $Q$ could generate at the same time a violated SEC and a violated CC for $(y^{*}, x^{*})$. Since the CCs do not depend on the value of the vertices, while the SECs do, the CCs tend to be more violated and more stable, i.e., remain active in subsequent updates of the LP$_0$, than the SECs. Therefore, we treat the CCs with a higher priority.

      Although SECs are part of the OP model, in order to control the size of the working LP$_0$, they are included only when required. This strategy is reasonable since there exist polynomial exact separation algorithms for SECs.
      In contrast, the separation problem for CCs is not known to be polynomial, and it can be modeled as follows:

      \begin{subequations}
        \begin{align}
          \min & \quad 2\sum_{v\in V^{*}} y^{*}_v z_v - 2\sum_{v\in V^{*}} x^{*}_{(v,u)} z_v z_u \\
          s.t: & \quad \sum_{v \in S} s_v z_v\leq LB \\
               & \quad z_1=1 \\
               & \quad z_v \in \{0,1\} \qquad  \forall v\in V
        \end{align}\label{qkp}
      \end{subequations}

      \noindent where $z=(z_v)$ are binary variables whose values are $z_v=1$ if the node $v$ is selected and 0 otherwise. The problem~\eqref{qkp} is a Quadratic Knapsack Problem (QKP) with a fixed variable. Consequently, there exists a violated CC for $(y^{*},x^{*})$ if and only if the optimal solution of Problem~\eqref{qkp} has a value lower than 2.
      Taking into consideration that repeatedly solving QKPs during the B\&C is not viable, the CCs are not separated exactly, but in a heuristic manner take advantage of the SEC separation algorithm. The well-known approaches for the separation of SECs in the TSP, the connected component heuristic and Hong's approach can be extended to jointly separate the SECs and CCs:

      \textit {Connected components heuristic}.
      The straightforward heuristic to find violated SECs and CCs is to search for the connected components of $G^{*}$ using the depth-first-search algorithm. When a connected component contains the depot vertex $1$ and the sum of the vertices scores in the component is lower than $LB$, we record the associated CC of the component, otherwise, we record the associated SECs. 

      \textit {Extended Hong's approach}.
      There are two main strategies to exactly separate SEC inequalities in cycle problems, which are extensions of Hong's approach and the Padberg-Gr{\"o}tschel approach (also known as the Gomory-Hu tree-based approach) for the TSP, see \cite{kobeaga2020}.
      In both approaches, the separation is carried out by solving a sequence of $|V^{*}|-1$ $(s,t)$-minimum cut problems. On the one hand, in the extended Hong's approach, the vertex with a higher $y^{*}$ value (the depot vertex $1$) is fixed to be the source, $s$, and the sink vertices, $t$, are chosen from the set $V^{*}-\{1\}$. On the other hand, the extended Padberg-Gr{\"o}tschel approach is based on the so-called Gomory-Hu tree (directed and rooted in $1$), which is constructed by solving $|V^{*}|-1$ $(s,t)$-minimum cut problems.

      As mentioned above, and as already proposed in the literature, the SEC separation strategies are leveraged to find violated CCs as well.
      Although the extended Padberg-Gr{\"o}tschel approach obtains a larger number of violated SECs, it is not appropriate to find violated CCs, since the obtained sets do not contain the depot vertex $1$.
      Contrarily, the extended Hong's approach for SECs can be easily adapted to additionally find violated CCs.
      It can be achieved, by solving at each step of the separation algorithm the $(1, v)$-minimum cut (useful to find violated SECs) and $(v, 1)$-minimum cut (useful to find violated CCs) problems.
      For these reasons, we use the extended Hong's approach as the base strategy for the joint separation algorithm.

      The running time of these SEC separation algorithms can be improved using the shrinking techniques for cycle problems, as was seen in~\cite{kobeaga2020}. In this publication, three general shrinking rules (C1, C2, and C3) and three SEC specific shrinking rules (S1, S2, and S3) for cycle problems were presented. However, although the shrinking is a key strategy for efficiently separating the SECs, it might be unfavorable for the separation of CCs. The point is that when the vertices are contracted and grouped, the chance to obtain the subset of vertices with a score sum lower than $LB$ decreases, consequently, some violated CCs might vanish. Note that, the mentioned shrinking techniques are safe for valid inequalities of the cycle polytope and CCs are not. Therefore, since CCs are important cuts for OP, shrinking might have a negative impact on the performance of the overall B\&C algorithm for the OP. One contribution in this paper is to propose strategies to minimize the possible disadvantages of the shrinking (which is important to speed up the separation) in the joint separation algorithm for SECs and CCs.

        Following this, not all the shrinking strategies for cycle problems described in~\cite{kobeaga2020} are adequate for the OP problem. Particularly, we exclude the S2 shrinking rule (which leads to excessively aggressive shrinking strategies and hence to vanish violated CCs in some cases) and only consider the shrinking strategies C1C2 and S1 in the preprocess of the joint separation algorithm. Once entered in the separation algorithm, the shrinking rule S3, which contracts the sink and target of the solved minimum cut,
        contributes positively to separating both families of constraints since it enables a wider family of subset candidates to be obtained. Hence, the S3 rule is used in combination with the C1C2 and S1 shrinking strategies in the separation algorithm. After the S3 rule is applied, we search for new shrinkable sets using the selected shrinking strategy.

      Classically, the candidate subsets for SECs and CCs are obtained by the minimum cut algorithm. However, considering the importance of CCs, we intensify the search for extra candidate subsets for CCs, which is made more efficient by taking advantage of the vertex clustering obtained by the shrinking. We propose new strategies based on the following lemma:

    \begin{lemma}\label{thm:setdiff}
      Let $(y,x)$ be a vector that satisfies the degree constraints. If $U$ and $W$ are subsets of $V$ such that $W \subset U$, the following inequality is satisfied:
      \begin{equation}
        x(\delta(U-W)) \leq x(\delta(U)) + x(\delta(W)) \label{ineq:cliqmincut}
      \end{equation}
    \end{lemma}
    \begin{proof}
      When $(y, x)$ satisfies the degree constraints, the identity $x(\delta(T)) = 2y(T) - 2x(E(T))$ is valid for every $T \subset V$. Replacing the respective expressions in the inequality~\eqref{ineq:cliqmincut} we obtain:
      \begin{subequations}
        \begin{align}
          2y(U-W) - 2x(E(U-W)) & \leq  2y(U) - 2x(E(U)) + 2y(W) - 2x(E(W)) \notag \\
          \intertext{Considering the hypothesis $W \subset U$, we have $y(U-W)=y(U)-y(W)$.}
          x(E(U)) - x(E(U-W)) & \leq 2y(W) - x(E(W)) \notag \\
          \intertext{Also, if $Q \subset S$, the equality $E(S-Q)=E(S)-E(Q) - \delta(Q)\cap E(S)$ holds.}
          x(E(U)) - x(E(U)) + x(E(W)) + x(\delta(W) \cap E(U)) & \leq 2y(W) - x(E(W)) \notag \\
          x(\delta(W) \cap E(U)) & \leq 2y(W) - 2x(E(W)) \notag \\
          x(\delta(W) \cap E(U)) & \leq x(\delta(W)) \notag
        \end{align}
      \end{subequations}
        This last inequality is satisfied due to $\delta(Q) \cap E(S) \subset \delta(Q)$, which proves the lemma.
    \end{proof}

      We use the following notation for shrinking. Let $\bar{G}=(\bar{V}, \bar{E})$ be the graph and $(\bar{x}, \bar{y})$ the vector obtained by applying a shrinking strategy to $G^{*}$ and $(y^{*}, x^{*})$, respectively, and $\pi: \mathcal{P}(\bar{V}) \rightarrow \mathcal{P}(V)$ the unshrinking function. Let $\bar{S}$ be the subset obtained by the $(\bar{v},\bar{1})$-minimum cut (where $\bar{1}$ is the contracted vertex that contains the depot vertex $1$), so $1 \in \pi(\bar{S})$, and suppose that $x(\delta(\bar{S})) < 2$. Note that, $x(\delta(S)) = \bar{x}(\delta(\bar{S}))$, where $S=\pi(\bar{S})$. If $\sum_{v \in S} us_v \leq LB$, the subset $S$ defines a violated CC. Otherwise, after each $(\bar{v},\bar{1})$-minimum cut problem is solved, and in the case that $x(\delta(\bar{S})) < 2$, we test the following strategies to find candidate subsets for CCs:
      \begin{enumerate}[i)]
        \item First, when $|\pi(\bar{1})| > 2$, we check if $y_{\bar{1}} <1$ and $\sum_{v \in \pi(\bar{1})} s_v \leq LB$. If this is the case, the subset $Q=\pi(\bar{1})$ defines a violated CC.
        \item Then, we check if there exists $\bar{v} \in \bar{V} - \bar{1}$, such that $ x(\delta(\bar{S})) + 2 y_{\bar{v}} <2$ and $\sum_{v \in \pi(\bar{S} - \bar{v})} s_v \leq LB$. If both inequalities are satisfied for $\bar{v}$, the subset $ \pi(\bar{S}-\bar{v})$ defines a violated CC.
        \item Finally, we sort the vertices in $\bar{S}-\bar{1}$ in non-decreasing order of $\bar{y}$, and check greedily for the greatest subset $S^{'}=\{\bar{v_1}, \ldots, \bar{v_k}\}$ of $\bar{S}$ such that $\bar{x}(\delta(\bar{S})) + 2\sum_{v \in S^{'}} y_{\bar{v}} < 2$. If $\sum_{\bar{v} \in \pi(\bar{S}-S^{'})} s_v \leq LB$, the subset $\pi(\bar{S}-S^{'})$ defines a violated CC.
      \end{enumerate}

  \end{subsubsection}

  \begin{subsubsection}{Comb Inequalities (blossoms)}\label{sec:blossom}

    For the B\&C presented in this work, we only use the blossom subfamily of comb inequalities.
    In this section, we present two heuristics to search for violated blossom inequalities in cycle problems, and in particular, for the OP. The heuristics are extensions of the~\cite{Padberg1980} and~\cite{Grotschel1991} separation algorithms, developed in the context of the TSP.

    The key point of the heuristics for blossom inequalities is to identify a subset of candidate handles to restrict the search of violated blossoms.
    In the literature of OP, a heuristic to find handle candidates is detailed in~\cite{fischetti5}. In this heuristic, the search is guided by the greedy algorithm of Kruskal for the Minimum Spanning Tree. At each iteration of the Kruskal algorithm, a new edge is inserted into the tree, and the connected component containing the edge is chosen as a candidate handle.
    In this work, we consider two alternative approaches to finding candidate handles: the Extended Padberg-Hong heuristic and the Extended Gr{\"o}tschel-Holland heuristic.

    \textit {Extended Padberg-Hong heuristic (EPH)}.
    \cite{Padberg1980} proposed a blossom separation heuristic for the TSP, which is known as the odd-component heuristic. In this heuristic for the TSP, the violated blossoms are found by restricting the set of candidate handles to the connected components of the fractional graph $G^{*}_{1} = (V^{*}_{1}, E^{*}_{1})$, where $E^{*}_{1} = \{ e \in E^{*}: 0 < x^{*}_e < 1\}$ and $V^{*}_{1}=V(E^{*}_{1})$. 

    We generalize this heuristic for the general cycle problems by applying the Padberg-Hong algorithm by levels. A level, $\lambda$, is defined by each different value of the set ${\{y^*_v\}}_v$. We call $L$ the set of distinct levels. Note that, the number of levels, $|L|$, is bounded by $|V|$. Associated with a level we have the level graph $G^{*}_{\lambda} = (V^{*}_{\lambda}, E^{*}_{\lambda})$, where $E^{*}_{\lambda} = \{ e \in E^{*}: 0 < x^{*}_e < \lambda\}$ and $V^{*}_{\lambda}=V(E^{*}_{\lambda})$.

    A faster heuristic to find the handle candidates can be designed by omitting some connected components of $G^{*}_{\lambda}$.  At every level, $\lambda$, we discard the connected components, $C_i^{\lambda}$, such that $y_v \neq \lambda$ for all $v \in C_i^{\lambda}$. Now, we identify the connected component of vertices with $y_v=\lambda$. So, in total, we search for $|V^{*}|$ different connected components of, in the worst case, $G_1^{*}$.

    Once we have identified an initial list of candidate handles, the next step is to find the associated teeth for these handles. Let $H$ be a candidate handle, and define the set of teeth as $\mathcal{T}_H= \{ e \in \delta(H): x^{*}_e \geq \lambda\}$. Recall that the teeth of blossoms are edges.
    Not all the teeth families obtained using this strategy satisfy the comb (blossom) definition. If two teeth overlap in $v \notin H$, then these two teeth are removed from the family of teeth $\mathcal{T}_H$ and the handle is updated as $H=H \cup \{v\}$. If, eventually, the list of teeth $\mathcal{T}_H$ consists of an odd number of at least three disjoint teeth, $\langle H, \mathcal{T}_H, L, R \rangle$ forms a blossom inequality where $L_i=T^j_i \cap H$ and  $R_i=T^j_i - H$. If there is just one tooth i.e., $\mathcal{T}_H=\{T\}$, we test if $H$ defines a violated CC. In the case that it does not, then $H$ alone defines a violated SEC.

    \textit {Extended Gr\"{o}tschel-Holland heuristic (EGH)}.
    Another fast heuristic for the TSP was proposed in~\cite{Grotschel1991} whose aim was to minimize the influence of small perturbations of $x^{*}$ in the separation algorithm. We have adapted this heuristic for the OP using the strategy of levels mentioned above. In this approach, the handles are considered as the vertex sets of the connected components of the graph $G^{*}_{\lambda, \epsilon}= (V^{*}, E^{*}_{\lambda, \epsilon})$ where
    \begin{equation}
      E^{*}_{\lambda, \epsilon} = \{ e \in E^{*}_{\lambda}: \epsilon \leq x^{*}_e \leq (1- \epsilon) \lambda\} \notag
    \end{equation}
    for a small $\epsilon$, $0<\epsilon<1$. Let $H$ denote the vertex set of such a component, a candidate handle, and let $e_1, \ldots, e_t$ be the edges in the set
    \begin{equation}
      \mathcal{T}_H=\{e \in \delta(H) \subset E^{*}: x^{*}_e >  (1 - \epsilon) \lambda \} \notag
    \end{equation}
    in the non-increasing order of $x^{*}_e$. If $t$ is even, then append to $\mathcal{T}_H$ the edge with the highest $x^{*}_e$ in
    \begin{equation}
      \{e \in \delta(H) \subset E^{*}: x^{*}_e < \epsilon \} \notag
    \end{equation}
    If the edges intersect, the strategy outlined above is followed to obtain a handle $H$ and a teeth family $\mathcal{T}_H$ that satisfies the blossom definition.

\begin{figure}
  \begin{center}
    \includegraphics[width=\columnwidth]{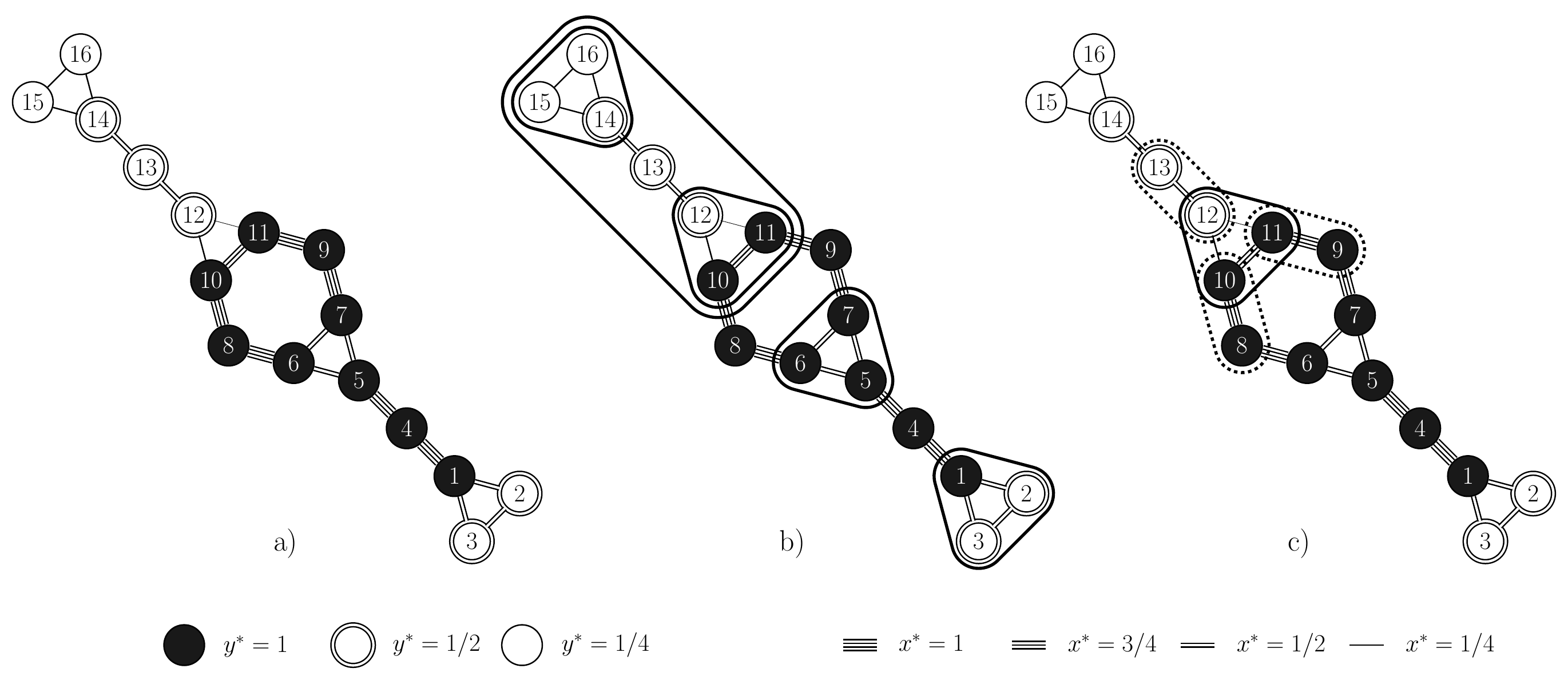}
    \caption{ Illustration for the Extended Padberg-Hong blossom heuristic. Figure a) represents the support graph, with the vertex and edge values detailed in the bottom legend. Figure b) shows all the handle candidates obtained by the heuristic. Figure c) a violated blossom found by the heuristic involving vertices with different $y$ values.}\label{fig:comb}
  \end{center}
\end{figure}

    In Figure~\ref{fig:comb} we illustrate the EPH blossom heuristic for cycle problems. In Figure \ref{fig:comb}.a) the given support graph is presented, where there are three distinct levels, $L=\{1, 1/2, 1/4\}$. In Figure \ref{fig:comb}.b) the candidate handles are presented. Three candidate handles are obtained in level 1: $\{1,2,3\}$, $\{5,6,7\}$ and $\{10, 11, 12, 13, 14, 15, 16 \}$. Two candidate handles are obtained in level $1/2$: $\{10, 11, 12\}$ and $\{14, 15, 16\}$. There are no candidate handles obtained in level $1/4$. Next, we check for violated cuts. The star-set of $\{10, 11, 12, 13, 14, 15, 16 \}$ is formed by two non-overlapping edges, so it is excluded.
    The candidates $\{5, 6, 7\}$ and $\{10, 11, 12\}$ define violated blossoms, e.g., $\langle \{10, 11, 12\}, \{\{8, 10\},\{9,11\},\{12,13\}\}, L, R \rangle$ where $L=\{10, 11, 12\}$ and $R=\{8, 9, 13\}$ shown in Figure \ref{fig:comb}.c).
    The candidates $\{1, 2, 3\}$ and $\{14, 15, 16\}$ define violated SECs, e.g. $\langle \{1, 2, 3\}, 1, 4 \rangle$ and  $\langle \{14, 15, 16\}, 14, 1 \rangle$, but first for $\{1, 2, 3\}$ it should be checked whether it defines a violated CC.
  \end{subsubsection}

\end{subsection}

\begin{subsection}{Column Generation}\label{sec:pricing}

  During the B\&C algorithm, only a subset of edges is included in the working LP$_0$\@. At certain points of the algorithm, we need to price the excluded edge variables, and add to the LP$_0$: 1) to guarantee that the working relaxation is an upper bound of the problem or branched subproblem and 2) to recover, whenever it is possible, a feasible LP$_0$ after feasibility breaking cuts have been added to the LP$_0$. Taking into account that usually only a small subset of variables is included in the LP$_0$, and that the excluded variables could participate in multiple cuts of the LP$_0$, the pricing phase could constitute a bottleneck of the B\&C algorithm. In this section, we develop a technique, inspired by that used in~\cite{concorde}, which enables us to avoid repetitive calculations and to skip the exact calculation of the reduced cost of some variables.

  Let us call $\mathcal{L}^V$ the family of SECs \eqref{op:sec}, CC \eqref{valid:cc}, and comb \eqref{valid:comb} cuts. In these cuts, the edge variables with non-negative coefficients can be represented as the sum star-set of subsets of vertices. Complementarily, let us call $\mathcal{L}^E$ the family of Logical \eqref{op:log}, Edge Cover \eqref{valid:ecover}, Cycle Cover \eqref{valid:ccover} and Path \eqref{valid:path} cuts. Note that the Vertex Cover \eqref{valid:vcover} inequalities do not contribute to the reduced cost of the edge variables.
  So, in the OP, the reduced cost of an edge variable, $e=[v,w]$, can be calculated by:
  \begin{equation}
    rc_e =  - d_e \pi_{d_0} - \pi_v -\pi_w + rc^V_e + rc^E_e
  \end{equation}
  \noindent where $\pi_{d_0}$ is the dual variable of the maximum length constraint \eqref{op:d0}, $\pi_v$ and $\pi_w$ are the dual variables of the degree constraints \eqref{op:deg} of $v$ and $w$ respectively, and $rc^V_e$ and $rc^E_e$ are the contributions made by the cuts in $\mathcal{L}^V$ and $\mathcal{L}^E$, respectively. We will see that the $rc^E_e$ values can be obtained in linear time in terms of $|V|$ and $|\mathcal{L}^E|$, and we will reproduce the pricing strategy used in~\cite{concorde} to calculate the $rc^V_e$ values.

  It can be seen that the cost of the calculation of all the $rc^E_e$ is $O(|\mathcal{L}^E| |V|)$. To that aim, it is sufficient to check that the number of edges with a non-negative coefficient in each cut of $\mathcal{L}^E$ is bounded by $|V|$. In the case of Logical, Cycle Cover, and Path inequalities, it is derived from the definition of the valid inequality. For Edge Cover inequalities, this bound is obtained in Lemma~\ref{lema:cover}.

  \begin{lemma}\label{lema:cover}
    Let $T\subset E$ denote a subset defining a violated cover inequality. If the degree equations~\eqref{op:deg} are satisfied by $(y, x) \in \mathbb{R}^{V \times E}$ then $|T| \leq |V|$.
  \end{lemma}

  \begin{proof}
    When the degree constraints are satisfied by $(y, x)$, as a consequence of the well-known equality $x(\delta(S)) = 2 y(S) - 2 x(E(S))$, the inequality $x(E(V(T))) \leq y(V(T))$ is always satisfied.
    Suppose that $T$ violates the cover inequality~\eqref{valid:ecover} then
    \begin{equation}
      |T|-1 < x(T) \leq x(E(V(T))) \leq y(V(T)) \leq |V|
    \end{equation}
  \end{proof}

  Calculating all the $rc^V_e$ values has a $O(|\mathcal{L}^V||V|^2)$ complexity when the cuts are stored externally as edge variable coefficient arrays.
  The strategy used in~\cite{concorde} speeds up the pricing by obtaining a fast lower bound of the reduced cost $rc^V_e$ (TSP is a minimization problem) and excluding for exact pricing the edges that have a negative lower bound.
  In order to use this strategy for the OP, first, the edge variables of the cuts in $\mathcal{L}^V$ must be represented and stored as a family of subsets of vertices, as we have done in Section~\ref{sec:valid}. Let $ \mathcal{S} = \mathcal{F}_1  \cup \ldots \cup \mathcal{F}_r$ be the family of all the subsets involved in the cuts of $\mathcal{L}^V$ where $\mathcal{F}_i=\{H_i\} \cup \mathcal{T}_i$.
  For combs, $H_i$ and $\mathcal{T}_i$ represent the handle and teeth set, respectively. For SECs and CCs we can assume that $\mathcal{T}_i=\emptyset$ and $H_i=\emptyset$, respectively.

  Based on the representation of the cuts in $\mathcal{L}^V$ by means of subsets of vertices, the cuts are stored in an efficient data structure by pointing to the subsets involved in the cut. This way each subset is saved once at most for all the cuts. Moreover, it allows us to speed up the evaluation of $rc^V_e$ values as explained below.

  Since the OP is a maximization problem, during the pricing, we need to identify the edge with positive reduced cost. We aim to define upper bounds, $\hat{rc}_e$, of the reduced costs $rc_e$, to exclude for exactly pricing the edges that have a non-positive upper bound $\hat{rc}^V_e$.

  For each subset, $S \in \mathcal{S}$, let us call $\pi_S$ the dual of the subset $S$  defined as:

  \begin{equation}
    \pi_S = \sum_{j=1}^r\chi_j(S) \pi_{j}
  \end{equation}

  \noindent where $\chi_j(S)=1$ if $S \in \mathcal{F}_j$ and 0 otherwise,  and $\pi_j$ is the dual variable associated with the cut $j$. Then, the contribution of the cuts in $\mathcal{L}^V$ in the reduced cost of an edge $e$ can be written as:
  \begin{equation}
    rc^V_e =  \sum_{\substack{S \in \mathcal{F}\\ V(e)\cap S \neq \emptyset \\ V(e) - S \neq \emptyset}}{\pi_S }
  \end{equation}
  where $\pi_S$ is the dual of a subset $S$. Since, for the edge $e=[v, w]$, each $S$ must contain either $v$ or $w$, an upper bound, $\hat{rc}^V_e$, of $rc^V_e$ can be obtained by:
  \begin{equation}
    \hat{rc}^V_e = \sum_{\substack{S \in \mathcal{F} \\ v \in S }}{\pi_S} + \sum_{\substack{S \in \mathcal{F} \\ w \in S}}{\pi_S } \notag
  \end{equation}

  \noindent which satisfies $rc^V_e \leq \hat{rc}^V_e$. Therefore, we have the desired upper bound:
  \begin{equation}
    \hat{rc}_e = - d_e \pi_{n+1} - \pi_v -\pi_w + rc^E_e + \hat{rc}^V_e
  \end{equation}

  Note that, each edge appears at most twice in a comb inequality, so the calculation of all the $\hat{rc}^V_e$ has a $O(M |\mathcal{L}^V| |V|)$ time complexity where $M$ is the maximum number of subsets involved in a cut. Therefore, the calculation of all the $\hat{rc}_e$ has a $O(M |\mathcal{L}^V| |V|)$ time complexity. In our B\&C, the value of $M$ is related to the number of teeth in the combs. To ensure a true linear time complexity procedure, one could limit the number of teeth in the combs. However, in practice, the number of teeth tends to be small and it can be assumed that $M << |V|$.

  We can exclude for exactly pricing the edges that $\hat{rc_e} \leq 0$. For those edges that $\hat{rc_e} > 0$, the exact reduced cost, $rc_e$, can be calculated by using the upper bound value:
  \begin{equation}
    rc_e = \hat{rc}_e - 2 \sum_{\substack{S \in \mathcal{F} \\ V(e) \subset S}}{\pi_S }
  \end{equation}

  The pricing loop is done in batches. In the first step, a fixed number of $\hat{rc}_e$ are calculated, the first batch of variables and those with positive values are preselected. In the next step, for those preselected variables, we calculate the exact reduced cost, $rc_e$, and add to the LP$_0$ the edges whose value is positive. Then, the LP$_0$ is updated. Next, we select the second batch of variables and we repeat the procedure. When the pricing aims to obtain the upper bound of the branched subproblem, we exit the pricing loop when a whole round of evaluation is performed without introducing a variable to the LP$_0$. When the pricing aims to recover a feasible LP$_0$, we exit the pricing loop once a feasible LP$_0$ is obtained without the need to price all the excluded variables.
\end{subsection}

\begin{subsection}{Separation Loop}\label{sec:sep}

  The separation loop to find the violated cuts is accomplished in three subloops. In the inner loop, we consider the separation of logical constraints~\eqref{op:log} and the connected components heuristic for SECs and CCs. In the middle loop, we consider the separations of cuts which are related to the cycle essence of the OP, i.e., SECs, CCs, blossoms, and Cycle Cover cuts. In the outer loop, we consider the rest of the cuts, i.e., the Edge Cover, Vertex Cover, and the Path inequalities. The separation loop is illustrated in Figure~\ref{fig:loop}.

\begin{figure}[htb!]
  \centering
  \includegraphics[width=0.9\columnwidth]{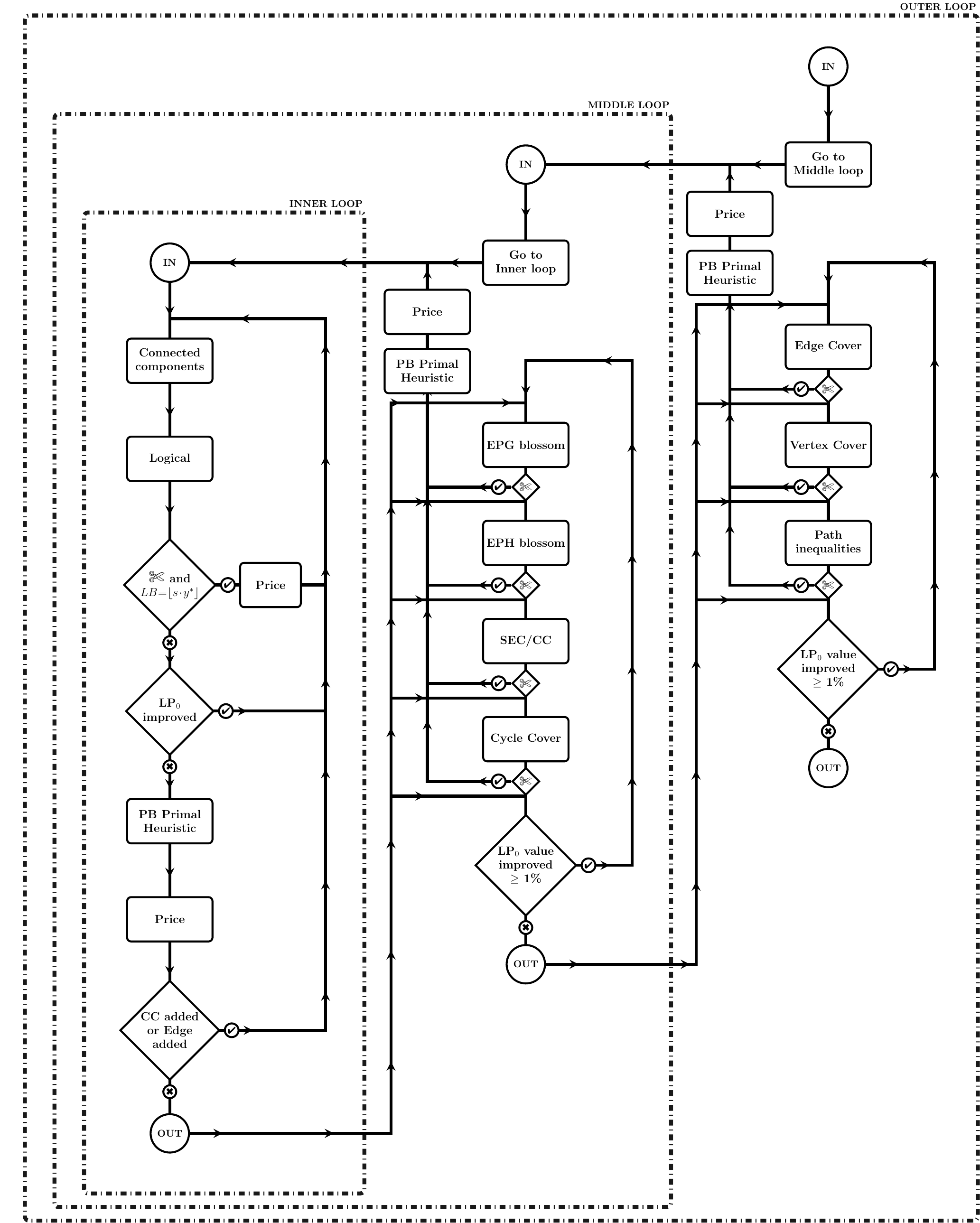}
  \caption{Illustration of the separation loop. The symbol \ding{36} represents that some cuts have been added to the LP$_0$.}\label{fig:loop}
\end{figure}

  At each subloop, the separation of the considered cuts is performed sequentially, instead of restarting from the beginning of the list.
  This is, we always carry out the next separation in the subloop list, regardless of whether or not we are coming from an interior subloop.
  This way, we give the same chance to all separations in a subloop and decrease the probability of bounding in the same separation algorithm in consecutive iterations of the subloop.

  The separation algorithms of the inner loop are fast since both have a $O(|E^{*}|)$ time complexity. First, we carry out the connected components heuristic and then the separation of logical constraints. In the inner loop, intending to keep it as a fast loop, we price the edge variables only when the floor part of the objective value is equal to the lower bound of the OP, i.e., if $\lfloor s \cdot y^{*} \rfloor = LB$.
  When both separations fail and no new edges have been added, we find a feasible solution using the PB primal heuristic (see Section~\ref{sec:heur}) and update the LB if needed. We add the associated CC of the heuristic solution if it is violated and then we price the variables. When a new CC cut or a priced edge has been added to the LP$_0$, the inner loop is repeated. Otherwise, we return to the middle loop.

  The middle and outer loops only differ in the considered constraint families. In the middle loop, we consider the separation algorithms in the following order: the extended Padberg-Hong algorithm for blossom, the extended Gr{\"o}tschel-Holland algorithm for blossom, the joint SEC/CC separation algorithm, and Cycle Cover separation algorithm. In the outer loop, we consider the Edge Cover algorithm, the Vertex Cover algorithm, the Path algorithm. When we enter in any of the loops, the first step is to execute the lower level subloops. Then, we start with the first algorithm on the list. If no violated cuts are found we move on to the next algorithm. If violated cuts are found, we first add the cuts and optimize the LP$_0$. Then, we search for a feasible solution using a primal heuristic and update the LB if needed. We add the associated CC of the heuristic solution in case it is violated and then we price the variables. At this point, we move to the lower level loop and continue with the next separation in the list.

  In the separation loop, after adding the violated cuts found in a separation algorithm, we check if any edge variable or constraint can be removed from the LP$_0$\@. We remove an edge variable from the LP$_0$ if, during a number of consecutive evaluations, its associated value, $x^{*}_e$, has been zero. We remove a constraint from the LP$_0$ if during a number of consecutive evaluations its slack has been higher than zero.

\end{subsection}

\begin{subsection}{Primal Heuristics and Lower Bounds}\label{sec:heur}

  We use two primal heuristics to obtain feasible solutions from a fractional solution $(y^{*}, x^{*})$. In the first heuristic, we obtain a single solution, by using the $x^{*}$ values related to edges, inspired by the heuristic proposed in~\cite{fischetti5}. In the second heuristic, first, we build a population of cycles and then evolve it using the EA4OP metaheuristic, see \cite{kobeaga2018}. The cycles in the population are constructed by selecting first the subset of vertices in each cycle using the $y^{*}$ values.

  \textit{Path Building primal heuristic (PB)}. The PB heuristic was presented in~\cite{fischetti5}. First, the edges $e \in E^{*}$ are sorted in decreasing order of $x^{*}_e$, and the ties are randomly broken. The procedure starts with an empty path $T=\emptyset$. At each step we select an edge $e \in E^{*}$ whose $x^{*}_e$ has the largest value from the set of edges which have not been considered yet. If the inclusion of $e$ in $T$ does not lead to a vertex with a degree larger than $2$, then $T=T\cup \{e\}$ otherwise we exclude $e$ and repeat the process. The path building heuristic finishes when the inclusion of $e$ leads to $T$ being a cycle or when there are no edges left to check. If the depot vertex is not in one of the paths in $T$, it is included as a single point path. If $T$ consists of multiple paths, we extend it to a cycle by randomly connecting the extreme vertices (in the original paper the paths were joined using the nearest neighbor heuristic). Since this primal heuristic is fast, it is used in the separation loop.

  \textit{Vertex Picking primal heuristic (VP) with the EA4OP metaheuristic}. In the VP heuristic, we first select a collection of vertices in $V^{*}$ and then build a random cycle through the selected vertices. Each vertex $v$ is selected according to a Bernouilli distribution with parameter $y^{*}_v$. By applying multiple times the VP strategy to obtain feasible solutions from $(y^{*}, x^{*})$, we build a small population. Then, as explained below, we ensure that the solutions in the population are feasible and improve when it is possible. Once we have a population with feasible solutions, it is evolved using the EA4OP metaheuristic proposed in~\cite{kobeaga2018}. The EA4OP with VP heuristic is used to find feasible solutions after an edge is branched, as shown in Figure~\ref{fig:scheme}.

  For solutions obtained by PB and VP heuristics, we improve the route lengths using the Lin-Kernighan heuristic for the TSP, and then first check if it satisfies the constraint~\eqref{op:d0}.  If it does not, we apply the drop operator which consists in deleting vertices from the solution until the cycle satisfies the length constraint. Then we try to improve the solution by the k-d tree based vertex inclusion procedure as explained in~\cite{kobeaga2018}.

\end{subsection}

\begin{subsection}{Branching and Upper Bounds}\label{sec:branching}

  The branching is carried out in a classical way following a depth-first-search, where the edges are branched first to 1 and then to 0. In order to select the edge variable to branch, we use the classical branching strategy: the edge $e$, with the fractional value closest to $0.5$ is selected, i.e., the edge that minimizes $|x^{*}_e - 0.5|$.

  The global upper bound and branch node upper bound are calculated just before pruning a branch. The branch node upper bound, $UB^N$, is used to verify the pruning, i.e, that $LB \geq \lfloor UB^N \rfloor$. The global upper bound is calculated with two aims: firstly, to use it in Vertex Cover separation, and secondly, to compute the optimality gap when the algorithm finishes due to time limitations.

  The global upper bound, $UB^G$ of OP, is obtained using the dual solution $\pi^{*}$ of the solution $(y^{*}, x^{*})$ of the LP$_0$:

  \begin{equation}
    UB^G = \sum_{i=1}^c \pi^{*}_i b_i+ rc^{*}_1 + \sum_{\substack{v \in V - \{1\}\\ rc^{*}_v > 0}} rc^{*}_v + \sum_{\substack{e \in E\\ rc^{*}_e > 0}} rc^{*}_e
  \end{equation}
  \noindent where the reduced costs $rc_v^{*}$ and $rc_e^{*}$ are calculated using the dual variables $\pi^{*}$ and $c$ is the number of constraints.

  The upper bound of a branch node, $UB^N$, can be calculated by subtracting the contributions of the branched edges to $UB^G$. Let $B_0, B_1 \subset E$ be the subset of edges branched to 0 and 1, respectively. Then, we obtain $UB^N$ by:
  \begin{equation}
    UB^N = \sum_{i=1}^c \pi^{*}_i b_i+ rc^{*}_1 + \sum_{\substack{v \in V - \{1\}\\ rc^{*}_v > 0}} rc^{*}_v + \sum_{\substack{e \in E\\ rc^{*}_e > 0}} rc^{*}_e - \sum_{\substack{e \in B_0\\ rc^{*}_e > 0}} rc^{*}_e + \sum_{\substack{e \in B_1\\ rc^{*}_e < 0}} rc^{*}_e
  \end{equation}

\end{subsection}

\section{Computational Experiments}\label{sec:exp}

In this section, we present the results of the computational experiments. Firstly, we evaluate the new designed components for the revisited B\&C algorithm (RB\&C); and secondly, we compare the performance of RB\&C with state-of-the-art B\&C and heuristic algorithms. The software used for the experiments is publicly available on \url{github.com/gkobeaga/op-solver}.

The experiments are carried out using well-known instances in the literature. These instances, which are based on the TSPLIB library, were first proposed in~\cite{fischetti5} and then extended to larger problems in~\cite{kobeaga2018}. The instances are split into two groups: medium-sized instances (up to 400 nodes) and large-sized instances (up to 7397 nodes). In total, we consider 258 benchmark instances. They are also classified into three generations (Gen1, Gen2 and Gen3) according to the definition of scores, see~\cite{fischetti5}. For all of these three generations, the distance limitation is set as half of the TSP solution value. These instances are publicly available at \url{github.com/bcamath-ds/OPLib}.

In order to measure the performance of the algorithms, we compare the quality of the returned best solutions (LB) and the mean running time (in seconds) of the algorithms. In addition, in the case of the B\&C algorithms, we also compare the obtained upper bounds (UB). All the experiments for the compared algorithms have been carried out using a 5-hour time limit.

In Table~\ref{tab:common-param}, we detail the values of the common parameters for all the simulations of the RB\&C algorithm. They were chosen inspired by the parameters used in~\cite{concorde} and our preliminary experiments for the OP\@.

\begin{table}[htb!]
  \begin{center}
    \caption{Common parameters.}\label{tab:common-param}
    \centering
    \begin{tabular}{lrl}
      \toprule
      Parameter & Value & Description \\
      \midrule
      ZERO & $10^{-7}$ & Sensibility of fractional numbers\\
      \cmidrule(lr){1-1} \cmidrule(lr){2-2} \cmidrule(lr){3-3}
      ADD\_CUT\_BATCH & 250 & Maximum number of cuts added to the LP$_0$ at once\\
      ADD\_MIN\_VIOL & $10^{-6}$ & Minimum violation of a cut to include it in the LP$_0$\\
      SUBLOOP\_IMPR & $1\%$ & Minimum improvement to repeat the subloops \\
      ADD\_SEC\_PER\_SET & 50 & Amount of SECs considered for each subset\\
      ADD\_PATH\_MAX & 500 & Maximum cuts for Path inequalities separation\\
      ADD\_EGH\_EPSILON & 0.3 & Epsilon value for the EGH blossom heuristic\\
      \cmidrule(lr){1-1} \cmidrule(lr){2-2} \cmidrule(lr){3-3}
      PRICE\_MAX\_ADD & 200 & Maximum number of variables added to the LP$_0$\\
      PRICE\_RC\_THRESH & $10^{-5}$ & Minimum penalty of a variable to add to the LP$_0$\\
      DEL\_DUST\_VAR & $10^{-3}$ & Minimum $y$ value to consider an edge as active\\
      DEL\_DUST\_CUT & $10^{-3}$ &  Maximum slack value to consider a cut as active\\
      DEL\_MAX\_AGE\_CUT & 5 & Consecutive inactivity to delete a cut from the LP$_0$\\
      DEL\_MAX\_AGE\_VAR & 100 &  Consecutive inactivity to delete an edge from the LP$_0$\\
      \cmidrule(lr){1-1} \cmidrule(lr){2-2} \cmidrule(lr){3-3}
      XHEUR\_GREEDY\_XMIN & 0.3 & Use arcs larger than this value in PB primal heuristic\\
      XHEUR\_EA4OP\_POP\_SIZE & 10 & Population size for EA4OP\\
      XHEUR\_EA4OP\_D2D & 5 & Iterations before checking feasibility in EA4OP\\
      XHEUR\_EA4OP\_NPAR & 3 & Number of parents preselected in EA4OP\\
      \bottomrule
    \end{tabular}
  \end{center}
\end{table}

\subsection{Evaluation of Components}\label{sec:eval}
In this section, we evaluate the designed components for the RB\&C algorithm in Section~\ref{sec:bac}.
We have carried out experiments with several alternative configurations of the components.
To that aim, a subset of 15 OP instances were selected: 5 TSP instances (pr76, att532, vm1084, rl1323 and vm1748, inspired by the subset selected in~\cite{goldberg2001}) with their respective score generations proposed in~\cite{fischetti5}. Then, for each instance and generation, we have executed the different B\&C configurations 5 times.

In order to evaluate our contributions, we have chosen a reference configuration, REFERENCE, that  incorporates the components proposed in Section~\ref{sec:bac} and compared it with its alternative configurations.
The reference RB\&C algorithm considers the following components:

\begin{itemize}
  \item SEC/CC separation algorithm (Section \ref{sec:seccc}):
    \begin{enumerate}[i)]
      \item SRK=S1S3: Uses shrinking rules S1 and S3.
      \item CC STRATS: Uses strategies to find extra violated CCs.
    \end{enumerate}
  \item Blossom separation algorithms (Section \ref{sec:blossom}):
    \begin{enumerate}[i)]
      \item EPH BLOSSOM: Uses Extended Padberg-Hong blossom heuristic.
      \item EGH BLOSSOM: Uses Extended Gr{\"o}tschel-Holland blossom heuristic.
    \end{enumerate}
  \item Separation algorithms from the literature:
    \begin{enumerate}[i)]
      \item CYCLE: Uses Cycle Cover inequalities.
      \item EDGE: Uses Edge Cover inequalities.
      \item PATH: Uses Path inequalities.
    \end{enumerate}
  \item Separation Loop strategy:
    \begin{enumerate}[i)]
      \item SEP=THREE SUBLOOPS: Uses the separation loop strategy presented in Section~\ref{sec:sep}.
    \end{enumerate}
  \item Primal heuristics (Section \ref{sec:heur}):
    \begin{enumerate}[i)]
      \item BRANCH XHEUR=VP + EA4OP: Constructs a small population using VP heuristic and evolves it with EA4OP.
      \item SEP XHEUR=PB: Constructs a single solution using PB in the separation loop.
    \end{enumerate}
\end{itemize}

The alternative configurations are obtained by modifying a single component in REFERENCE, while the rest of the components remain untouched. These changes to REFERENCE are made by deleting a component(-), adding a new component(+) or replacing a component (COMP=).
The tested alternative strategies are the following:
\begin{itemize}
  \item SEC/CC separation algorithm:
    \begin{enumerate}[i)]
      \item -SRK: Does not use any shrinking technique. As a consequence, CC STRATS are not used either.
      \item SRK=C1C2S3: The shrinking rule S1 is replaced with the rules C1C2.
      \item -CC STRATS: Does not use strategies to find extra violated CCs.
    \end{enumerate}
  \item Blossom separation algorithms:
    \begin{enumerate}[i)]
      \item -EPH BLOSSOM: Does not use the Extended Padberg-Hong blossom heuristic
      \item -EGH BLOSSOM: Does not use the Extended Gr{\"o}tschel-Holland blossom heuristic
      \item +FST BLOSSOM: Uses the blossom separation heuristic in~\cite{fischetti5}
    \end{enumerate}
  \item Separation algorithms from the literature:
    \begin{enumerate}[i)]
      \item -CYCLE COVER: Does not use Cycle Cover inequalities
      \item -EDGE COVER: Does not use Edge Cover inequalities
      \item +VERTEX COVER: Uses Vertex Cover inequalities
      \item -PATH: Does not use Path inequalities
    \end{enumerate}
  \item Separation Loop strategy:
    \begin{enumerate}[i)]
      \item SEP=TWO SUBLOOPS: The separations algorithms in the outer subloop are appended to the middle subloop.
    \end{enumerate}
  \item Primal heuristic in the branch node:
    \begin{enumerate}[i)]
      \item BRANCH XHEUR=PB: Constructs a single solution using PB heuristic.
      \item BRANCH XHEUR=VP - EA4OP: Constructs a single solution using VP heuristic.
  \end{enumerate}
\end{itemize}

In Table~\ref{tab:valid} we summarize the mean relative difference to the best achieved LB and UB, as well as the mean relative difference to the best performing configuration in terms of running time. The results grouped by instances are presented in Appendix \ref{appendix:valid}.

\begin{table}[htb!]
  \centering
  \scriptsize
  \caption{Results of the alternative configurations for RB\&C.  In bold, the values of the alternatives that are worse than those obtained by the REFERENCE configuration.}\label{tab:valid}
  \begin{tabular}{lrrrrrrrrr}
    \toprule
 & \multicolumn{9}{c}{Gap} \\ \cmidrule(lr){2-10}
 & \multicolumn{3}{c}{Gen1} & \multicolumn{3}{c}{Gen2} & \multicolumn{3}{c}{Gen3} \\ \cmidrule(lr){2-4}\cmidrule(lr){5-7}\cmidrule(lr){8-10}
    Strategy  & LB & UB & Time & LB & UB & Time & LB & UB & \multicolumn{1}{c}{Time} \\
    \midrule
    REFERENCE  & $0.05$ & $0.00$ & $262.06$ & $0.05$ & $0.04$ & $\phantom{0}23.11$ & $0.02$ & $0.01$ & $\phantom{0}44.02$ \\
    \midrule
    - SRK  & $\mathbf{0.13}$ & $0.00$ & $\mathbf{532.37}$ & $\mathbf{0.10}$ & $0.04$ & $\phantom{0}\mathbf{25.86}$ & $0.02$ & $\mathbf{0.02}$ & $\mathbf{134.74}$ \\
    SRK=C1C2S3  & $0.02$ & $0.00$ & $\phantom{0}88.32$ & $\mathbf{0.09}$ & $0.04$ & $\phantom{0}\mathbf{31.72}$ & $0.01$ & $0.01$ & $\phantom{0}\mathbf{79.81}$ \\
    - CC STRATS  & $0.02$ & $0.00$ & $115.91$ & $0.04$ & $0.01$ & $\phantom{0}21.85$ & $0.01$ & $0.01$ & $\mathbf{449.90}$ \\
    \midrule
    - EPH BLOSSOM  & $\mathbf{0.09}$ & $\mathbf{0.15}$ & $208.65$ & $\mathbf{0.12}$ & $\mathbf{0.15}$ & $\phantom{0}\mathbf{33.64}$ & $\mathbf{0.10}$ & $\mathbf{0.22}$ & $\mathbf{199.79}$ \\
    - EGH BLOSSOM  & $0.02$ & $0.00$ & $\mathbf{296.71}$ & $0.04$ & $0.04$ & $\phantom{0}\mathbf{26.18}$ & $\textbf{0.03}$ & $0.01$ & $\phantom{0}\mathbf{91.83}$ \\
    + FST BLOSSOM  & $0.00$ & $0.00$ & $\mathbf{345.32}$ & $0.04$ & $0.00$ & $\phantom{0}\mathbf{26.43}$ & $\mathbf{0.04}$ & $0.00$ & $\phantom{0}\mathbf{66.54}$ \\
    \midrule
    - EDGE COVER  & $\mathbf{0.11}$ & $0.00$ & $137.73$ & $\mathbf{0.13}$ & $0.04$ & $\phantom{0}\mathbf{30.04}$ & $\mathbf{0.05}$ & $0.01$ & $\phantom{0}35.50$ \\
    - CYCLE COVER  & $\mathbf{0.06}$ & $0.00$ & $124.79$ & $0.02$ & $0.04$ & $\phantom{0}\mathbf{25.60}$ & $\mathbf{0.03}$ & $0.01$ & $\phantom{0}\mathbf{48.18}$ \\
    - PATH  & $\mathbf{0.08}$ & $0.00$ & $183.86$ & $\mathbf{0.10}$ & $0.04$ & $\phantom{0}\mathbf{32.00}$ & $\mathbf{0.03}$ & $0.01$ & $\phantom{0}\mathbf{69.01}$ \\
    + VERTEX COVER  & $0.05$ & $0.00$ & $\phantom{0}61.10$ & $0.03$ & $0.04$ & $\phantom{0}22.33$ & $\mathbf{0.03}$ & $0.01$ & $\mathbf{104.82}$ \\
    \midrule
    SEP: TWO SUBLOOPS  & $0.05$ & $0.00$ & $\mathbf{315.34}$ & $\mathbf{0.06}$ & $0.04$ & $\phantom{0}17.05$ & $\mathbf{0.03}$ & $0.01$ & $\mathbf{164.44}$ \\
    \midrule
    BRANCH XHEUR=PB  & $\mathbf{0.08}$ & $0.00$ & $179.14$ & $\mathbf{0.12}$ & $0.01$ & $\phantom{00}2.37$ & $\mathbf{0.04}$ & $0.01$ & $\phantom{0}62.74$ \\
    BRANCH XHEUR=VP - EA4OP  & $0.02$ & $0.00$ & $222.46$ & $\mathbf{0.07}$ & $0.04$ & $\phantom{00}7.17$ & $0.01$ & $0.01$ & $\mathbf{168.63}$ \\
    \bottomrule
  \end{tabular}
\end{table}

The results show that the alternatives decrease the performance of the REFERENCE configuration for the RB\&C algorithm either in terms of solution quality, upper bound value, or running time.
The experiments restate the importance of the shrinking techniques for the SEC/CC separation algorithm, as can be seen in the results for -SRK. It is not only worse not using the shrinking in terms of time, but indeed, the obtained LB values are also worse. In addition, the results suggest that the S1 shrinking technique, which is considered in REFERENCE, might be preferable to the C1C2 technique. Regarding the CC STRATS, the results for Gen3 suggest that not considering the strategies to find extra violated CCs might have a negative impact on the running time of the algorithm.

  Next, looking at the separation algorithms for blossoms, the results show that the EPH heuristic is crucial in the RB\&C, particularly, if we focus on the obtained LB and UB values. From the table, we can also extract that the EGH heuristic improves the running time of the B\&C algorithm. Alternatively, although the FST blossom heuristic might improve the quality of the solutions, it reports worse running times.

  With respect to the rest of the separation algorithms proposed in the literature for the OP, we include in REFERECE all but Vertex Cover inequalities. This way, the RB\&C uses the same families of cuts as in~\cite{fischetti5}, which enables us to evaluate the contributions in this paper in a better way.

  Finally, the experiments show that the VP primal heuristic plays an important role in obtaining better LB values, particularly for large problems, as can be seen in the detailed results in Appendix \ref{appendix:valid}.  However, solving the VP primal heuristic in the branch node is more costly than PB primal heuristic, hence the running time of the RB\&C is worsened in the smallest instances. Similarly, by using the EA4OP to improve the results by VP heuristic, the obtained LB values are improved in large problems at the expense of worsening the running time in the smallest instances.

\subsection{Comparison with state-of-the-art Algorithms}\label{sec:res}

The proposed reference RB\&C has been compared with the state-of-the-art B\&C algorithm in~\cite{fischetti5} (FST) and two state-of-the-art heuristics,~\cite{kobeaga2018} (EA4OP) and~\cite{santini2019} (ALNS). The detailed results can be found in Appendix \ref{appendix:comp}.

 Three notes before moving on to the discussion. First, the FST code reports the running times using one trailing digit while the rest of the algorithms report the times using two trailing digits. In order to make use of the reported times in the literature of the FST, we round the obtained times by the RB\&C to one trailing digit when we compare it with the FST algorithm. Secondly, the FST returns a false optimum for pa561 in Gen1. We assume that this is a consequence of the rounding sensibility and we accept as valid the rest of the reported optima by FST. Thirdly, eight instances (rat99, rat195, tsp225, pa561, rat575, rat783, nrw1379, and fnl4461) of Gen3 have been excluded for the comparison of the RB\&C with the EA4OP and the ALNS, due to an issue in the generation of scores of the instances used by those algorithms. Since the results of the current comparison are clear enough, we have discarded rerunning the experiments with the updated scores.

First, we compare the RB\&C algorithm with the B\&C by~\cite{fischetti5}. The results of the FST algorithm were updated using CPLEX12.5 in~\cite{kobeaga2018}, which is the same version of CPLEX used for the experiments of RB\&C. Moreover, the new experiments are run on the same machine with the same amount of reserved memory (4GB).
In Table~\ref{tab:new_comp} we summarize, by size and generation, the number of instances returning a feasible solution, \#, the obtained optimality certifications, OPT, the number of best-known solution (LB), and upper bound (UB) values.

\begin{table}[!ht]
  \centering
  \caption{Comparison of the  number of instances in which a feasible solution (\#), an optimal (OPT), a best-known solution (LB) or a best upper bound value (UB) were obtained.}\label{tab:new_comp}
  \begin{tabular}{llcccccccc}
    \toprule
& & \multicolumn{2}{c}{\#} & \multicolumn{2}{c}{OPT} & \multicolumn{2}{c}{LB} & \multicolumn{2}{c}{UB} \\ \cmidrule(lr){3-4}\cmidrule(lr){5-6}\cmidrule(lr){7-8}\cmidrule(lr){9-10}
    Size & Gen & FST & RB\&C & FST & RB\&C & FST & RB\&C & FST & \multicolumn{1}{c}{RB\&C} \\
    \midrule
    Medium & Gen1  & $\phantom{0}45$ & $\phantom{0}45$ & $\phantom{0}\mathbf{45}$ & $\phantom{0}44$ & $\phantom{0}45$ & $\phantom{0}45$ & $\phantom{0}\mathbf{45}$ & $\phantom{0}44$ \\
           & Gen2  & $\phantom{0}45$ & $\phantom{0}45$ & $\phantom{0}45$ & $\phantom{0}45$ & $\phantom{0}45$ & $\phantom{0}45$ & $\phantom{0}45$ & $\phantom{0}45$ \\
           & Gen3  & $\phantom{0}45$ & $\phantom{0}45$ & $\phantom{0}45$ & $\phantom{0}45$ & $\phantom{0}45$ & $\phantom{0}45$ & $\phantom{0}45$ & $\phantom{0}45$ \\
           \midrule
    Large & Gen1  & $\phantom{0}21$ & $\phantom{0}\mathbf{41}$ & $\phantom{0}12$ & $\phantom{0}\mathbf{24}$ & $\phantom{0}13$ & $\phantom{0}\mathbf{39}$ & $\phantom{0}13$ & $\phantom{0}\mathbf{40}$ \\
          & Gen2  & $\phantom{0}22$ & $\phantom{0}\mathbf{41}$ & $\phantom{00}9$ & $\phantom{0}\mathbf{10}$ & $\phantom{00}9$ & $\phantom{0}\mathbf{36}$ & $\phantom{0}13$ & $\phantom{0}\mathbf{38}$ \\
          & Gen3  & $\phantom{0}29$ & $\phantom{0}\mathbf{41}$ & $\phantom{00}9$ & $\phantom{0}\mathbf{12}$ & $\phantom{0}13$ & $\phantom{0}\mathbf{35}$ & $\phantom{0}12$ & $\phantom{0}\mathbf{37}$ \\
          \midrule
          & All  & $207$ & $\mathbf{258}$ & $165$ & $\mathbf{180}$ & $170$ & $\mathbf{245}$ & $173$ & $\mathbf{249}$ \\
          \bottomrule
  \end{tabular}
\end{table}

In Table~\ref{tab:new_comp} it can be seen that the RB\&C algorithm is able to obtain the best-known solutions value in all the medium-sized instances.
Moving on to large-sized instances, the superiority of the RB\&C algorithm compared to the FST approach becomes evident. While the FST algorithm fails to output a solution in almost half of the instances (mainly because of running out of memory), the  RB\&C algorithm returns a solution for every instance. Moreover, it obtains the best-known solution in significantly more instances than FST (245 against 170) and UB (249 against 173) values. Even more, it obtains more optimality certifications (180 against 165).

\begin{table}[!ht]
  \caption{ Comparison of the number of obtained optimal solutions (OPT), number of best-known solutions (LB) and number of best upper bounds (UB) in the instances that FST does return a solution.}\label{tab:new_finishes}
  \centering
  \begin{tabular}{lccccccccc}
    \toprule
 &  & \multicolumn{2}{c}{OPT} & \multicolumn{2}{c}{LB} & \multicolumn{2}{c}{UB} & \multicolumn{2}{c}{Time} \\ \cmidrule(lr){3-4}\cmidrule(lr){5-6}\cmidrule(lr){7-8}\cmidrule(lr){9-10}
 & \# & FST & RB\&C & FST & RB\&C & FST & RB\&C & FST & \multicolumn{1}{c}{RB\&C} \\
 \midrule
    Gen1  & $\phantom{0}66$ & $1$ & $\mathbf{4}$ & $0$ & $\phantom{0}\mathbf{6}$ & $2$ & $\phantom{0}\mathbf{8}$ & $15$ & $\phantom{0}\mathbf{40}$ \\
    Gen2  & $\phantom{0}67$ & $\mathbf{1}$ & $0$ & $0$ & $\mathbf{11}$ & $3$ & $\phantom{0}\mathbf{9}$ & $25$ & $\phantom{0}\mathbf{27}$ \\
    Gen3  & $\phantom{0}74$ & $1$ & $\mathbf{3}$ & $1$ & $\mathbf{14}$ & $4$ & $\mathbf{17}$ & $23$ & $\phantom{0}\mathbf{33}$ \\
    \midrule
    All  & $207$ & $3$ & $\mathbf{7}$ & $1$ & $\mathbf{31}$ & $9$ & $\mathbf{34}$ & $63$ & $\mathbf{100}$ \\
    \bottomrule
  \end{tabular}
\end{table}

  In Table~\ref{tab:new_finishes} we compare the quality of the solutions and running times, restricted to those instances in which FST actually returns a solution. We particularly focus on the number of solutions (optimality certifications, best-known solutions and upper bounds) that are new in the literature, i.e., values not obtained by the rest of the algorithms. Thus, for the lower-bound values, we also take into account the results obtained by the EA4OP and ALNS heuristics. Additionally, we show the number of instances in which the considered B\&C algorithms are faster than the competitor. When we restrict the considered instances to the instances where the FST obtains a feasible solution, the RB\&C outperforms the results of the FST. While the FST obtains 1 new best-known solution (not obtained by any other algorithm) and 9 new UB values, the RB\&C obtains 31 LB and 34 UB new values. In the same set of instances, the FST obtains 3 optimality certifications that the RB\&C is not able to obtain, while the RB\&C obtains 7 optimality certifications that the FST is unable to obtain. Moreover, it turns out that the RB\&C is faster than the FST in 100 instances while the FST is faster than the RB\&C in 63 instances.

Next, we compare the RB\&C algorithm against state-of-the-art algorithms in terms of solution quality, running time, and Pareto efficiency. In Table~\ref{tab:pareto_medium} and Table~\ref{tab:pareto_large} the algorithms are compared pairwise and instance-by-instance for medium-sized and large-sized instances respectively. The aim is to measure the number of instances where an algorithm is simultaneously as least as fast as the opponent and obtains a better quality solution.

\begin{table}[!ht]
  \centering
  \caption{Comparison in medium-sized instances against state-of-the-art algorithms in terms of quality, time and Pareto efficiency.}\label{tab:pareto_medium}
  \begin{tabular}{lccccccccc}
    \toprule
 & \multicolumn{3}{c}{Gen1} & \multicolumn{3}{c}{Gen2} & \multicolumn{3}{c}{Gen3} \\ \cmidrule(lr){2-4}\cmidrule(lr){5-7}\cmidrule(lr){8-10}
 & EA4OP & tie & RB\&C & EA4OP & tie & RB\&C & EA4OP & tie & \multicolumn{1}{c}{RB\&C} \\
 \midrule
    Quality  & $\phantom{0}0$ & $30$ & $\mathbf{15}$ & $\phantom{0}0$ & $14$ & $\mathbf{31}$ & $\phantom{0}0$ & $15$ & $\mathbf{27}$ \\
    Time  & $15$ & $\phantom{0}0$ & $\mathbf{30}$ & $\mathbf{37}$ & $\phantom{0}0$ & $\phantom{0}8$ & $\mathbf{39}$ & $\phantom{0}0$ & $\phantom{0}3$ \\
    Pareto  & $\phantom{0}7$ & $\phantom{0}0$ & $\mathbf{30}$ & $\mathbf{10}$ & $\phantom{0}0$ & $\phantom{0}8$ & $\mathbf{13}$ & $\phantom{0}0$ & $\phantom{0}3$ \\

            & \multicolumn{3}{c}{ } & \multicolumn{3}{c}{ } & \multicolumn{3}{c}{ } \\ \cmidrule(lr){2-4}\cmidrule(lr){5-7}\cmidrule(lr){8-10}
            & ALNS & tie & RB\&C & ALNS & tie & RB\&C & ALNS & tie & \multicolumn{1}{c}{RB\&C} \\
            \midrule
    Quality  & $0$ & $40$ & $\phantom{0}\mathbf{5}$ & $0$ & $29$ & $\mathbf{16}$ & $0$ & $29$ & $\mathbf{13}$ \\
    Time  & $1$ & $\phantom{0}0$ & $\mathbf{44}$ & $4$ & $\phantom{0}0$ & $\mathbf{41}$ & $8$ & $\phantom{0}0$ & $\mathbf{34}$ \\
    Pareto  & $1$ & $\phantom{0}0$ & $\mathbf{44}$ & $1$ & $\phantom{0}0$ & $\mathbf{41}$ & $5$ & $\phantom{0}0$ & $\mathbf{34}$ \\

            & \multicolumn{3}{c}{ } & \multicolumn{3}{c}{ } & \multicolumn{3}{c}{ } \\ \cmidrule(lr){2-4}\cmidrule(lr){5-7}\cmidrule(lr){8-10}
            & FST & tie & RB\&C & FST & tie & RB\&C & FST & tie & \multicolumn{1}{c}{RB\&C} \\
            \midrule
    Quality  & $\phantom{0}0$ & $45$ & $\phantom{0}0$ & $\phantom{0}0$ & $45$ & $\phantom{0}0$ & $\phantom{0}0$ & $45$ & $\phantom{0}0$ \\
    Time  & $14$ & $\phantom{0}6$ & $\mathbf{25}$ & $17$ & $\phantom{0}2$ & $\mathbf{26}$ & $18$ & $\phantom{0}1$ & $\mathbf{26}$ \\
    Pareto  & $14$ & $\phantom{0}6$ & $\mathbf{25}$ & $17$ & $\phantom{0}2$ & $\mathbf{26}$ & $18$ & $\phantom{0}1$ & $\mathbf{26}$ \\
    \bottomrule
  \end{tabular}
\end{table}

\begin{table}[!ht]
  \centering
  \caption{Comparison in large-sized instances against state-of-the-art algorithms in terms of quality, time and Pareto efficiency.}\label{tab:pareto_large}
  \begin{tabular}{lccccccccc}
    \toprule
 & \multicolumn{3}{c}{Gen1} & \multicolumn{3}{c}{Gen2} & \multicolumn{3}{c}{Gen3} \\ \cmidrule(lr){2-4}\cmidrule(lr){5-7}\cmidrule(lr){8-10}
 & EA4OP & tie & RB\&C & EA4OP & tie & RB\&C & EA4OP & tie & \multicolumn{1}{c}{RB\&C} \\
 \midrule
    Quality  & $\phantom{0}1$ & $0$ & $\mathbf{40}$ & $\phantom{0}5$ & $0$ & $\mathbf{36}$ & $\phantom{0}3$ & $0$ & $\mathbf{33}$ \\
    Time  & $\mathbf{39}$ & $0$ & $\phantom{0}2$ & $\mathbf{40}$ & $1$ & $\phantom{0}0$ & $\mathbf{35}$ & $1$ & $\phantom{0}0$ \\
    Pareto  & $\phantom{0}1$ & $0$ & $\phantom{0}\mathbf{2}$ & $\phantom{0}\mathbf{5}$ & $0$ & $\phantom{0}1$ & $\phantom{0}\mathbf{3}$ & $0$ & $\phantom{0}1$ \\

            & \multicolumn{3}{c}{ } & \multicolumn{3}{c}{ } & \multicolumn{3}{c}{ } \\ \cmidrule(lr){2-4}\cmidrule(lr){5-7}\cmidrule(lr){8-10}
            & ALNS & tie & RB\&C & ALNS & tie & RB\&C & ALNS & tie & \multicolumn{1}{c}{RB\&C} \\
            \midrule
    Quality  & $2$ & $\phantom{0}2$ & $\mathbf{37}$ & $\phantom{0}4$ & $\phantom{0}1$ & $\mathbf{36}$ & $\phantom{0}4$ & $\phantom{0}0$ & $\mathbf{32}$ \\
    Time  & $6$ & $11$ & $\mathbf{24}$ & $\mathbf{13}$ & $25$ & $\phantom{0}3$ & $\mathbf{13}$ & $19$ & $\phantom{0}4$ \\
    Pareto  & $4$ & $\phantom{0}0$ & $\mathbf{34}$ & $\phantom{0}5$ & $\phantom{0}0$ & $\mathbf{24}$ & $\phantom{0}4$ & $\phantom{0}0$ & $\mathbf{20}$ \\

            & \multicolumn{3}{c}{ } & \multicolumn{3}{c}{ } & \multicolumn{3}{c}{ } \\ \cmidrule(lr){2-4}\cmidrule(lr){5-7}\cmidrule(lr){8-10}
            & FST & tie & RB\&C & FST & tie & RB\&C & FST & tie & \multicolumn{1}{c}{RB\&C} \\
            \midrule
    Quality  & $0$ & $13$ & $\mathbf{28}$ & $0$ & $\phantom{0}9$ & $\mathbf{32}$ & $3$ & $11$ & $\mathbf{27}$ \\
    Time  & $1$ & $\phantom{0}5$ & $\mathbf{35}$ & $8$ & $13$ & $\mathbf{20}$ & $5$ & $17$ & $\mathbf{19}$ \\
    Pareto  & $1$ & $\phantom{0}1$ & $\mathbf{39}$ & $8$ & $\phantom{0}0$ & $\mathbf{33}$ & $7$ & $\phantom{0}2$ & $\mathbf{32}$ \\
    \bottomrule
  \end{tabular}
\end{table}

Table~\ref{tab:pareto_medium} shows that the RB\&C algorithm is competitive in medium-sized instances. Compared to the ALNS heuristic and FST algorithm, it obtains better Pareto efficiency results in the three generations.
Comparing it to EA4OP, the Pareto efficiency is lower because the heuristic is a faster algorithm. Nevertheless, the RB\&C obtains much better quality solutions.

Table~\ref{tab:pareto_large} shows that RB\&C is the best performing algorithm in large-sized instances. Particularly, it behaves better than the FST algorithm, obtaining the best quality and time solutions in most of the instances, hence obtaining better Pareto results. The ALNS algorithm is able to return some solutions with better quality or running time, however, overall, the RB\&C performs better in large-sized instances.
The EA4OP metaheuristic is faster than the B\&C but, in general, obtains worse quality solutions.

\begin{table}[!ht]
  \caption{New best-known optimum, lower bound and upper bound values.}\label{tab:new_all}
  \centering
  \begin{tabular}{lccc}
    \toprule
  & OPT & LB & \multicolumn{1}{c}{UB} \\
  \midrule
    Gen1  & $12$ & $25$ & $28$ \\
    Gen2  & $\phantom{0}2$ & $27$ & $28$ \\
    Gen3  & $\phantom{0}4$ & $24$ & $29$ \\
    \midrule
    All  & $18$ & $76$ & $85$ \\
    \bottomrule
  \end{tabular}
\end{table}

Finally, in Table~\ref{tab:new_all}, we summarize the new best-known results obtained in the experiments. The RB\&C algorithm obtains 18 new optimality certifications, 76 new best-known solution values and 85 new upper-bound values.

\section{Conclusions and Future Work}\label{sec:concl}

We have presented a revisited version of the B\&C algorithm for the OP that brings multiple contributions together. We have proposed a joint separation algorithm for SECs and CCs, which efficiently uses the shrinking technique for cycle problems by reducing the adverse effects of the shrinking for CCs.
We have developed two blossom heuristics for cycle problems which generalize the well-known approaches in the literature of the TSP. We have designed an efficient variable pricing procedure for the OP which enables us to avoid repetitive calculations and to skip the exact calculation of the reduced cost of some variables. We have proposed a separation loop for the OP that takes into consideration the different contributions and separation costs of the valid inequalities. We have used alternative primal heuristics, one of which is based on a metaheuristic, and a mechanism to update the global upper bound during the branching phase to tighten the lower and upper bounds for the cases when the algorithm finishes without an optimality certification.

The experiments have shown that the RB\&C algorithm for OP is a more efficient approach than the state-of-the-art B\&C algorithm.
The introduced algorithm has increased the number of solved problems, obtained better running times in more instances, succeeded in returning new optimality certifications, new best known solutions, and new upper-bound values for large problems. Additionally, it has been shown that the RB\&C algorithm obtains better quality solutions than the state-of-the-art heuristics for the OP within the 5-hour running time limit.

Nevertheless, there are many research lines that remain open after this work. One of the most demanding aspects to improve in the presented approach is the implementation of advanced branching techniques.  The use of more general cuts, such as combs and clique trees, and the development of their respective separation algorithms for cycle problems might help to improve the performance of the RB\&C algorithm.
All these future contributions might help to solve the remaining instances until optimality, but we can anticipate it will not be an easy challenge. Implementing the contributions in this paper to other cycle problems which are different from the OP will definitely help to comprehend their importance in the context of cycle problems with a more general view.

        {\textbf{Acknowledgements}}  \ \
{\small
  The first and second authors are partially supported by the project PID2019-104933GB-I00 (Spanish Ministry of Science and Innovation).
  The first and third authors are partially supported by the projects BERC 2018{-}2021 (Basque Government) and by SEV{-}2017{-}0718 (Spanish Ministry of Economy and Competitiveness).
  The first author is also supported by the grant BES{-}2015{-}072036 (Spanish Ministry of Economy and Competitiveness) and project ELKARTEK (Basque Government).
  The second author is supported by IT{-}1252{-}19 (Basque Government) and GIU17/011 (University of the Basque Country).
  The third author is also supported by  IT{-}1244{-}19 (Basque Government) and TIN2016{-}78365R (Spanish Ministry of Science and Innovation).
  We gratefully acknowledge the authors of the TSP solver Concorde for making their code available to the public, since it has been the working basis of our implementations.  We also thank Prof. J.J. Salazar-Gonzalez who provided us with the codes used in~\cite{fischetti5}.
}


\clearpage

  \clearpage
  \appendix
  \renewcommand*{\thesection}{\Alph{section}}

  \section*{Appendices}\label{appendix}
  \section{Configuration of Components: Detailed Results}\label{appendix:valid}

  In this section, we show the detailed results of the alternative RB\&C configurations by instances and generations. Each configuration has been executed five times with a 5-hour execution time limit. We show the obtained results of the configuration in terms of lower-bound values, LB, upper-bound values, UB, and time (in seconds) performance, Time. For the LB and UB, the obtained best value for each configuration (the maximum for LB and the minimum for the UB) is presented in the Best column. Regarding the Time, the Mean column shows the meantime of the five executions. The Gap column represents the relative distance to best-known value (higher Best value in the case of LB, and lower Best in the case of UB and Mean in the case of Time, respectively).

\begin{landscape}
  \begin{table}[htb!]
    \centering
    \tiny
    \begin{tabular}{lcccccccccccccccccc}
      \toprule
 & \multicolumn{18}{c}{Gen} \\ \cmidrule(lr){2-19}
 & \multicolumn{6}{c}{Gen1} & \multicolumn{6}{c}{Gen2} & \multicolumn{6}{c}{Gen3} \\ \cmidrule(lr){2-7}\cmidrule(lr){8-13}\cmidrule(lr){14-19}
 & \multicolumn{2}{c}{LB} & \multicolumn{2}{c}{UB} & \multicolumn{2}{c}{Time} & \multicolumn{2}{c}{LB} & \multicolumn{2}{c}{UB} & \multicolumn{2}{c}{Time} & \multicolumn{2}{c}{LB} & \multicolumn{2}{c}{UB} & \multicolumn{2}{c}{Time} \\ \cmidrule(lr){2-3}\cmidrule(lr){4-5}\cmidrule(lr){6-7}\cmidrule(lr){8-9}\cmidrule(lr){10-11}\cmidrule(lr){12-13}\cmidrule(lr){14-15}\cmidrule(lr){16-17}\cmidrule(lr){18-19}
      Strategy  & Best & Gap & Best & Gap & Mean & Gap & Best & Gap & Best & Gap & Mean & Gap & Best & Gap & Best & Gap & Mean & \multicolumn{1}{c}{Gap} \\
      \midrule
      REFERENCE  & $\phantom{00}49$ & $0$ & $\phantom{00}49$ & $0$ & $0.04$ & $123.66$ & $2708$ & $0$ & $2708$ & $0$ & $1.13$ & $\phantom{0}90.94$ & $2430$ & $0$ & $2430$ & $0$ & $1.03$ & $\phantom{0}39.55$ \\
      \midrule
      - SRK  & $\phantom{00}49$ & $0$ & $\phantom{00}49$ & $0$ & $0.04$ & $119.35$ & $2708$ & $0$ & $2708$ & $0$ & $1.21$ & $104.70$ & $2430$ & $0$ & $2430$ & $0$ & $1.06$ & $\phantom{0}42.60$ \\
      SRK=C1C2S3  & $\phantom{00}49$ & $0$ & $\phantom{00}49$ & $0$ & $0.04$ & $111.83$ & $2708$ & $0$ & $2708$ & $0$ & $1.38$ & $133.98$ & $2430$ & $0$ & $2430$ & $0$ & $0.90$ & $\phantom{0}21.84$ \\
      - CC STRATS  & $\phantom{00}49$ & $0$ & $\phantom{00}49$ & $0$ & $0.04$ & $124.73$ & $2708$ & $0$ & $2708$ & $0$ & $1.13$ & $\phantom{0}90.57$ & $2430$ & $0$ & $2430$ & $0$ & $1.04$ & $\phantom{0}39.74$ \\
      \midrule
      - EPH BLOSSOM  & $\phantom{00}49$ & $0$ & $\phantom{00}49$ & $0$ & $0.03$ & $\phantom{0}64.52$ & $2708$ & $0$ & $2708$ & $0$ & $1.44$ & $143.58$ & $2430$ & $0$ & $2430$ & $0$ & $0.74$ & $\phantom{00}0.00$ \\
      - EGH BLOSSOM  & $\phantom{00}49$ & $0$ & $\phantom{00}49$ & $0$ & $0.02$ & $\phantom{00}0.00$ & $2708$ & $0$ & $2708$ & $0$ & $1.22$ & $106.29$ & $2430$ & $0$ & $2430$ & $0$ & $0.97$ & $\phantom{0}30.56$ \\
      + FST BLOSSOM  & $\phantom{00}49$ & $0$ & $\phantom{00}49$ & $0$ & $0.09$ & $398.92$ & $2708$ & $0$ & $2708$ & $0$ & $1.32$ & $123.29$ & $2430$ & $0$ & $2430$ & $0$ & $0.85$ & $\phantom{0}14.47$ \\
      \midrule
      - EDGE COVER  & $\phantom{00}49$ & $0$ & $\phantom{00}49$ & $0$ & $0.03$ & $\phantom{0}83.87$ & $2708$ & $0$ & $2708$ & $0$ & $1.33$ & $125.56$ & $2430$ & $0$ & $2430$ & $0$ & $1.76$ & $136.91$ \\
      - CYCLE COVER  & $\phantom{00}49$ & $0$ & $\phantom{00}49$ & $0$ & $0.05$ & $174.19$ & $2708$ & $0$ & $2708$ & $0$ & $1.20$ & $103.38$ & $2430$ & $0$ & $2430$ & $0$ & $0.97$ & $\phantom{0}30.70$ \\
      - PATH  & $\phantom{00}49$ & $0$ & $\phantom{00}49$ & $0$ & $0.04$ & $116.13$ & $2708$ & $0$ & $2708$ & $0$ & $1.39$ & $135.40$ & $2430$ & $0$ & $2430$ & $0$ & $0.79$ & $\phantom{00}7.02$ \\
      + VERTEX COVER  & $\phantom{00}49$ & $0$ & $\phantom{00}49$ & $0$ & $0.04$ & $104.30$ & $2708$ & $0$ & $2708$ & $0$ & $1.11$ & $\phantom{0}87.02$ & $2430$ & $0$ & $2430$ & $0$ & $0.98$ & $\phantom{0}31.94$ \\
      \midrule
      SEP: TWO SUBLOOPS  & $\phantom{00}49$ & $0$ & $\phantom{00}49$ & $0$ & $0.07$ & $266.67$ & $2708$ & $0$ & $2708$ & $0$ & $0.95$ & $\phantom{0}60.62$ & $2430$ & $0$ & $2430$ & $0$ & $1.00$ & $\phantom{0}34.99$ \\
      \midrule
      BRANCH HEUR=PB  & $\phantom{00}49$ & $0$ & $\phantom{00}49$ & $0$ & $0.05$ & $175.27$ & $2708$ & $0$ & $2708$ & $0$ & $0.59$ & $\phantom{00}0.00$ & $2430$ & $0$ & $2430$ & $0$ & $1.41$ & $\phantom{0}90.47$ \\
      BRANCH HEUR=VP - EA4OP  & $\phantom{00}49$ & $0$ & $\phantom{00}49$ & $0$ & $0.04$ & $119.35$ & $2708$ & $0$ & $2708$ & $0$ & $0.66$ & $\phantom{0}11.22$ & $2430$ & $0$ & $2430$ & $0$ & $0.96$ & $\phantom{0}29.35$ \\
      \bottomrule
    \end{tabular}
    \caption{pr76.}\label{tab:pr76-1}
  \end{table}

  \begin{table}[htb!]
    \centering
    \tiny
    \begin{tabular}{lcccccccccccccccccc}
      \toprule
 & \multicolumn{18}{c}{Gen} \\ \cmidrule(lr){2-19}
 & \multicolumn{6}{c}{Gen1} & \multicolumn{6}{c}{Gen2} & \multicolumn{6}{c}{Gen3} \\ \cmidrule(lr){2-7}\cmidrule(lr){8-13}\cmidrule(lr){14-19}
 & \multicolumn{2}{c}{LB} & \multicolumn{2}{c}{UB} & \multicolumn{2}{c}{Time} & \multicolumn{2}{c}{LB} & \multicolumn{2}{c}{UB} & \multicolumn{2}{c}{Time} & \multicolumn{2}{c}{LB} & \multicolumn{2}{c}{UB} & \multicolumn{2}{c}{Time} \\ \cmidrule(lr){2-3}\cmidrule(lr){4-5}\cmidrule(lr){6-7}\cmidrule(lr){8-9}\cmidrule(lr){10-11}\cmidrule(lr){12-13}\cmidrule(lr){14-15}\cmidrule(lr){16-17}\cmidrule(lr){18-19}
      Strategy  & Best & Gap & Best & Gap & Mean & Gap & Best & Gap & Best & Gap & Mean & Gap & Best & Gap & Best & Gap & Mean & \multicolumn{1}{c}{Gap} \\
      \midrule
      REFERENCE  & $\phantom{00}363$ & $0.00$ & $\phantom{00}363$ & $0.00$ & $\phantom{00}359.51$ & $1031.58$ & $19633$ & $0.06$ & $19801$ & $0.01$ & $18000.00$ & $\phantom{000}0.00$ & $15498$ & $0.00$ & $15498$ & $0.00$ & $\phantom{00}166.80$ & $\phantom{00}29.99$ \\
      \midrule
      - SRK  & $\phantom{00}363$ & $0.00$ & $\phantom{00}363$ & $0.00$ & $\phantom{00}643.50$ & $1925.50$ & $19635$ & $0.05$ & $19800$ & $0.01$ & $18000.00$ & $\phantom{000}0.00$ & $15498$ & $0.00$ & $15498$ & $0.00$ & $\phantom{00}219.86$ & $\phantom{00}71.34$ \\
      SRK=C1C2S3  & $\phantom{00}363$ & $0.00$ & $\phantom{00}363$ & $0.00$ & $\phantom{00}120.89$ & $\phantom{0}280.53$ & $19634$ & $0.05$ & $19800$ & $0.01$ & $18000.00$ & $\phantom{000}0.00$ & $15498$ & $0.00$ & $15498$ & $0.00$ & $\phantom{00}284.01$ & $\phantom{0}121.34$ \\
      - CC STRATS  & $\phantom{00}363$ & $0.00$ & $\phantom{00}363$ & $0.00$ & $\phantom{00}118.09$ & $\phantom{0}271.70$ & $19633$ & $0.06$ & $19802$ & $0.02$ & $18000.00$ & $\phantom{000}0.00$ & $15498$ & $0.00$ & $15498$ & $0.00$ & $\phantom{0}2696.49$ & $2001.40$ \\
      \midrule
      - EPH BLOSSOM  & $\phantom{00}363$ & $0.00$ & $\phantom{00}363$ & $0.00$ & $\phantom{000}31.77$ & $\phantom{000}0.00$ & $19643$ & $0.01$ & $19801$ & $0.01$ & $18000.00$ & $\phantom{000}0.00$ & $15498$ & $0.00$ & $15498$ & $0.00$ & $\phantom{00}316.47$ & $\phantom{0}146.63$ \\
      - EGH BLOSSOM  & $\phantom{00}363$ & $0.00$ & $\phantom{00}363$ & $0.00$ & $\phantom{00}420.83$ & $1224.61$ & $19634$ & $0.05$ & $19801$ & $0.01$ & $18000.00$ & $\phantom{000}0.00$ & $15498$ & $0.00$ & $15498$ & $0.00$ & $\phantom{00}252.53$ & $\phantom{00}96.80$ \\
      + FST BLOSSOM  & $\phantom{00}363$ & $0.00$ & $\phantom{00}363$ & $0.00$ & $\phantom{00}423.05$ & $1231.61$ & $19644$ & $0.00$ & $19801$ & $0.01$ & $18000.00$ & $\phantom{000}0.00$ & $15498$ & $0.00$ & $15498$ & $0.00$ & $\phantom{00}210.91$ & $\phantom{00}64.36$ \\
      \midrule
      - EDGE COVER  & $\phantom{00}363$ & $0.00$ & $\phantom{00}363$ & $0.00$ & $\phantom{00}176.79$ & $\phantom{0}456.47$ & $19636$ & $0.04$ & $19800$ & $0.01$ & $18000.00$ & $\phantom{000}0.00$ & $15498$ & $0.00$ & $15498$ & $0.00$ & $\phantom{00}180.40$ & $\phantom{00}40.59$ \\
      - CYCLE COVER  & $\phantom{00}363$ & $0.00$ & $\phantom{00}363$ & $0.00$ & $\phantom{00}110.11$ & $\phantom{0}246.59$ & $19642$ & $0.01$ & $19801$ & $0.01$ & $18000.00$ & $\phantom{000}0.00$ & $15498$ & $0.00$ & $15498$ & $0.00$ & $\phantom{00}221.88$ & $\phantom{00}72.91$ \\
      - PATH  & $\phantom{00}363$ & $0.00$ & $\phantom{00}363$ & $0.00$ & $\phantom{00}252.18$ & $\phantom{0}693.77$ & $19629$ & $0.08$ & $19801$ & $0.01$ & $18000.00$ & $\phantom{000}0.00$ & $15498$ & $0.00$ & $15498$ & $0.00$ & $\phantom{00}212.04$ & $\phantom{00}65.25$ \\
      + VERTEX COVER  & $\phantom{00}363$ & $0.00$ & $\phantom{00}363$ & $0.00$ & $\phantom{000}81.69$ & $\phantom{0}157.14$ & $19637$ & $0.04$ & $19799$ & $0.00$ & $18000.00$ & $\phantom{000}0.00$ & $15498$ & $0.00$ & $15498$ & $0.00$ & $\phantom{00}305.62$ & $\phantom{0}138.18$ \\
      \midrule
      SEP: TWO SUBLOOPS  & $\phantom{00}363$ & $0.00$ & $\phantom{00}363$ & $0.00$ & $\phantom{00}300.17$ & $\phantom{0}844.81$ & $19631$ & $0.07$ & $19801$ & $0.01$ & $18000.00$ & $\phantom{000}0.00$ & $15498$ & $0.00$ & $15498$ & $0.00$ & $\phantom{00}146.51$ & $\phantom{00}14.18$ \\
      \midrule
      BRANCH HEUR=PB  & $\phantom{00}363$ & $0.00$ & $\phantom{00}363$ & $0.00$ & $\phantom{00}190.93$ & $\phantom{0}500.97$ & $19611$ & $0.17$ & $19800$ & $0.01$ & $18000.00$ & $\phantom{000}0.00$ & $15498$ & $0.00$ & $15498$ & $0.00$ & $\phantom{00}194.63$ & $\phantom{00}51.68$ \\
      BRANCH HEUR=VP - EA4OP  & $\phantom{00}363$ & $0.00$ & $\phantom{00}363$ & $0.00$ & $\phantom{00}270.75$ & $\phantom{0}752.20$ & $19619$ & $0.13$ & $19801$ & $0.01$ & $18000.00$ & $\phantom{000}0.00$ & $15498$ & $0.00$ & $15498$ & $0.00$ & $\phantom{0}1000.74$ & $\phantom{0}679.89$ \\
      \bottomrule
    \end{tabular}
    \caption{att532.}\label{tab:att532-1}
  \end{table}

  \begin{table}[htb!]
    \centering
    \tiny
    \begin{tabular}{lcccccccccccccccccc}
      \toprule
 & \multicolumn{18}{c}{Gen} \\ \cmidrule(lr){2-19}
 & \multicolumn{6}{c}{Gen1} & \multicolumn{6}{c}{Gen2} & \multicolumn{6}{c}{Gen3} \\ \cmidrule(lr){2-7}\cmidrule(lr){8-13}\cmidrule(lr){14-19}
 & \multicolumn{2}{c}{LB} & \multicolumn{2}{c}{UB} & \multicolumn{2}{c}{Time} & \multicolumn{2}{c}{LB} & \multicolumn{2}{c}{UB} & \multicolumn{2}{c}{Time} & \multicolumn{2}{c}{LB} & \multicolumn{2}{c}{UB} & \multicolumn{2}{c}{Time} \\ \cmidrule(lr){2-3}\cmidrule(lr){4-5}\cmidrule(lr){6-7}\cmidrule(lr){8-9}\cmidrule(lr){10-11}\cmidrule(lr){12-13}\cmidrule(lr){14-15}\cmidrule(lr){16-17}\cmidrule(lr){18-19}
      Strategy  & Best & Gap & Best & Gap & Mean & Gap & Best & Gap & Best & Gap & Mean & Gap & Best & Gap & Best & Gap & Mean & \multicolumn{1}{c}{Gap} \\
      \midrule
      REFERENCE  & $\phantom{00}777$ & $0.00$ & $\phantom{00}777$ & $0.00$ & $\phantom{0}5378.5$ & $144.79$ & $40770$ & $0.02$ & $40954$ & $0.02$ & $18000.0$ & $\phantom{00}0.00$ & $37669$ & $0.00$ & $37669$ & $0.00$ & $\phantom{0}4735.9$ & $150.57$ \\
      \midrule
      - SRK  & $\phantom{00}777$ & $0.00$ & $\phantom{00}777$ & $0.00$ & $\phantom{0}9969.9$ & $353.75$ & $40777$ & $0.00$ & $40952$ & $0.01$ & $18000.0$ & $\phantom{00}0.00$ & $37669$ & $0.00$ & $37669$ & $0.00$ & $12469.7$ & $559.74$ \\
      SRK=C1C2S3  & $\phantom{00}777$ & $0.00$ & $\phantom{00}777$ & $0.00$ & $\phantom{0}2731.9$ & $\phantom{0}24.34$ & $40765$ & $0.03$ & $40953$ & $0.02$ & $18000.0$ & $\phantom{00}0.00$ & $37669$ & $0.00$ & $37669$ & $0.00$ & $\phantom{0}6725.9$ & $255.85$ \\
      - CC STRATS  & $\phantom{00}777$ & $0.00$ & $\phantom{00}777$ & $0.00$ & $\phantom{0}4937.8$ & $124.73$ & $40772$ & $0.01$ & $40953$ & $0.02$ & $18000.0$ & $\phantom{00}0.00$ & $37669$ & $0.00$ & $37669$ & $0.00$ & $\phantom{0}5828.3$ & $208.36$ \\
      \midrule
      - EPH BLOSSOM  & $\phantom{00}777$ & $0.00$ & $\phantom{00}777$ & $0.00$ & $13669.6$ & $522.14$ & $40777$ & $0.00$ & $41006$ & $0.15$ & $18000.0$ & $\phantom{00}0.00$ & $37665$ & $0.01$ & $37758$ & $0.24$ & $18000.0$ & $852.33$ \\
      - EGH BLOSSOM  & $\phantom{00}777$ & $0.00$ & $\phantom{00}777$ & $0.00$ & $\phantom{0}7073.6$ & $221.94$ & $40773$ & $0.01$ & $40948$ & $0.00$ & $18000.0$ & $\phantom{00}0.00$ & $37669$ & $0.00$ & $37669$ & $0.00$ & $\phantom{0}8161.1$ & $331.78$ \\
      + FST BLOSSOM  & $\phantom{00}777$ & $0.00$ & $\phantom{00}777$ & $0.00$ & $\phantom{0}2197.2$ & $\phantom{00}0.00$ & $40775$ & $0.00$ & $40946$ & $0.00$ & $18000.0$ & $\phantom{00}0.00$ & $37669$ & $0.00$ & $37669$ & $0.00$ & $\phantom{0}6688.2$ & $253.86$ \\
      \midrule
      - EDGE COVER  & $\phantom{00}777$ & $0.00$ & $\phantom{00}777$ & $0.00$ & $\phantom{0}3303.3$ & $\phantom{0}50.34$ & $40773$ & $0.01$ & $40954$ & $0.02$ & $18000.0$ & $\phantom{00}0.00$ & $37669$ & $0.00$ & $37669$ & $0.00$ & $\phantom{0}1890.1$ & $\phantom{00}0.00$ \\
      - CYCLE COVER  & $\phantom{00}777$ & $0.00$ & $\phantom{00}777$ & $0.00$ & $\phantom{0}4072.0$ & $\phantom{0}85.33$ & $40775$ & $0.00$ & $40950$ & $0.01$ & $18000.0$ & $\phantom{00}0.00$ & $37669$ & $0.00$ & $37669$ & $0.00$ & $\phantom{0}4485.4$ & $137.31$ \\
      - PATH  & $\phantom{00}777$ & $0.00$ & $\phantom{00}777$ & $0.00$ & $\phantom{0}4103.7$ & $\phantom{0}86.77$ & $40775$ & $0.00$ & $40952$ & $0.01$ & $18000.0$ & $\phantom{00}0.00$ & $37669$ & $0.00$ & $37669$ & $0.00$ & $\phantom{0}7045.8$ & $272.77$ \\
      + VERTEX COVER  & $\phantom{00}777$ & $0.00$ & $\phantom{00}777$ & $0.00$ & $\phantom{0}3165.5$ & $\phantom{0}44.07$ & $40777$ & $0.00$ & $40953$ & $0.02$ & $18000.0$ & $\phantom{00}0.00$ & $37669$ & $0.00$ & $37669$ & $0.00$ & $\phantom{0}8580.9$ & $353.99$ \\
      \midrule
      SEP: TWO SUBLOOPS  & $\phantom{00}777$ & $0.00$ & $\phantom{00}777$ & $0.00$ & $\phantom{0}5145.1$ & $134.17$ & $40773$ & $0.01$ & $40950$ & $0.01$ & $18000.0$ & $\phantom{00}0.00$ & $37669$ & $0.00$ & $37669$ & $0.00$ & $16501.5$ & $773.05$ \\
      \midrule
      BRANCH HEUR=PB  & $\phantom{00}777$ & $0.00$ & $\phantom{00}777$ & $0.00$ & $\phantom{0}2596.3$ & $\phantom{0}18.16$ & $40767$ & $0.02$ & $40955$ & $0.02$ & $18000.0$ & $\phantom{00}0.00$ & $37669$ & $0.00$ & $37669$ & $0.00$ & $\phantom{0}5133.0$ & $171.57$ \\
      BRANCH HEUR=VP - EA4OP  & $\phantom{00}777$ & $0.00$ & $\phantom{00}777$ & $0.00$ & $\phantom{0}3767.8$ & $\phantom{0}71.48$ & $40763$ & $0.03$ & $40956$ & $0.02$ & $18000.0$ & $\phantom{00}0.00$ & $37669$ & $0.00$ & $37669$ & $0.00$ & $\phantom{0}4421.5$ & $133.93$ \\
      \bottomrule
    \end{tabular}
    \caption{vm1084.}\label{vm1084-1}
  \end{table}

  \begin{table}[htb!]
    \centering
    \tiny
    \begin{tabular}{lcccccccccccccccccc}
      \toprule
 & \multicolumn{18}{c}{Gen} \\ \cmidrule(lr){2-19}
 & \multicolumn{6}{c}{Gen1} & \multicolumn{6}{c}{Gen2} & \multicolumn{6}{c}{Gen3} \\ \cmidrule(lr){2-7}\cmidrule(lr){8-13}\cmidrule(lr){14-19}
 & \multicolumn{2}{c}{LB} & \multicolumn{2}{c}{UB} & \multicolumn{2}{c}{Time} & \multicolumn{2}{c}{LB} & \multicolumn{2}{c}{UB} & \multicolumn{2}{c}{Time} & \multicolumn{2}{c}{LB} & \multicolumn{2}{c}{UB} & \multicolumn{2}{c}{Time} \\ \cmidrule(lr){2-3}\cmidrule(lr){4-5}\cmidrule(lr){6-7}\cmidrule(lr){8-9}\cmidrule(lr){10-11}\cmidrule(lr){12-13}\cmidrule(lr){14-15}\cmidrule(lr){16-17}\cmidrule(lr){18-19}
      Strategy  & Best & Gap & Best & Gap & Mean & Gap & Best & Gap & Best & Gap & Mean & Gap & Best & Gap & Best & Gap & Mean & \multicolumn{1}{c}{Gap} \\
      \midrule
      REFERENCE  & $\phantom{00}814$ & $0.00$ & $\phantom{00}814$ & $0.00$ & $\phantom{0}3565.7$ & $\phantom{0}10.26$ & $43377$ & $0.00$ & $43454$ & $0.18$ & $18000.0$ & $\phantom{0}24.62$ & $47162$ & $0.11$ & $47373$ & $0.00$ & $18000.0$ & $\phantom{00}0.00$ \\
      \midrule
      - SRK  & $\phantom{00}814$ & $0.00$ & $\phantom{00}814$ & $0.00$ & $11747.3$ & $263.25$ & $43378$ & $0.00$ & $43452$ & $0.17$ & $18000.0$ & $\phantom{0}24.62$ & $47195$ & $0.04$ & $47408$ & $0.08$ & $18000.0$ & $\phantom{00}0.00$ \\
      SRK=C1C2S3  & $\phantom{00}814$ & $0.00$ & $\phantom{00}814$ & $0.00$ & $\phantom{0}4039.8$ & $\phantom{0}24.92$ & $43378$ & $0.00$ & $43457$ & $0.18$ & $18000.0$ & $\phantom{0}24.62$ & $47212$ & $0.01$ & $47382$ & $0.02$ & $18000.0$ & $\phantom{00}0.00$ \\
      - CC STRATS  & $\phantom{00}814$ & $0.00$ & $\phantom{00}814$ & $0.00$ & $\phantom{0}5121.9$ & $\phantom{0}58.38$ & $43378$ & $0.00$ & $43378$ & $0.00$ & $17145.6$ & $\phantom{0}18.71$ & $47213$ & $0.00$ & $47386$ & $0.03$ & $18000.0$ & $\phantom{00}0.00$ \\
      \midrule
      - EPH BLOSSOM  & $\phantom{00}814$ & $0.00$ & $\phantom{00}819$ & $0.61$ & $18000.0$ & $456.60$ & $43371$ & $0.02$ & $43543$ & $0.38$ & $18000.0$ & $\phantom{0}24.62$ & $47075$ & $0.30$ & $47698$ & $0.69$ & $18000.0$ & $\phantom{00}0.00$ \\
      - EGH BLOSSOM  & $\phantom{00}814$ & $0.00$ & $\phantom{00}814$ & $0.00$ & $\phantom{0}4431.3$ & $\phantom{0}37.02$ & $43377$ & $0.00$ & $43455$ & $0.18$ & $18000.0$ & $\phantom{0}24.62$ & $47190$ & $0.05$ & $47394$ & $0.05$ & $18000.0$ & $\phantom{00}0.00$ \\
      + FST BLOSSOM  & $\phantom{00}814$ & $0.00$ & $\phantom{00}814$ & $0.00$ & $\phantom{0}6341.4$ & $\phantom{0}96.09$ & $43378$ & $0.00$ & $43378$ & $0.00$ & $15722.3$ & $\phantom{00}8.85$ & $47200$ & $0.03$ & $47371$ & $0.00$ & $18000.0$ & $\phantom{00}0.00$ \\
      \midrule
      - EDGE COVER  & $\phantom{00}814$ & $0.00$ & $\phantom{00}814$ & $0.00$ & $\phantom{0}6401.6$ & $\phantom{0}97.95$ & $43273$ & $0.24$ & $43456$ & $0.18$ & $18000.0$ & $\phantom{0}24.62$ & $47109$ & $0.22$ & $47381$ & $0.02$ & $18000.0$ & $\phantom{00}0.00$ \\
      - CYCLE COVER  & $\phantom{00}814$ & $0.00$ & $\phantom{00}814$ & $0.00$ & $\phantom{0}7045.5$ & $117.86$ & $43378$ & $0.00$ & $43449$ & $0.16$ & $18000.0$ & $\phantom{0}24.62$ & $47193$ & $0.05$ & $47385$ & $0.03$ & $18000.0$ & $\phantom{00}0.00$ \\
      - PATH  & $\phantom{00}814$ & $0.00$ & $\phantom{00}814$ & $0.00$ & $\phantom{0}3965.2$ & $\phantom{0}22.61$ & $43378$ & $0.00$ & $43446$ & $0.16$ & $18000.0$ & $\phantom{0}24.62$ & $47201$ & $0.03$ & $47379$ & $0.02$ & $18000.0$ & $\phantom{00}0.00$ \\
      + VERTEX COVER  & $\phantom{00}814$ & $0.00$ & $\phantom{00}814$ & $0.00$ & $\phantom{0}3233.9$ & $\phantom{00}0.00$ & $43377$ & $0.00$ & $43450$ & $0.17$ & $18000.0$ & $\phantom{0}24.62$ & $47171$ & $0.09$ & $47379$ & $0.02$ & $18000.0$ & $\phantom{00}0.00$ \\
      \midrule
      SEP: TWO SUBLOOPS  & $\phantom{00}814$ & $0.00$ & $\phantom{00}814$ & $0.00$ & $13939.5$ & $331.04$ & $43373$ & $0.01$ & $43451$ & $0.17$ & $18000.0$ & $\phantom{0}24.62$ & $47196$ & $0.04$ & $47378$ & $0.01$ & $18000.0$ & $\phantom{00}0.00$ \\
      \midrule
      BRANCH HEUR=PB  & $\phantom{00}814$ & $0.00$ & $\phantom{00}814$ & $0.00$ & $\phantom{0}9743.9$ & $201.30$ & $43378$ & $0.00$ & $43378$ & $0.00$ & $16153.4$ & $\phantom{0}11.84$ & $47215$ & $0.00$ & $47387$ & $0.03$ & $18000.0$ & $\phantom{00}0.00$ \\
      BRANCH HEUR=VP - EA4OP  & $\phantom{00}814$ & $0.00$ & $\phantom{00}814$ & $0.00$ & $\phantom{0}8707.5$ & $169.25$ & $43378$ & $0.00$ & $43449$ & $0.16$ & $18000.0$ & $\phantom{0}24.62$ & $47195$ & $0.04$ & $47376$ & $0.01$ & $18000.0$ & $\phantom{00}0.00$ \\
      \bottomrule
    \end{tabular}
    \caption{rl1323.}\label{rl1323-1}
  \end{table}

  \begin{table}[htb!]
    \centering
    \tiny
    \begin{tabular}{lcccccccccccccccccc}
      \toprule
 & \multicolumn{18}{c}{Gen} \\ \cmidrule(lr){2-19}
 & \multicolumn{6}{c}{Gen1} & \multicolumn{6}{c}{Gen2} & \multicolumn{6}{c}{Gen3} \\ \cmidrule(lr){2-7}\cmidrule(lr){8-13}\cmidrule(lr){14-19}
 & \multicolumn{2}{c}{LB} & \multicolumn{2}{c}{UB} & \multicolumn{2}{c}{Time} & \multicolumn{2}{c}{LB} & \multicolumn{2}{c}{UB} & \multicolumn{2}{c}{Time} & \multicolumn{2}{c}{LB} & \multicolumn{2}{c}{UB} & \multicolumn{2}{c}{Time} \\ \cmidrule(lr){2-3}\cmidrule(lr){4-5}\cmidrule(lr){6-7}\cmidrule(lr){8-9}\cmidrule(lr){10-11}\cmidrule(lr){12-13}\cmidrule(lr){14-15}\cmidrule(lr){16-17}\cmidrule(lr){18-19}
      Strategy  & Best & Gap & Best & Gap & Mean & Gap & Best & Gap & Best & Gap & Mean & Gap & Best & Gap & Best & Gap & Mean & \multicolumn{1}{c}{Gap} \\
      \midrule
      REFERENCE  & $\phantom{0}1276$ & $0.23$ & $\phantom{0}1282$ & $0.00$ & $18000$ & $0$ & $68013$ & $0.16$ & $68305$ & $0.01$ & $18000$ & $0$ & $71903$ & $0.01$ & $72018$ & $0.02$ & $18000$ & $0$ \\
      \midrule
      - SRK  & $\phantom{0}1271$ & $0.63$ & $\phantom{0}1282$ & $0.00$ & $18000$ & $0$ & $67812$ & $0.45$ & $68306$ & $0.01$ & $18000$ & $0$ & $71853$ & $0.08$ & $72012$ & $0.01$ & $18000$ & $0$ \\
      SRK=C1C2S3  & $\phantom{0}1278$ & $0.08$ & $\phantom{0}1282$ & $0.00$ & $18000$ & $0$ & $67863$ & $0.38$ & $68306$ & $0.01$ & $18000$ & $0$ & $71887$ & $0.03$ & $72010$ & $0.01$ & $18000$ & $0$ \\
      - CC STRATS  & $\phantom{0}1278$ & $0.08$ & $\phantom{0}1282$ & $0.00$ & $18000$ & $0$ & $68016$ & $0.15$ & $68304$ & $0.01$ & $18000$ & $0$ & $71894$ & $0.02$ & $72012$ & $0.01$ & $18000$ & $0$ \\
      \midrule
      - EPH BLOSSOM  & $\phantom{0}1273$ & $0.47$ & $\phantom{0}1284$ & $0.16$ & $18000$ & $0$ & $67735$ & $0.57$ & $68460$ & $0.23$ & $18000$ & $0$ & $71755$ & $0.21$ & $72118$ & $0.16$ & $18000$ & $0$ \\
      - EGH BLOSSOM  & $\phantom{0}1278$ & $0.08$ & $\phantom{0}1282$ & $0.00$ & $18000$ & $0$ & $68029$ & $0.14$ & $68311$ & $0.02$ & $18000$ & $0$ & $71854$ & $0.08$ & $72016$ & $0.02$ & $18000$ & $0$ \\
      + FST BLOSSOM  & $\phantom{0}1279$ & $0.00$ & $\phantom{0}1282$ & $0.00$ & $18000$ & $0$ & $67986$ & $0.20$ & $68300$ & $0.00$ & $18000$ & $0$ & $71773$ & $0.19$ & $72003$ & $0.00$ & $18000$ & $0$ \\
      \midrule
      - EDGE COVER  & $\phantom{0}1272$ & $0.55$ & $\phantom{0}1282$ & $0.00$ & $18000$ & $0$ & $67877$ & $0.36$ & $68306$ & $0.01$ & $18000$ & $0$ & $71873$ & $0.05$ & $72017$ & $0.02$ & $18000$ & $0$ \\
      - CYCLE COVER  & $\phantom{0}1275$ & $0.31$ & $\phantom{0}1282$ & $0.00$ & $18000$ & $0$ & $68055$ & $0.10$ & $68302$ & $0.00$ & $18000$ & $0$ & $71845$ & $0.09$ & $72014$ & $0.02$ & $18000$ & $0$ \\
      - PATH  & $\phantom{0}1274$ & $0.39$ & $\phantom{0}1282$ & $0.00$ & $18000$ & $0$ & $67831$ & $0.43$ & $68309$ & $0.01$ & $18000$ & $0$ & $71808$ & $0.14$ & $72013$ & $0.01$ & $18000$ & $0$ \\
      + VERTEX COVER  & $\phantom{0}1276$ & $0.23$ & $\phantom{0}1282$ & $0.00$ & $18000$ & $0$ & $68032$ & $0.13$ & $68300$ & $0.00$ & $18000$ & $0$ & $71883$ & $0.04$ & $72016$ & $0.02$ & $18000$ & $0$ \\
      \midrule
      SEP: TWO SUBLOOPS  & $\phantom{0}1276$ & $0.23$ & $\phantom{0}1282$ & $0.00$ & $18000$ & $0$ & $67967$ & $0.23$ & $68314$ & $0.02$ & $18000$ & $0$ & $71830$ & $0.11$ & $72017$ & $0.02$ & $18000$ & $0$ \\
      \midrule
      BRANCH HEUR=PB  & $\phantom{0}1274$ & $0.39$ & $\phantom{0}1282$ & $0.00$ & $18000$ & $0$ & $67830$ & $0.43$ & $68300$ & $0.00$ & $18000$ & $0$ & $71779$ & $0.18$ & $72017$ & $0.02$ & $18000$ & $0$ \\
      BRANCH HEUR=VP - EA4OP  & $\phantom{0}1278$ & $0.08$ & $\phantom{0}1282$ & $0.00$ & $18000$ & $0$ & $67981$ & $0.21$ & $68307$ & $0.01$ & $18000$ & $0$ & $71890$ & $0.03$ & $72016$ & $0.02$ & $18000$ & $0$ \\
      \bottomrule
    \end{tabular}
    \caption{vm1748.}\label{vm1748-1}
  \end{table}
\end{landscape}

\section{Comparison with state-of-the-art Algorithms: Detailed Results}\label{appendix:comp}

In this appendix, we detail the experimental results for the four algorithms (FST B\&C, EA4OP, ALNS and RB\&C). Table \ref{tab:best-II-1-medium} shows the results for medium-sized instances of generation 1, Table \ref{tab:best-II-1-big} for large-sized instances of  generation 1, Table \ref{tab:best-II-2-medium} for medium-sized instances of  generation 2, Table \ref{tab:best-II-2-big} for  large-sized instances of generation 2, Table \ref{tab:best-II-3-medium} for medium-sized instances of generation 3 and Table \ref{tab:best-II-3-big} for large-sized instances of generation 3.

In the Best column, we show the global best-known lower and upper-bound values. For each algorithm, we detail the best LB, the goodness gap GGap, the best UB, and the meantime (in seconds). The GGap represents the relative distance between the algorithm's best LB and the global best-known LB. For the RB\&C algorithm we also detail the optimality gap OGap which represents the relative distance between the obtained LB and UB by RB\&C.

For each algorithm, generation and size, we have calculated the average gap and running time over the instances where a feasible solution was obtained by the algorithm. In those instances where the time limit was reached, a running time of 5 hours has been used. These averages are shown in the last row of the tables.
The symbols in the tables mean the following:
\begin{itemize}
  \item[\;$*$\;:] best-known solution achieved
  \item[$-$\;:] not comparable result
  \item[\;.\;:] the code finished unexpectedly
\end{itemize}

\begin{landscape}

  \begin{table}[p!]
    \centering
    \caption{Generation 1, $n\le 400$}
    \label{tab:best-II-1-medium}
    \begin{scriptsize}
\begin{tabular}{rrrrrrrrrrrrrrrrrr}
  \toprule \multicolumn{1}{c}{} &\multicolumn{2}{c}{Best} &
                        \multicolumn{4}{c}{FST} &
                        \multicolumn{3}{c}{EA4OP} &
                        \multicolumn{3}{c}{ALNS} &
                        \multicolumn{5}{c}{RB\&C} \\
                       \cmidrule(lr){2-3}\cmidrule(lr){4-7}\cmidrule(lr){8-10}\cmidrule(lr){11-13}\cmidrule(lr){14-18}\   Instance & LB & UB
                          & LB & GGap & UB & Time
                          & LB & GGap & Time
                          & LB & GGap & Time
                          & LB & GGap & UB & OGap & Time\\
                          \midrule att48 & 31 & 31 & \textbf{31} & * & \textbf{31} & \textbf{0.00} & \textbf{31} & * & 0.25 & \textbf{31} & * & 6.77 & \textbf{31} & * & \textbf{31} & * & \textbf{0.03} \\ 
  gr48 & 31 & 31 & \textbf{31} & * & \textbf{31} & \textbf{0.00} & \textbf{31} & * & 0.13 & \textbf{31} & * & 9.99 & \textbf{31} & * & \textbf{31} & * & \textbf{0.02} \\ 
  hk48 & 30 & 30 & \textbf{30} & * & \textbf{30} & \textbf{0.00} & \textbf{30} & * & 0.24 & \textbf{30} & * & 7.20 & \textbf{30} & * & \textbf{30} & * & \textbf{0.01} \\ 
  eil51 & 29 & 29 & \textbf{29} & * & \textbf{29} & \textbf{0.00} & \textbf{29} & * & 0.24 & \textbf{29} & * & 9.51 & \textbf{29} & * & \textbf{29} & * & \textbf{0.01} \\ 
  berlin52 & 37 & 37 & \textbf{37} & * & \textbf{37} & \textbf{0.00} & \textbf{37} & * & 0.30 & \textbf{37} & * & 9.42 & \textbf{37} & * & \textbf{37} & * & \textbf{0.02} \\ 
  brazil58 & 46 & 46 & \textbf{46} & * & \textbf{46} & \textbf{0.00} & \textbf{46} & * & 1.00 & \textbf{46} & * & 9.13 & \textbf{46} & * & \textbf{46} & * & 0.07 \\ 
  st70 & 43 & 43 & \textbf{43} & * & \textbf{43} & 0.10 & \textbf{43} & * & 0.32 & \textbf{43} & * & 15.99 & \textbf{43} & * & \textbf{43} & * & \textbf{0.05} \\ 
  eil76 & 47 & 47 & \textbf{47} & * & \textbf{47} & 0.10 & 46 & 2.13 & 0.33 & \textbf{47} & * & 20.51 & \textbf{47} & * & \textbf{47} & * & \textbf{0.04} \\ 
  pr76 & 49 & 49 & \textbf{49} & * & \textbf{49} & 0.10 & \textbf{49} & * & 0.61 & \textbf{49} & * & 18.64 & \textbf{49} & * & \textbf{49} & * & \textbf{0.06} \\ 
  gr96 & 64 & 64 & \textbf{64} & * & \textbf{64} & 0.10 & \textbf{64} & * & 1.44 & \textbf{64} & * & 20.31 & \textbf{64} & * & \textbf{64} & * & \textbf{0.08} \\ 
  rat99 & 52 & 52 & \textbf{52} & * & \textbf{52} & \textbf{0.40} & \textbf{52} & * & 0.66 & \textbf{52} & * & 27.75 & \textbf{52} & * & \textbf{52} & * & 0.47 \\ 
  kroA100 & 56 & 56 & \textbf{56} & * & \textbf{56} & \textbf{0.40} & 55 & 1.79 & 0.34 & \textbf{56} & * & 34.75 & \textbf{56} & * & \textbf{56} & * & \textbf{0.41} \\ 
  kroB100 & 58 & 58 & \textbf{58} & * & \textbf{58} & 95.40 & 57 & 1.72 & 0.63 & \textbf{58} & * & 43.06 & \textbf{58} & * & \textbf{58} & * & \textbf{0.27} \\ 
  kroC100 & 56 & 56 & \textbf{56} & * & \textbf{56} & 0.40 & \textbf{56} & * & 0.48 & \textbf{56} & * & 34.32 & \textbf{56} & * & \textbf{56} & * & \textbf{0.25} \\ 
  kroD100 & 59 & 59 & \textbf{59} & * & \textbf{59} & 0.10 & 58 & 1.69 & 0.65 & \textbf{59} & * & 34.61 & \textbf{59} & * & \textbf{59} & * & \textbf{0.09} \\ 
  kroE100 & 57 & 57 & \textbf{57} & * & \textbf{57} & 159.20 & \textbf{57} & * & 0.50 & \textbf{57} & * & 32.26 & \textbf{57} & * & \textbf{57} & * & \textbf{5.53} \\ 
  rd100 & 61 & 61 & \textbf{61} & * & \textbf{61} & 0.20 & \textbf{61} & * & 0.74 & \textbf{61} & * & 29.49 & \textbf{61} & * & \textbf{61} & * & \textbf{0.12} \\ 
  eil101 & 64 & 64 & \textbf{64} & * & \textbf{64} & 0.10 & \textbf{64} & * & 0.79 & \textbf{64} & * & 31.73 & \textbf{64} & * & \textbf{64} & * & \textbf{0.06} \\ 
  lin105 & 66 & 66 & \textbf{66} & * & \textbf{66} & \textbf{0.30} & \textbf{66} & * & 1.42 & \textbf{66} & * & 32.11 & \textbf{66} & * & \textbf{66} & * & 0.48 \\ 
  pr107 & 54 & 54 & \textbf{54} & * & \textbf{54} & 0.30 & \textbf{54} & * & 0.93 & \textbf{54} & * & 78.46 & \textbf{54} & * & \textbf{54} & * & \textbf{0.08} \\ 
  gr120 & 75 & 75 & \textbf{75} & * & \textbf{75} & \textbf{0.10} & 74 & 1.33 & 1.20 & \textbf{75} & * & 29.58 & \textbf{75} & * & \textbf{75} & * & 0.28 \\ 
  pr124 & 75 & 75 & \textbf{75} & * & \textbf{75} & \textbf{0.30} & \textbf{75} & * & 1.11 & \textbf{75} & * & 49.64 & \textbf{75} & * & \textbf{75} & * & 0.35 \\ 
  bier127 & 103 & 103 & \textbf{103} & * & \textbf{103} & \textbf{0.30} & \textbf{103} & * & 1.18 & \textbf{103} & * & 40.84 & \textbf{103} & * & \textbf{103} & * & 0.38 \\ 
  pr136 & 71 & 71 & \textbf{71} & * & \textbf{71} & \textbf{1.40} & \textbf{71} & * & 0.96 & \textbf{71} & * & 29.97 & \textbf{71} & * & \textbf{71} & * & 1.75 \\ 
  gr137 & 81 & 81 & \textbf{81} & * & \textbf{81} & 1.50 & 78 & 3.70 & 3.44 & \textbf{81} & * & 59.21 & \textbf{81} & * & \textbf{81} & * & \textbf{0.24} \\ 
  pr144 & 77 & 77 & \textbf{77} & * & \textbf{77} & \textbf{1.30} & \textbf{77} & * & 2.61 & \textbf{77} & * & 87.82 & \textbf{77} & * & \textbf{77} & * & 1.46 \\ 
  kroA150 & 86 & 86 & \textbf{86} & * & \textbf{86} & 175.40 & \textbf{86} & * & 1.17 & \textbf{86} & * & 82.79 & \textbf{86} & * & \textbf{86} & * & \textbf{33.87} \\ 
  kroB150 & 87 & 87 & \textbf{87} & * & \textbf{87} & \textbf{1.20} & 86 & 1.15 & 1.00 & \textbf{87} & * & 61.64 & \textbf{87} & * & \textbf{87} & * & 2.21 \\ 
  pr152 & 77 & 77 & \textbf{77} & * & \textbf{77} & 1.40 & \textbf{77} & * & 3.64 & \textbf{77} & * & 91.38 & \textbf{77} & * & \textbf{77} & * & \textbf{1.29} \\ 
  u159 & 93 & 93 & \textbf{93} & * & \textbf{93} & 3.40 & 92 & 1.08 & 1.11 & \textbf{93} & * & 99.63 & \textbf{93} & * & \textbf{93} & * & \textbf{1.82} \\ 
  rat195 & 102 & 102 & \textbf{102} & * & \textbf{102} & \textbf{2.60} & 99 & 2.94 & 1.78 & \textbf{102} & * & 195.57 & \textbf{102} & * & \textbf{102} & * & 3.71 \\ 
  d198 & 123 & 123 & \textbf{123} & * & \textbf{123} & \textbf{3.20} & \textbf{123} & * & 6.68 & \textbf{123} & * & 65.57 & \textbf{123} & * & \textbf{123} & * & 5.28 \\ 
  kroA200 & 117 & 117 & \textbf{117} & * & \textbf{117} & \textbf{1.20} & \textbf{117} & * & 1.74 & \textbf{117} & * & 114.75 & \textbf{117} & * & \textbf{117} & * & 2.50 \\ 
  kroB200 & 119 & 119 & \textbf{119} & * & \textbf{119} & 14.10 & \textbf{119} & * & 1.67 & \textbf{119} & * & 86.58 & \textbf{119} & * & \textbf{119} & * & \textbf{9.91} \\ 
  gr202 & 145 & 145 & \textbf{145} & * & \textbf{145} & 12.70 & \textbf{145} & * & 6.89 & \textbf{145} & * & 187.56 & \textbf{145} & * & \textbf{145} & * & \textbf{2.71} \\ 
  ts225 & 124 & 124 & \textbf{124} & * & \textbf{124} & \textbf{10216.30} & \textbf{124} & * & 1.28 & \textbf{124} & * & 279.52 & \textbf{124} & * & 126 & 1.59 & 18000.00 \\ 
  tsp225 & 129 & 129 & \textbf{129} & * & \textbf{129} & 94.40 & 127 & 1.55 & 2.29 & 128 & 0.78 & 198.47 & \textbf{129} & * & \textbf{129} & * & \textbf{4.31} \\ 
  pr226 & 126 & 126 & \textbf{126} & * & \textbf{126} & 166.20 & \textbf{126} & * & 6.61 & \textbf{126} & * & 181.94 & \textbf{126} & * & \textbf{126} & * & \textbf{107.69} \\ 
  gr229 & 176 & 176 & \textbf{176} & * & \textbf{176} & 0.90 & \textbf{176} & * & 8.81 & 173 & 1.70 & 108.27 & \textbf{176} & * & \textbf{176} & * & \textbf{0.32} \\ 
  gil262 & 158 & 158 & \textbf{158} & * & \textbf{158} & 0.90 & 156 & 1.27 & 2.83 & \textbf{158} & * & 240.02 & \textbf{158} & * & \textbf{158} & * & \textbf{0.35} \\ 
  pr264 & 132 & 132 & \textbf{132} & * & \textbf{132} & 21.20 & \textbf{132} & * & 5.62 & \textbf{132} & * & 314.29 & \textbf{132} & * & \textbf{132} & * & \textbf{3.92} \\ 
  a280 & 147 & 147 & \textbf{147} & * & \textbf{147} & \textbf{13.60} & 143 & 2.72 & 3.00 & 144 & 2.04 & 239.06 & \textbf{147} & * & \textbf{147} & * & 40.65 \\ 
  pr299 & 162 & 162 & \textbf{162} & * & \textbf{162} & 111.50 & 160 & 1.23 & 3.12 & \textbf{162} & * & 410.90 & \textbf{162} & * & \textbf{162} & * & \textbf{48.85} \\ 
  lin318 & 205 & 205 & \textbf{205} & * & \textbf{205} & 22.40 & 202 & 1.46 & 7.15 & 203 & 0.98 & 294.23 & \textbf{205} & * & \textbf{205} & * & \textbf{5.49} \\ 
  rd400 & 239 & 239 & \textbf{239} & * & \textbf{239} & 37.40 & 234 & 2.09 & 6.59 & 237 & 0.84 & 422.56 & \textbf{239} & * & \textbf{239} & * & \textbf{36.71} \\ 
   \midrule\ average &  &  &  & * &  & \textbf{248.05} &  & 0.62 & 2.12 &  & 0.14 & 99.51 &  & * &  & 0.04 & 407.20 \\ 
   \bottomrule\end{tabular}

    \end{scriptsize}
  \end{table}

  \begin{table}[p!]
    \centering
    \caption{Generation 1, $n>400$}
    \label{tab:best-II-1-big}
    \begin{scriptsize}
\begin{tabular}{rrrrrrrrrrrrrrrrrr}
  \toprule \multicolumn{1}{c}{} &\multicolumn{2}{c}{Best} &
                        \multicolumn{4}{c}{FST} &
                        \multicolumn{3}{c}{EA4OP} &
                        \multicolumn{3}{c}{ALNS} &
                        \multicolumn{5}{c}{RB\&C} \\
                       \cmidrule(lr){2-3}\cmidrule(lr){4-7}\cmidrule(lr){8-10}\cmidrule(lr){11-13}\cmidrule(lr){14-18}\   Instance & LB & UB
                          & LB & GGap & UB & Time
                          & LB & GGap & Time
                          & LB & GGap & Time
                          & LB & GGap & UB & OGap & Time\\
                          \midrule fl417 & 228 & 230 & \textbf{228} & * & \textbf{230} & \textbf{18000.00} & 224 & 1.75 & 11.84 & \textbf{228} & * & 1056.07 & \textbf{228} & * & 231 & 1.30 & \textbf{18000.00} \\ 
  gr431 & 350 & 350 & \textbf{350} & * & \textbf{350} & 139.90 & 349 & 0.29 & 32.84 & 347 & 0.86 & 533.55 & \textbf{350} & * & \textbf{350} & * & \textbf{29.05} \\ 
  pr439 & 313 & 313 & \textbf{313} & * & \textbf{313} & 833.30 & 310 & 0.96 & 9.92 & 307 & 1.92 & 1263.74 & \textbf{313} & * & \textbf{313} & * & \textbf{414.00} \\ 
  pcb442 & 251 & 251 & \textbf{251} & * & \textbf{251} & 14.90 & 244 & 2.79 & 6.94 & 249 & 0.80 & 1328.72 & \textbf{251} & * & \textbf{251} & * & \textbf{7.21} \\ 
  d493 & 320 & 320 & \textbf{320} & * & \textbf{320} & 347.30 & 315 & 1.56 & 19.10 & 317 & 0.94 & 1291.93 & \textbf{320} & * & \textbf{320} & * & \textbf{13.37} \\ 
  att532 & 363 & 363 & \textbf{363} & * & \textbf{363} & 593.00 & 347 & 4.41 & 23.14 & 359 & 1.10 & 1380.54 & \textbf{363} & * & \textbf{363} & * & \textbf{312.50} \\ 
  ali535 & 425 & 426 & . & . & . & . & 424 & 0.24 & 73.03 & 422 & 0.71 & 1846.10 & \textbf{425} & * & \textbf{426} & 0.23 & \textbf{18000.00} \\ 
  pa561 & 357 & 357 & 356 & 0.28 & - & 2103.60 & 348 & 2.52 & 23.18 & 346 & 3.08 & 1605.42 & \textbf{357} & * & \textbf{357} & * & \textbf{245.42} \\ 
  u574 & 354 & 354 & \textbf{354} & * & \textbf{354} & 61.40 & 344 & 2.82 & 17.93 & 347 & 1.98 & 1204.18 & \textbf{354} & * & \textbf{354} & * & \textbf{24.00} \\ 
  rat575 & 322 & 322 & \textbf{322} & * & \textbf{322} & 59.50 & 309 & 4.04 & 13.76 & 317 & 1.55 & 3109.65 & \textbf{322} & * & \textbf{322} & * & \textbf{42.82} \\ 
  p654 & 343 & 396 & 327 & 4.66 & 553 & \textbf{18000.00} & 336 & 2.04 & 28.89 & \textbf{343} & * & 10866.70 & 342 & 0.29 & \textbf{396} & 13.64 & \textbf{18000.00} \\ 
  d657 & 386 & 386 & \textbf{386} & * & \textbf{386} & 715.70 & 377 & 2.33 & 23.24 & 380 & 1.55 & 3152.17 & \textbf{386} & * & \textbf{386} & * & \textbf{92.48} \\ 
  gr666 & 503 & 503 & \textbf{503} & * & \textbf{503} & 634.20 & 497 & 1.19 & 109.54 & 486 & 3.38 & 660.30 & \textbf{503} & * & \textbf{503} & * & \textbf{400.56} \\ 
  u724 & 439 & 439 & \textbf{439} & * & \textbf{439} & 1077.10 & 429 & 2.28 & 27.77 & 434 & 1.14 & 4157.30 & \textbf{439} & * & \textbf{439} & * & \textbf{188.61} \\ 
  rat783 & 438 & 438 & \textbf{438} & * & \textbf{438} & 594.30 & 422 & 3.65 & 34.59 & 428 & 2.28 & 2962.52 & \textbf{438} & * & \textbf{438} & * & \textbf{514.68} \\ 
  dsj1000 & 656 & 656 & . & . & . & . & 632 & 3.66 & 81.20 & 630 & 3.96 & 17284.30 & \textbf{656} & * & \textbf{656} & * & \textbf{3828.50} \\ 
  pr1002 & 606 & 606 & 604 & 0.33 & 608 & 18000.00 & 572 & 5.61 & 45.92 & 581 & 4.13 & 18000.00 & \textbf{606} & * & \textbf{606} & * & \textbf{4483.81} \\ 
  u1060 & 660 & 660 & . & . & . & . & 627 & 5.00 & 90.04 & 644 & 2.42 & 18000.00 & \textbf{660} & * & \textbf{660} & * & \textbf{16716.01} \\ 
  vm1084 & 777 & 777 & \textbf{777} & * & \textbf{777} & \textbf{4927.40} & 770 & 0.90 & 56.29 & 765 & 1.54 & 18000.00 & \textbf{777} & * & \textbf{777} & * & 5012.60 \\ 
  pcb1173 & 675 & 675 & . & . & . & . & 633 & 6.22 & 60.65 & 652 & 3.41 & 18000.00 & \textbf{675} & * & \textbf{675} & * & \textbf{6819.83} \\ 
  d1291 & 715 & 715 & . & . & . & . & 646 & 9.65 & 434.87 & 699 & 2.24 & 18000.00 & \textbf{715} & * & \textbf{715} & * & \textbf{7916.85} \\ 
  rl1304 & 802 & 802 & . & . & . & . & 766 & 4.49 & 102.45 & 788 & 1.75 & 18000.00 & \textbf{802} & * & \textbf{802} & * & \textbf{6269.39} \\ 
  rl1323 & 814 & 814 & 811 & 0.37 & 846 & 18000.00 & 782 & 3.93 & 89.68 & 785 & 3.56 & 14585.10 & \textbf{814} & * & \textbf{814} & * & \textbf{7740.17} \\ 
  nrw1379 & 815 & 817 & . & . & . & . & 771 & 5.40 & 106.97 & 790 & 3.07 & 18000.00 & \textbf{815} & * & \textbf{817} & 0.24 & \textbf{18000.00} \\ 
  fl1400 & 1048 & 1084 & 909 & 13.26 & 1230 & \textbf{18000.00} & 1043 & 0.48 & 518.25 & \textbf{1048} & * & 18000.00 & 1003 & 4.29 & \textbf{1084} & 7.47 & \textbf{18000.00} \\ 
  u1432 & 754 & 764 & . & . & . & . & 738 & 2.12 & 121.46 & 749 & 0.66 & 14573.50 & \textbf{754} & * & \textbf{764} & 1.31 & \textbf{18000.00} \\ 
  fl1577 & 897 & 900 & . & . & . & . & 880 & 1.90 & 286.47 & 748 & 16.61 & 18000.00 & \textbf{897} & * & \textbf{900} & 0.33 & \textbf{18000.00} \\ 
  d1655 & 922 & 924 & . & . & . & . & 846 & 8.24 & 757.70 & 890 & 3.47 & 18000.00 & \textbf{922} & * & \textbf{924} & 0.22 & \textbf{18000.00} \\ 
  vm1748 & 1276 & 1282 & 873 & 31.58 & . & \textbf{18000.00} & 1246 & 2.35 & 178.50 & 1252 & 1.88 & 16959.80 & \textbf{1276} & * & \textbf{1282} & 0.47 & \textbf{18000.00} \\ 
  u1817 & 983 & 983 & . & . & . & . & 879 & 10.58 & 975.58 & 947 & 3.66 & 18000.00 & \textbf{983} & * & \textbf{983} & * & \textbf{11226.88} \\ 
  rl1889 & 1226 & 1226 & 890 & 27.41 & 1296 & 18000.00 & 1167 & 4.81 & 269.81 & 1156 & 5.71 & 18000.00 & \textbf{1226} & * & \textbf{1226} & * & \textbf{17010.43} \\ 
  d2103 & 1200 & 1200 & . & . & . & . & 1069 & 10.92 & 951.27 & 1171 & 2.42 & 18000.00 & \textbf{1200} & * & \textbf{1200} & * & \textbf{15855.62} \\ 
  u2152 & 1151 & 1151 & . & . & . & . & 1048 & 8.95 & 1350.23 & 1111 & 3.48 & 18000.00 & \textbf{1151} & * & \textbf{1151} & * & \textbf{14703.25} \\ 
  u2319 & 1170 & 1171 & . & . & . & . & 1167 & 0.26 & 423.26 & \textbf{1170} & * & 6088.42 & \textbf{1170} & * & \textbf{1171} & 0.09 & \textbf{18000.00} \\ 
  pr2392 & 1316 & 1415 & 1140 & 13.37 & . & \textbf{18000.00} & 1292 & 1.82 & 402.29 & 1294 & 1.67 & 18000.00 & \textbf{1316} & * & \textbf{1415} & 7.00 & \textbf{18000.00} \\ 
  pcb3038 & 1727 & 1730 & . & . & . & . & 1572 & 8.98 & 681.94 & 1626 & 5.85 & 18000.00 & \textbf{1727} & * & \textbf{1730} & 0.17 & \textbf{18000.00} \\ 
  fl3795 & 1965 & 2249 & . & . & . & . & 1815 & 7.63 & 2994.90 & 1818 & 7.48 & 18000.00 & \textbf{1965} & * & \textbf{2249} & 12.63 & \textbf{18000.00} \\ 
  fnl4461 & 2541 & 2570 & . & . & . & . & 2350 & 7.52 & 2462.65 & 2342 & 7.83 & 18000.00 & \textbf{2541} & * & \textbf{2570} & 1.13 & \textbf{18000.00} \\ 
  rl5915 & 3593 & 3786 & . & . & . & . & 3358 & 6.54 & 5361.54 & 3328 & 7.38 & 18000.00 & \textbf{3593} & * & \textbf{3786} & 5.10 & \textbf{18000.00} \\ 
  rl5934 & 3632 & 3752 & . & . & . & . & 3145 & 13.41 & 5382.25 & 3276 & 9.80 & 18000.00 & \textbf{3632} & * & \textbf{3752} & 3.20 & \textbf{18000.00} \\ 
  pla7397 & 5289 & 5657 & . & . & . & . & 5141 & 2.80 & 15981.78 & 5140 & 2.82 & 18000.00 & \textbf{5289} & * & \textbf{5657} & 6.51 & \textbf{18000.00} \\ 
   \midrule\ average &  &  &  & 4.35 &  & \textbf{7433.41} &  & 4.32 & 990.82 &  & 3.12 & 11802.68 &  & 0.11 &  & 1.49 & 10387.02 \\ 
   \bottomrule\end{tabular}

    \end{scriptsize}
  \end{table}

  \begin{table}[p!]
    \centering
    \caption{Generation 2, $n\le 400$}
    \label{tab:best-II-2-medium}
    \begin{scriptsize}
\begin{tabular}{rrrrrrrrrrrrrrrrrr}
  \toprule \multicolumn{1}{c}{} &\multicolumn{2}{c}{Best} &
                        \multicolumn{4}{c}{FST} &
                        \multicolumn{3}{c}{EA4OP} &
                        \multicolumn{3}{c}{ALNS} &
                        \multicolumn{5}{c}{RB\&C} \\
                       \cmidrule(lr){2-3}\cmidrule(lr){4-7}\cmidrule(lr){8-10}\cmidrule(lr){11-13}\cmidrule(lr){14-18}\   Instance & LB & UB
                          & LB & GGap & UB & Time
                          & LB & GGap & Time
                          & LB & GGap & Time
                          & LB & GGap & UB & OGap & Time\\
                          \midrule att48 & 1717 & 1717 & \textbf{1717} & * & \textbf{1717} & \textbf{0.00} & \textbf{1717} & * & 0.32 & \textbf{1717} & * & 6.77 & \textbf{1717} & * & \textbf{1717} & * & \textbf{0.04} \\ 
  gr48 & 1761 & 1761 & \textbf{1761} & * & \textbf{1761} & \textbf{0.20} & 1749 & 0.68 & 0.20 & \textbf{1761} & * & 7.87 & \textbf{1761} & * & \textbf{1761} & * & 1.32 \\ 
  hk48 & 1614 & 1614 & \textbf{1614} & * & \textbf{1614} & \textbf{0.10} & \textbf{1614} & * & 0.15 & \textbf{1614} & * & 7.19 & \textbf{1614} & * & \textbf{1614} & * & \textbf{0.10} \\ 
  eil51 & 1674 & 1674 & \textbf{1674} & * & \textbf{1674} & \textbf{0.40} & 1668 & 0.36 & 0.18 & \textbf{1674} & * & 10.13 & \textbf{1674} & * & \textbf{1674} & * & 0.96 \\ 
  berlin52 & 1897 & 1897 & \textbf{1897} & * & \textbf{1897} & 93.40 & \textbf{1897} & * & 0.35 & \textbf{1897} & * & 10.74 & \textbf{1897} & * & \textbf{1897} & * & \textbf{3.23} \\ 
  brazil58 & 2220 & 2220 & \textbf{2220} & * & \textbf{2220} & \textbf{0.10} & 2218 & 0.09 & 1.52 & \textbf{2220} & * & 12.32 & \textbf{2220} & * & \textbf{2220} & * & 0.46 \\ 
  st70 & 2286 & 2286 & \textbf{2286} & * & \textbf{2286} & 19.40 & 2285 & 0.04 & 0.31 & \textbf{2286} & * & 21.65 & \textbf{2286} & * & \textbf{2286} & * & \textbf{1.77} \\ 
  eil76 & 2550 & 2550 & \textbf{2550} & * & \textbf{2550} & \textbf{0.10} & \textbf{2550} & * & 0.43 & \textbf{2550} & * & 16.06 & \textbf{2550} & * & \textbf{2550} & * & 0.62 \\ 
  pr76 & 2708 & 2708 & \textbf{2708} & * & \textbf{2708} & \textbf{0.40} & \textbf{2708} & * & 0.48 & \textbf{2708} & * & 19.48 & \textbf{2708} & * & \textbf{2708} & * & 1.46 \\ 
  gr96 & 3396 & 3396 & \textbf{3396} & * & \textbf{3396} & \textbf{1.70} & 3394 & 0.06 & 1.44 & 3394 & 0.06 & 31.98 & \textbf{3396} & * & \textbf{3396} & * & 9.50 \\ 
  rat99 & 2944 & 2944 & \textbf{2944} & * & \textbf{2944} & \textbf{0.90} & \textbf{2944} & * & 0.49 & \textbf{2944} & * & 32.08 & \textbf{2944} & * & \textbf{2944} & * & 3.25 \\ 
  kroA100 & 3212 & 3212 & \textbf{3212} & * & \textbf{3212} & 0.90 & \textbf{3212} & * & 0.57 & \textbf{3212} & * & 32.85 & \textbf{3212} & * & \textbf{3212} & * & \textbf{0.70} \\ 
  kroB100 & 3241 & 3241 & \textbf{3241} & * & \textbf{3241} & \textbf{6.70} & 3238 & 0.09 & 0.52 & 3239 & 0.06 & 48.39 & \textbf{3241} & * & \textbf{3241} & * & 13.28 \\ 
  kroC100 & 2947 & 2947 & \textbf{2947} & * & \textbf{2947} & 85.60 & 2931 & 0.54 & 0.60 & \textbf{2947} & * & 39.27 & \textbf{2947} & * & \textbf{2947} & * & \textbf{2.22} \\ 
  kroD100 & 3307 & 3307 & \textbf{3307} & * & \textbf{3307} & 45.00 & \textbf{3307} & * & 0.65 & \textbf{3307} & * & 30.52 & \textbf{3307} & * & \textbf{3307} & * & \textbf{3.62} \\ 
  kroE100 & 3090 & 3090 & \textbf{3090} & * & \textbf{3090} & 230.10 & 3082 & 0.26 & 0.50 & \textbf{3090} & * & 39.57 & \textbf{3090} & * & \textbf{3090} & * & \textbf{11.31} \\ 
  rd100 & 3359 & 3359 & \textbf{3359} & * & \textbf{3359} & \textbf{0.20} & \textbf{3359} & * & 0.50 & \textbf{3359} & * & 30.80 & \textbf{3359} & * & \textbf{3359} & * & 0.36 \\ 
  eil101 & 3655 & 3655 & \textbf{3655} & * & \textbf{3655} & 153.00 & \textbf{3655} & * & 0.82 & \textbf{3655} & * & 26.19 & \textbf{3655} & * & \textbf{3655} & * & \textbf{4.15} \\ 
  lin105 & 3544 & 3544 & \textbf{3544} & * & \textbf{3544} & 67.30 & 3530 & 0.40 & 1.10 & \textbf{3544} & * & 36.22 & \textbf{3544} & * & \textbf{3544} & * & \textbf{2.51} \\ 
  pr107 & 2667 & 2667 & \textbf{2667} & * & \textbf{2667} & 0.60 & \textbf{2667} & * & 1.05 & \textbf{2667} & * & 69.67 & \textbf{2667} & * & \textbf{2667} & * & \textbf{0.20} \\ 
  gr120 & 4371 & 4371 & \textbf{4371} & * & \textbf{4371} & 35.80 & 4356 & 0.34 & 1.37 & \textbf{4371} & * & 40.41 & \textbf{4371} & * & \textbf{4371} & * & \textbf{6.57} \\ 
  pr124 & 3917 & 3917 & \textbf{3917} & * & \textbf{3917} & \textbf{0.50} & 3899 & 0.46 & 1.34 & \textbf{3917} & * & 55.25 & \textbf{3917} & * & \textbf{3917} & * & 1.07 \\ 
  bier127 & 5383 & 5383 & \textbf{5383} & * & \textbf{5383} & 58.80 & 5381 & 0.04 & 1.71 & 5366 & 0.32 & 23.01 & \textbf{5383} & * & \textbf{5383} & * & \textbf{0.96} \\ 
  pr136 & 4309 & 4309 & \textbf{4309} & * & \textbf{4309} & 2.10 & \textbf{4309} & * & 1.15 & \textbf{4309} & * & 35.63 & \textbf{4309} & * & \textbf{4309} & * & \textbf{1.25} \\ 
  gr137 & 4286 & 4286 & \textbf{4286} & * & \textbf{4286} & 196.90 & 4099 & 4.36 & 3.09 & \textbf{4286} & * & 639.80 & \textbf{4286} & * & \textbf{4286} & * & \textbf{10.65} \\ 
  pr144 & 4003 & 4003 & \textbf{4003} & * & \textbf{4003} & 90.40 & 3965 & 0.95 & 3.02 & 3969 & 0.85 & 100.20 & \textbf{4003} & * & \textbf{4003} & * & \textbf{32.23} \\ 
  kroA150 & 4918 & 4918 & \textbf{4918} & * & \textbf{4918} & 241.40 & 4902 & 0.33 & 1.26 & \textbf{4918} & * & 80.06 & \textbf{4918} & * & \textbf{4918} & * & \textbf{60.43} \\ 
  kroB150 & 4869 & 4869 & \textbf{4869} & * & \textbf{4869} & 24.80 & \textbf{4869} & * & 1.19 & \textbf{4869} & * & 61.96 & \textbf{4869} & * & \textbf{4869} & * & \textbf{16.94} \\ 
  pr152 & 4279 & 4279 & \textbf{4279} & * & \textbf{4279} & 2.20 & 4245 & 0.79 & 3.47 & \textbf{4279} & * & 67.41 & \textbf{4279} & * & \textbf{4279} & * & \textbf{1.85} \\ 
  u159 & 4960 & 4960 & \textbf{4960} & * & \textbf{4960} & 192.20 & 4941 & 0.38 & 1.44 & 4950 & 0.20 & 109.59 & \textbf{4960} & * & \textbf{4960} & * & \textbf{14.96} \\ 
  rat195 & 5791 & 5791 & \textbf{5791} & * & \textbf{5791} & 128.80 & 5703 & 1.52 & 1.55 & 5782 & 0.16 & 263.23 & \textbf{5791} & * & \textbf{5791} & * & \textbf{46.09} \\ 
  d198 & 6670 & 6670 & \textbf{6670} & * & \textbf{6670} & \textbf{74.20} & 6660 & 0.15 & 7.33 & 6661 & 0.13 & 88.47 & \textbf{6670} & * & \textbf{6670} & * & 298.24 \\ 
  kroA200 & 6547 & 6547 & \textbf{6547} & * & \textbf{6547} & 68.70 & 6534 & 0.20 & 1.71 & \textbf{6547} & * & 116.11 & \textbf{6547} & * & \textbf{6547} & * & \textbf{16.18} \\ 
  kroB200 & 6419 & 6419 & \textbf{6419} & * & \textbf{6419} & 34.70 & 6278 & 2.20 & 1.97 & 6413 & 0.09 & 189.98 & \textbf{6419} & * & \textbf{6419} & * & \textbf{20.62} \\ 
  gr202 & 7789 & 7789 & \textbf{7789} & * & \textbf{7789} & \textbf{85.70} & \textbf{7789} & * & 8.77 & 7719 & 0.90 & 188.27 & \textbf{7789} & * & \textbf{7789} & * & 139.90 \\ 
  ts225 & 6834 & 6834 & \textbf{6834} & * & \textbf{6834} & \textbf{6.60} & 6819 & 0.22 & 1.47 & 6782 & 0.76 & 394.00 & \textbf{6834} & * & \textbf{6834} & * & 95.22 \\ 
  tsp225 & 6987 & 6987 & \textbf{6987} & * & \textbf{6987} & 174.50 & 6936 & 0.73 & 1.87 & 6980 & 0.10 & 299.73 & \textbf{6987} & * & \textbf{6987} & * & \textbf{54.09} \\ 
  pr226 & 6662 & 6662 & \textbf{6662} & * & \textbf{6662} & \textbf{74.10} & 6658 & 0.06 & 7.29 & \textbf{6662} & * & 201.68 & \textbf{6662} & * & \textbf{6662} & * & 2894.81 \\ 
  gr229 & 9177 & 9177 & \textbf{9177} & * & \textbf{9177} & 182.60 & 9174 & 0.03 & 13.19 & \textbf{9177} & * & 1379.35 & \textbf{9177} & * & \textbf{9177} & * & \textbf{16.67} \\ 
  gil262 & 8321 & 8321 & \textbf{8321} & * & \textbf{8321} & 89.60 & 8175 & 1.75 & 3.47 & 8269 & 0.62 & 487.41 & \textbf{8321} & * & \textbf{8321} & * & \textbf{64.63} \\ 
  pr264 & 6654 & 6654 & \textbf{6654} & * & \textbf{6654} & 23.00 & 6173 & 7.23 & 5.94 & \textbf{6654} & * & 314.27 & \textbf{6654} & * & \textbf{6654} & * & \textbf{13.33} \\ 
  a280 & 8428 & 8428 & \textbf{8428} & * & \textbf{8428} & \textbf{103.80} & 8304 & 1.47 & 2.85 & 8404 & 0.28 & 215.31 & \textbf{8428} & * & \textbf{8428} & * & 519.95 \\ 
  pr299 & 9182 & 9182 & \textbf{9182} & * & \textbf{9182} & \textbf{426.50} & 9112 & 0.76 & 3.23 & 9147 & 0.38 & 393.12 & \textbf{9182} & * & \textbf{9182} & * & 623.34 \\ 
  lin318 & 10923 & 10923 & \textbf{10923} & * & \textbf{10923} & 862.40 & 10866 & 0.52 & 8.29 & 10801 & 1.12 & 370.64 & \textbf{10923} & * & \textbf{10923} & * & \textbf{367.53} \\ 
  rd400 & 13652 & 13652 & \textbf{13652} & * & \textbf{13652} & \textbf{293.50} & 13442 & 1.54 & 6.80 & 13562 & 0.66 & 1174.91 & \textbf{13652} & * & \textbf{13652} & * & 769.66 \\ 
   \midrule\ average &  &  &  & * &  & \textbf{92.89} &  & 0.63 & 2.38 &  & 0.15 & 173.77 &  & * &  & * & 136.63 \\ 
   \bottomrule\end{tabular}

    \end{scriptsize}
  \end{table}

  \begin{table}[p!]
    \centering
    \caption{Generation 2, $n>400$}
    \label{tab:best-II-2-big}
    \begin{scriptsize}
\begin{tabular}{rrrrrrrrrrrrrrrrrr}
  \toprule \multicolumn{1}{c}{} &\multicolumn{2}{c}{Best} &
                        \multicolumn{4}{c}{FST} &
                        \multicolumn{3}{c}{EA4OP} &
                        \multicolumn{3}{c}{ALNS} &
                        \multicolumn{5}{c}{RB\&C} \\
                       \cmidrule(lr){2-3}\cmidrule(lr){4-7}\cmidrule(lr){8-10}\cmidrule(lr){11-13}\cmidrule(lr){14-18}\   Instance & LB & UB
                          & LB & GGap & UB & Time
                          & LB & GGap & Time
                          & LB & GGap & Time
                          & LB & GGap & UB & OGap & Time\\
                          \midrule fl417 & 11933 & 12294 & 11894 & 0.33 & \textbf{12294} & \textbf{18000.00} & 11787 & 1.22 & 16.73 & 11923 & 0.08 & 2144.94 & \textbf{11933} & * & 12387 & 3.67 & \textbf{18000.00} \\ 
  gr431 & 18318 & 18318 & \textbf{18318} & * & \textbf{18318} & \textbf{969.50} & 18287 & 0.17 & 51.38 & \textbf{18318} & * & 2740.82 & \textbf{18318} & * & \textbf{18318} & * & 2809.41 \\ 
  pr439 & 16171 & 16171 & \textbf{16171} & * & \textbf{16171} & \textbf{1298.30} & 16085 & 0.53 & 11.77 & 16128 & 0.27 & 629.44 & \textbf{16171} & * & \textbf{16171} & * & 3765.86 \\ 
  pcb442 & 14484 & 14484 & \textbf{14484} & * & \textbf{14484} & \textbf{6259.10} & 14273 & 1.46 & 6.83 & 14411 & 0.50 & 4410.74 & \textbf{14484} & * & \textbf{14484} & * & 13760.94 \\ 
  d493 & 16995 & 17007 & . & . & . & . & 16729 & 1.57 & 17.15 & 16820 & 1.03 & 6231.42 & \textbf{16995} & * & \textbf{17007} & 0.07 & \textbf{18000.00} \\ 
  att532 & 19635 & 19800 & 19598 & 0.19 & \textbf{19800} & \textbf{18000.00} & 19265 & 1.88 & 23.43 & 19465 & 0.87 & 1564.89 & \textbf{19635} & * & \textbf{19800} & 0.83 & \textbf{18000.00} \\ 
  ali535 & 21954 & 21954 & \textbf{21954} & * & \textbf{21954} & \textbf{2099.70} & 21910 & 0.20 & 95.05 & 21761 & 0.88 & 1537.87 & \textbf{21954} & * & 21973 & 0.09 & 18000.00 \\ 
  pa561 & 19576 & 19576 & \textbf{19576} & * & \textbf{19576} & \textbf{1487.10} & 18894 & 3.48 & 23.45 & 19092 & 2.47 & 790.31 & \textbf{19576} & * & \textbf{19576} & * & 1961.95 \\ 
  u574 & 19351 & 19351 & \textbf{19351} & * & \textbf{19351} & \textbf{612.50} & 18966 & 1.99 & 16.33 & 19028 & 1.67 & 5389.10 & \textbf{19351} & * & \textbf{19351} & * & 1026.82 \\ 
  rat575 & 18251 & 18251 & \textbf{18251} & * & \textbf{18251} & \textbf{931.10} & 17705 & 2.99 & 14.97 & 17984 & 1.46 & 2089.02 & \textbf{18251} & * & \textbf{18251} & * & 9616.70 \\ 
  p654 & 17900 & 21566 & 17160 & 4.13 & \textbf{21566} & \textbf{18000.00} & 17821 & 0.44 & 42.82 & \textbf{17900} & * & 18000.00 & 17753 & 0.82 & 22248 & 20.20 & \textbf{18000.00} \\ 
  d657 & 21503 & 21503 & \textbf{21503} & * & \textbf{21503} & 2682.40 & 21162 & 1.59 & 22.90 & 21231 & 1.26 & 4161.44 & \textbf{21503} & * & \textbf{21503} & * & \textbf{554.67} \\ 
  gr666 & 26514 & 26569 & . & . & . & . & 26336 & 0.67 & 136.48 & 25971 & 2.05 & 1024.22 & \textbf{26514} & * & \textbf{26569} & 0.21 & \textbf{18000.00} \\ 
  u724 & 24223 & 24223 & \textbf{24223} & * & \textbf{24223} & \textbf{5830.50} & 23793 & 1.78 & 28.71 & 23878 & 1.42 & 5755.06 & \textbf{24223} & * & \textbf{24223} & * & 9829.42 \\ 
  rat783 & 25474 & 25474 & . & . & . & . & 24861 & 2.41 & 32.36 & 24987 & 1.91 & 6622.62 & \textbf{25474} & * & \textbf{25474} & * & \textbf{12246.90} \\ 
  dsj1000 & 35835 & 35915 & 35772 & 0.18 & 35917 & \textbf{18000.00} & 34463 & 3.83 & 83.34 & 34641 & 3.33 & 18000.00 & \textbf{35835} & * & \textbf{35915} & 0.22 & \textbf{18000.00} \\ 
  pr1002 & 33030 & 33092 & 27066 & 18.06 & . & \textbf{18000.00} & 31746 & 3.89 & 46.19 & 32120 & 2.76 & 18000.00 & \textbf{33030} & * & \textbf{33092} & 0.19 & \textbf{18000.00} \\ 
  u1060 & 36151 & 36291 & . & . & . & . & 35110 & 2.88 & 77.78 & 35284 & 2.40 & 18000.00 & \textbf{36151} & * & \textbf{36291} & 0.39 & \textbf{18000.00} \\ 
  vm1084 & 40777 & 40952 & 40687 & 0.22 & 40954 & \textbf{18000.00} & 40308 & 1.15 & 55.67 & 40240 & 1.32 & 18000.00 & \textbf{40777} & * & \textbf{40952} & 0.43 & \textbf{18000.00} \\ 
  pcb1173 & 37035 & 37100 & . & . & . & . & 35826 & 3.26 & 69.94 & 35946 & 2.94 & 18000.00 & \textbf{37035} & * & \textbf{37100} & 0.18 & \textbf{18000.00} \\ 
  d1291 & 37778 & 37854 & . & . & . & . & 35153 & 6.95 & 289.25 & 36815 & 2.55 & 18000.00 & \textbf{37778} & * & \textbf{37854} & 0.20 & \textbf{18000.00} \\ 
  rl1304 & 42275 & 42359 & . & . & . & . & 40561 & 4.05 & 97.68 & 40893 & 3.27 & 12853.40 & \textbf{42275} & * & \textbf{42359} & 0.20 & \textbf{18000.00} \\ 
  rl1323 & 43377 & 43450 & 43347 & 0.07 & \textbf{43450} & \textbf{18000.00} & 41459 & 4.42 & 89.78 & 41210 & 5.00 & 18000.00 & \textbf{43377} & * & \textbf{43450} & 0.17 & \textbf{18000.00} \\ 
  nrw1379 & 46676 & 46787 & . & . & . & . & 45602 & 2.30 & 117.51 & 45576 & 2.36 & 18000.00 & \textbf{46676} & * & \textbf{46787} & 0.24 & \textbf{18000.00} \\ 
  fl1400 & 56692 & 64298 & 53222 & 6.12 & 64726 & \textbf{18000.00} & 56258 & 0.77 & 794.15 & \textbf{56692} & * & 18000.00 & 54124 & 4.53 & \textbf{64298} & 15.82 & \textbf{18000.00} \\ 
  u1432 & 46946 & 47018 & . & . & . & . & 44810 & 4.55 & 100.91 & 44982 & 4.18 & 18000.00 & \textbf{46946} & * & \textbf{47018} & 0.15 & \textbf{18000.00} \\ 
  fl1577 & 45505 & 50154 & . & . & . & . & \textbf{45505} & * & 334.28 & 41148 & 9.57 & 18000.00 & 45326 & 0.39 & \textbf{50154} & 9.63 & \textbf{18000.00} \\ 
  d1655 & 49319 & 53083 & . & . & . & . & 47211 & 4.27 & 683.17 & \textbf{49319} & * & 18000.00 & 46158 & 6.41 & \textbf{53083} & 13.05 & \textbf{18000.00} \\ 
  vm1748 & 68042 & 68303 & . & . & . & . & 66685 & 1.99 & 195.85 & 66636 & 2.07 & 18000.00 & \textbf{68042} & * & \textbf{68303} & 0.38 & \textbf{18000.00} \\ 
  u1817 & 54245 & 54554 & . & . & . & . & 50366 & 7.15 & 734.39 & 51676 & 4.74 & 18000.00 & \textbf{54245} & * & \textbf{54554} & 0.57 & \textbf{18000.00} \\ 
  rl1889 & 63308 & 64425 & 52047 & 17.79 & . & \textbf{18000.00} & 60084 & 5.09 & 286.07 & 60928 & 3.76 & 18000.00 & \textbf{63308} & * & \textbf{64425} & 1.73 & \textbf{18000.00} \\ 
  d2103 & 63426 & 63426 & . & . & . & . & 57202 & 9.81 & 682.28 & 61636 & 2.82 & 18000.00 & \textbf{63426} & * & \textbf{63426} & * & \textbf{16593.51} \\ 
  u2152 & 64649 & 64775 & 53976 & 16.51 & . & \textbf{18000.00} & 60211 & 6.86 & 1164.38 & 61052 & 5.56 & 18000.00 & \textbf{64649} & * & \textbf{64775} & 0.19 & \textbf{18000.00} \\ 
  u2319 & 80914 & 81139 & 72790 & 10.04 & . & \textbf{18000.00} & 78102 & 3.48 & 447.06 & 77610 & 4.08 & 18000.00 & \textbf{80914} & * & \textbf{81139} & 0.28 & \textbf{18000.00} \\ 
  pr2392 & 72843 & 78237 & 64577 & 11.35 & . & \textbf{18000.00} & 71018 & 2.51 & 440.57 & 71851 & 1.36 & 18000.00 & \textbf{72843} & * & \textbf{78237} & 6.89 & \textbf{18000.00} \\ 
  pcb3038 & 97902 & 97995 & 83951 & 14.25 & . & \textbf{18000.00} & 91842 & 6.19 & 820.37 & 91457 & 6.58 & 18000.00 & \textbf{97902} & * & \textbf{97995} & 0.09 & \textbf{18000.00} \\ 
  fl3795 & 103397 & 142895 & . & . & . & . & \textbf{103397} & * & 4788.96 & 102642 & 0.73 & 18000.00 & 98998 & 4.25 & \textbf{142895} & 30.72 & \textbf{18000.00} \\ 
  fnl4461 & 147109 & 150189 & . & . & . & . & 140424 & 4.54 & 2618.15 & 135515 & 7.88 & 18000.00 & \textbf{147109} & * & \textbf{150189} & 2.05 & \textbf{18000.00} \\ 
  rl5915 & 184424 & 197729 & . & . & . & . & 176678 & 4.20 & 5512.40 & 173500 & 5.92 & 18000.00 & \textbf{184424} & * & \textbf{197729} & 6.73 & \textbf{18000.00} \\ 
  rl5934 & 187034 & 196805 & . & . & . & . & 171649 & 8.23 & 5757.80 & 166368 & 11.05 & 18000.00 & \textbf{187034} & * & \textbf{196805} & 4.96 & \textbf{18000.00} \\ 
  pla7397 & 281977 & 297246 & . & . & . & . & 272452 & 3.38 & 18000.00 & 266038 & 5.65 & 18000.00 & \textbf{281977} & * & \textbf{297246} & 5.14 & \textbf{18000.00} \\ 
   \midrule\ average &  &  &  & 4.51 &  & \textbf{11644.10} &  & 3.13 & 1093.37 &  & 2.87 & 12827.93 &  & 0.40 &  & 3.06 & 15369.91 \\ 
   \bottomrule\end{tabular}

    \end{scriptsize}
  \end{table}

  \begin{table}[p!]
    \centering
    \caption{Generation 3, $n\le 400$}
    \label{tab:best-II-3-medium}
    \begin{scriptsize}

\begin{tabular}{rrrrrrrrrrrrrrrrrr}
  \toprule \multicolumn{1}{c}{} &\multicolumn{2}{c}{Best} &
                        \multicolumn{4}{c}{FST} &
                        \multicolumn{3}{c}{EA4OP} &
                        \multicolumn{3}{c}{ALNS} &
                        \multicolumn{5}{c}{RB\&C} \\
                       \cmidrule(lr){2-3}\cmidrule(lr){4-7}\cmidrule(lr){8-10}\cmidrule(lr){11-13}\cmidrule(lr){14-18}\   Instance & LB & UB
                          & LB & GGap & UB & Time
                          & LB & GGap & Time
                          & LB & GGap & Time
                          & LB & GGap & UB & OGap & Time\\
                          \midrule att48 & 1049 & 1049 & \textbf{1049} & * & \textbf{1049} & 38.50 & \textbf{1049} & * & 0.259 & \textbf{1049} & * & 7.18 & \textbf{1049} & * & \textbf{1049} & * & \textbf{1.17} \\ 
  gr48 & 1480 & 1480 & \textbf{1480} & * & \textbf{1480} & \textbf{0.20} & \textbf{1480} & * & 0.13 & \textbf{1480} & * & 8.87 & \textbf{1480} & * & \textbf{1480} & * & 0.72 \\ 
  hk48 & 1764 & 1764 & \textbf{1764} & * & \textbf{1764} & \textbf{0.00} & \textbf{1764} & * & 0.215 & \textbf{1764} & * & 8.51 & \textbf{1764} & * & \textbf{1764} & * & 0.06 \\ 
  eil51 & 1399 & 1399 & \textbf{1399} & * & \textbf{1399} & \textbf{0.20} & 1398 & 0.07 & 0.222 & \textbf{1399} & * & 6.87 & \textbf{1399} & * & \textbf{1399} & * & 1.46 \\ 
  berlin52 & 1036 & 1036 & \textbf{1036} & * & \textbf{1036} & 124.70 & 1034 & 0.19 & 0.637 & \textbf{1036} & * & 12.84 & \textbf{1036} & * & \textbf{1036} & * & \textbf{4.61} \\ 
  brazil58 & 1702 & 1702 & \textbf{1702} & * & \textbf{1702} & \textbf{0.00} & \textbf{1702} & * & 0.711 & \textbf{1702} & * & 11.09 & \textbf{1702} & * & \textbf{1702} & * & \textbf{0.02} \\ 
  st70 & 2108 & 2108 & \textbf{2108} & * & \textbf{2108} & \textbf{0.40} & \textbf{2108} & * & 0.308 & \textbf{2108} & * & 9.65 & \textbf{2108} & * & \textbf{2108} & * & 0.49 \\ 
  eil76 & 2467 & 2467 & \textbf{2467} & * & \textbf{2467} & \textbf{0.40} & \textbf{2467} & * & 0.362 & \textbf{2467} & * & 20.48 & \textbf{2467} & * & \textbf{2467} & * & 2.96 \\ 
  pr76 & 2430 & 2430 & \textbf{2430} & * & \textbf{2430} & \textbf{0.20} & \textbf{2430} & * & 0.568 & \textbf{2430} & * & 20.43 & \textbf{2430} & * & \textbf{2430} & * & 1.07 \\ 
  gr96 & 3170 & 3170 & \textbf{3170} & * & \textbf{3170} & 61.50 & 3166 & 0.13 & 1.408 & 3166 & 0.13 & 15.22 & \textbf{3170} & * & \textbf{3170} & * & \textbf{5.66} \\ 
  rat99 & 2908 & 2908 & \textbf{2908} & * & \textbf{2908} & 4.90 & - & - & - & - & - & - & \textbf{2908} & * & \textbf{2908} & * & \textbf{3.01} \\ 
  kroA100 & 3211 & 3211 & \textbf{3211} & * & \textbf{3211} & 63.30 & 3180 & 0.97 & 0.379 & \textbf{3211} & * & 32.31 & \textbf{3211} & * & \textbf{3211} & * & \textbf{1.81} \\ 
  kroB100 & 2804 & 2804 & \textbf{2804} & * & \textbf{2804} & 0.60 & 2785 & 0.68 & 0.51 & \textbf{2804} & * & 35.83 & \textbf{2804} & * & \textbf{2804} & * & \textbf{0.35} \\ 
  kroC100 & 3155 & 3155 & \textbf{3155} & * & \textbf{3155} & \textbf{1.50} & \textbf{3155} & * & 0.439 & \textbf{3155} & * & 34.67 & \textbf{3155} & * & \textbf{3155} & * & 1.82 \\ 
  kroD100 & 3167 & 3167 & \textbf{3167} & * & \textbf{3167} & 10.70 & 3141 & 0.82 & 0.58 & \textbf{3167} & * & 31.08 & \textbf{3167} & * & \textbf{3167} & * & \textbf{0.70} \\ 
  kroE100 & 3049 & 3049 & \textbf{3049} & * & \textbf{3049} & 1.50 & \textbf{3049} & * & 0.471 & \textbf{3049} & * & 31.96 & \textbf{3049} & * & \textbf{3049} & * & \textbf{1.36} \\ 
  rd100 & 2926 & 2926 & \textbf{2926} & * & \textbf{2926} & 113.20 & 2923 & 0.10 & 0.482 & \textbf{2926} & * & 16.35 & \textbf{2926} & * & \textbf{2926} & * & \textbf{23.20} \\ 
  eil101 & 3345 & 3345 & \textbf{3345} & * & \textbf{3345} & 29.80 & \textbf{3345} & * & 0.564 & \textbf{3345} & * & 28.61 & \textbf{3345} & * & \textbf{3345} & * & \textbf{1.37} \\ 
  lin105 & 2986 & 2986 & \textbf{2986} & * & \textbf{2986} & 51.90 & 2973 & 0.44 & 2.094 & \textbf{2986} & * & 38.24 & \textbf{2986} & * & \textbf{2986} & * & \textbf{16.02} \\ 
  pr107 & 1877 & 1877 & \textbf{1877} & * & \textbf{1877} & \textbf{660.90} & 1802 & 4.00 & 0.816 & \textbf{1877} & * & 65.16 & \textbf{1877} & * & \textbf{1877} & * & 3297.37 \\ 
  gr120 & 3779 & 3779 & \textbf{3779} & * & \textbf{3779} & \textbf{1.50} & 3748 & 0.82 & 1.358 & 3777 & 0.05 & 37.94 & \textbf{3779} & * & \textbf{3779} & * & 2.65 \\ 
  pr124 & 3557 & 3557 & \textbf{3557} & * & \textbf{3557} & \textbf{1021.50} & 3455 & 2.87 & 0.882 & \textbf{3557} & * & 99.87 & \textbf{3557} & * & \textbf{3557} & * & 4507.38 \\ 
  bier127 & 2365 & 2365 & \textbf{2365} & * & \textbf{2365} & 79.90 & 2361 & 0.17 & 2.619 & 2361 & 0.17 & 49.9 & \textbf{2365} & * & \textbf{2365} & * & \textbf{40.07} \\ 
  pr136 & 4390 & 4390 & \textbf{4390} & * & \textbf{4390} & 86.70 & \textbf{4390} & * & 1.126 & \textbf{4390} & * & 61.84 & \textbf{4390} & * & \textbf{4390} & * & \textbf{30.50} \\ 
  gr137 & 3954 & 3954 & \textbf{3954} & * & \textbf{3954} & \textbf{8.60} & \textbf{3954} & * & 1.884 & \textbf{3954} & * & 637.09 & \textbf{3954} & * & \textbf{3954} & * & 14.01 \\ 
  pr144 & 3745 & 3745 & \textbf{3745} & * & \textbf{3745} & \textbf{112.60} & 3700 & 1.20 & 2.411 & 3744 & 0.03 & 112.92 & \textbf{3745} & * & \textbf{3745} & * & 116.68 \\ 
  kroA150 & 5039 & 5039 & \textbf{5039} & * & \textbf{5039} & 330.70 & 5019 & 0.40 & 1.07 & 5037 & 0.04 & 104.23 & \textbf{5039} & * & \textbf{5039} & * & \textbf{46.43} \\ 
  kroB150 & 5314 & 5314 & \textbf{5314} & * & \textbf{5314} & 107.60 & \textbf{5314} & * & 1.044 & \textbf{5314} & * & 63.05 & \textbf{5314} & * & \textbf{5314} & * & \textbf{28.53} \\ 
  pr152 & 3905 & 3905 & \textbf{3905} & * & \textbf{3905} & 1122.40 & 3902 & 0.08 & 3.625 & 3539 & 9.37 & 184.38 & \textbf{3905} & * & \textbf{3905} & * & \textbf{83.51} \\ 
  u159 & 5272 & 5272 & \textbf{5272} & * & \textbf{5272} & 52.20 & \textbf{5272} & * & 0.945 & \textbf{5272} & * & 94.27 & \textbf{5272} & * & \textbf{5272} & * & \textbf{8.59} \\ 
  rat195 & 6195 & 6195 & \textbf{6195} & * & \textbf{6195} & 49.90 & - & - & - & - & - & - & \textbf{6195} & * & \textbf{6195} & * & \textbf{33.56} \\ 
  d198 & 6320 & 6320 & \textbf{6320} & * & \textbf{6320} & \textbf{286.10} & 6290 & 0.47 & 7.145 & \textbf{6320} & * & 105.7 & \textbf{6320} & * & \textbf{6320} & * & 461.18 \\ 
  kroA200 & 6123 & 6123 & \textbf{6123} & * & \textbf{6123} & 122.30 & 6114 & 0.15 & 1.717 & 6118 & 0.08 & 232.2 & \textbf{6123} & * & \textbf{6123} & * & \textbf{92.41} \\ 
  kroB200 & 6266 & 6266 & \textbf{6266} & * & \textbf{6266} & 40.10 & 6213 & 0.85 & 1.775 & \textbf{6266} & * & 188.77 & \textbf{6266} & * & \textbf{6266} & * & \textbf{3.87} \\ 
  gr202 & 8616 & 8616 & \textbf{8616} & * & \textbf{8616} & \textbf{224.80} & 8605 & 0.13 & 10.452 & 8564 & 0.60 & 57.88 & \textbf{8616} & * & \textbf{8616} & * & 315.26 \\ 
  ts225 & 7575 & 7575 & \textbf{7575} & * & \textbf{7575} & 171.20 & \textbf{7575} & * & 1.136 & \textbf{7575} & * & 450.25 & \textbf{7575} & * & \textbf{7575} & * & \textbf{6.62} \\ 
  tsp225 & 7740 & 7740 & \textbf{7740} & * & \textbf{7740} & 150.30 & - & - & - & - & - & - & \textbf{7740} & * & \textbf{7740} & * & \textbf{38.61} \\ 
  pr226 & 6993 & 6993 & \textbf{6993} & * & \textbf{6993} & \textbf{32.60} & 6908 & 1.22 & 8.013 & \textbf{6993} & * & 177.59 & \textbf{6993} & * & \textbf{6993} & * & 1170.00 \\ 
  gr229 & 6328 & 6328 & \textbf{6328} & * & \textbf{6328} & \textbf{10.20} & 6297 & 0.49 & 11.655 & \textbf{6328} & * & 1298.8 & \textbf{6328} & * & \textbf{6328} & * & 42.63 \\ 
  gil262 & 9246 & 9246 & \textbf{9246} & * & \textbf{9246} & 133.40 & 9094 & 1.64 & 3.937 & 9210 & 0.39 & 649.54 & \textbf{9246} & * & \textbf{9246} & * & \textbf{83.29} \\ 
  pr264 & 8137 & 8137 & \textbf{8137} & * & \textbf{8137} & \textbf{20.70} & 8068 & 0.85 & 3.625 & \textbf{8137} & * & 357.8 & \textbf{8137} & * & \textbf{8137} & * & 186.59 \\ 
  a280 & 9774 & 9774 & \textbf{9774} & * & \textbf{9774} & 213.30 & 8684 & 11.15 & 3.22 & 8789 & 10.08 & 378.8 & \textbf{9774} & * & \textbf{9774} & * & \textbf{126.80} \\ 
  pr299 & 10343 & 10343 & \textbf{10343} & * & \textbf{10343} & \textbf{363.60} & 9959 & 3.71 & 3.952 & 10233 & 1.06 & 549.11 & \textbf{10343} & * & \textbf{10343} & * & 913.13 \\ 
  lin318 & 10368 & 10368 & \textbf{10368} & * & \textbf{10368} & 534.80 & 10273 & 0.92 & 6.327 & 10337 & 0.30 & 528.2 & \textbf{10368} & * & \textbf{10368} & * & \textbf{327.58} \\ 
  rd400 & 13223 & 13223 & \textbf{13223} & * & \textbf{13223} & 293.20 & 13088 & 1.02 & 7.738 & 13122 & 0.76 & 727.58 & \textbf{13223} & * & \textbf{13223} & * & \textbf{214.40} \\ 
   \midrule\ average &  &  &  & * &  & \textbf{149.66} &  & 0.85 & 2.35 &  & 0.55 & 180.55 &  & * &  & * & 272.43 \\ 
   \bottomrule\end{tabular}

    \end{scriptsize}
  \end{table}

  \begin{table}[p!]
    \centering
    \caption{Generation 3, $n>400$}
    \label{tab:best-II-3-big}
    \begin{scriptsize}
\begin{tabular}{rrrrrrrrrrrrrrrrrr}
  \toprule \multicolumn{1}{c}{} &\multicolumn{2}{c}{Best} &
                        \multicolumn{4}{c}{FST} &
                        \multicolumn{3}{c}{EA4OP} &
                        \multicolumn{3}{c}{ALNS} &
                        \multicolumn{5}{c}{RB\&C} \\
                       \cmidrule(lr){2-3}\cmidrule(lr){4-7}\cmidrule(lr){8-10}\cmidrule(lr){11-13}\cmidrule(lr){14-18}\   Instance & LB & UB
                          & LB & GGap & UB & Time
                          & LB & GGap & Time
                          & LB & GGap & Time
                          & LB & GGap & UB & OGap & Time\\
                          \midrule fl417 & 14220 & 14220 & \textbf{14220} & * & \textbf{14220} & \textbf{6227.60} & 14186 & 0.24 & 12.449 & \textbf{14220} & * & 1131.05 & 14219 & 0.01 & 14387 & 1.17 & 18000.00 \\ 
  gr431 & 10911 & 10911 & \textbf{10911} & * & \textbf{10911} & \textbf{1046.90} & 10817 & 0.86 & 54.504 & 10907 & 0.04 & 2411.45 & \textbf{10911} & * & \textbf{10911} & * & 7814.17 \\ 
  pr439 & 15176 & 15296 & 15160 & 0.11 & \textbf{15296} & \textbf{18000.00} & 15097 & 0.52 & 10.96 & 15080 & 0.63 & 1328.74 & \textbf{15176} & * & 15331 & 1.01 & \textbf{18000.00} \\ 
  pcb442 & 14819 & 14819 & \textbf{14819} & * & 14839 & 18000.00 & 14522 & 2.00 & 6.578 & 14695 & 0.84 & 1192.19 & \textbf{14819} & * & \textbf{14819} & * & \textbf{11574.76} \\ 
  d493 & 25167 & 25188 & \textbf{25167} & * & \textbf{25188} & \textbf{18000.00} & 24981 & 0.74 & 19.182 & 24849 & 1.26 & 3829.32 & \textbf{25167} & * & 25195 & 0.11 & \textbf{18000.00} \\ 
  att532 & 15498 & 15498 & \textbf{15498} & * & \textbf{15498} & 933.20 & 15342 & 1.01 & 22.747 & 15335 & 1.05 & 4533.36 & \textbf{15498} & * & \textbf{15498} & * & \textbf{318.44} \\ 
  ali535 & 9414 & 9472 & . & . & . & . & 9328 & 0.91 & 94.089 & 9308 & 1.13 & 13313.5 & \textbf{9414} & * & \textbf{9472} & 0.61 & \textbf{18000.00} \\ 
  pa561 & 14482 & 14482 & \textbf{14482} & * & \textbf{14482} & 10543.80 & - & - & - & - & - & - & \textbf{14482} & * & \textbf{14482} & * & \textbf{2539.41} \\ 
  u574 & 20064 & 20064 & \textbf{20064} & * & \textbf{20064} & \textbf{1409.30} & 19691 & 1.86 & 19.766 & 19841 & 1.11 & 1671.01 & \textbf{20064} & * & \textbf{20064} & * & 2693.59 \\ 
  rat575 & 20109 & 20109 & \textbf{20109} & * & \textbf{20109} & 1426.50 & - & - & - & - & - & - & \textbf{20109} & * & \textbf{20109} & * & \textbf{929.99} \\ 
  p654 & 24492 & 24518 & \textbf{24492} & * & 31914 & \textbf{18000.00} & 24130 & 1.48 & 18.541 & 24427 & 0.27 & 7543.02 & \textbf{24492} & * & \textbf{24518} & 0.11 & \textbf{18000.00} \\ 
  d657 & 24562 & 24562 & \textbf{24562} & * & \textbf{24562} & \textbf{4053.30} & 23772 & 3.22 & 21.887 & 23829 & 2.98 & 4600.87 & \textbf{24562} & * & \textbf{24562} & * & 8777.39 \\ 
  gr666 & 17023 & 17048 & 17020 & 0.02 & \textbf{17048} & \textbf{18000.00} & 16902 & 0.71 & 143.868 & 16709 & 1.84 & 2734.75 & \textbf{17023} & * & 17060 & 0.22 & \textbf{18000.00} \\ 
  u724 & 28348 & 28348 & \textbf{28348} & * & \textbf{28348} & \textbf{5870.60} & 27932 & 1.47 & 29.263 & 28033 & 1.11 & 12058.6 & \textbf{28348} & * & \textbf{28348} & * & 10332.54 \\ 
  rat783 & 27566 & 27566 & \textbf{27566} & * & \textbf{27566} & 7232.30 & - & - & - & - & - & - & \textbf{27566} & * & \textbf{27566} & * & \textbf{3812.98} \\ 
  dsj1000 & 31434 & 31454 & . & . & . & . & 30943 & 1.56 & 79.179 & 31040 & 1.25 & 15962 & \textbf{31434} & * & \textbf{31454} & 0.06 & \textbf{18000.00} \\ 
  pr1002 & 39526 & 39526 & 39449 & 0.19 & 39545 & 18000.00 & 38762 & 1.93 & 47.303 & 38502 & 2.59 & 18000 & \textbf{39526} & * & \textbf{39526} & * & \textbf{13955.69} \\ 
  u1060 & 37492 & 37569 & . & . & . & . & 36570 & 2.46 & 75.876 & 36598 & 2.38 & 18000 & \textbf{37492} & * & \textbf{37569} & 0.20 & \textbf{18000.00} \\ 
  vm1084 & 37669 & 37669 & 37653 & 0.04 & 37694 & 18000.00 & 37508 & 0.43 & 54.207 & 37178 & 1.30 & 3286.89 & \textbf{37669} & * & \textbf{37669} & * & \textbf{8710.50} \\ 
  pcb1173 & 41257 & 41257 & . & . & . & . & 40069 & 2.88 & 66.158 & 40513 & 1.80 & 18000 & \textbf{41257} & * & \textbf{41257} & * & \textbf{15133.74} \\ 
  d1291 & 41509 & 42153 & 30106 & 27.47 & . & \textbf{18000.00} & 38132 & 8.14 & 299.865 & 39919 & 3.83 & 18000 & \textbf{41509} & * & \textbf{42153} & 1.53 & \textbf{18000.00} \\ 
  rl1304 & 41881 & 42075 & 40478 & 3.35 & . & \textbf{18000.00} & 41214 & 1.59 & 81.109 & 41679 & 0.48 & 18000 & \textbf{41881} & * & \textbf{42075} & 0.46 & \textbf{18000.00} \\ 
  rl1323 & 47213 & 47384 & 44458 & 5.84 & . & \textbf{18000.00} & 46641 & 1.21 & 93.526 & 45500 & 3.63 & 8544.44 & \textbf{47213} & * & \textbf{47384} & 0.36 & \textbf{18000.00} \\ 
  nrw1379 & 42920 & 42975 & . & . & . & . & - & - & - & - & - & - & \textbf{42920} & * & \textbf{42975} & 0.13 & \textbf{18000.00} \\ 
  fl1400 & 57470 & 59491 & 54792 & 4.66 & 67053 & \textbf{18000.00} & 57226 & 0.42 & 599.811 & \textbf{57470} & * & 18000 & 54661 & 4.89 & \textbf{59491} & 8.12 & \textbf{18000.00} \\ 
  u1432 & 47778 & 47895 & . & . & . & . & 46657 & 2.35 & 138.016 & 47242 & 1.12 & 18000 & \textbf{47778} & * & \textbf{47895} & 0.24 & \textbf{18000.00} \\ 
  fl1577 & 45935 & 48809 & . & . & . & . & 45692 & 0.53 & 295.615 & \textbf{45935} & * & 18000 & 45768 & 0.36 & \textbf{48809} & 6.23 & \textbf{18000.00} \\ 
  d1655 & 62048 & 62945 & 51168 & 17.53 & . & \textbf{18000.00} & 58728 & 5.35 & 674.247 & 60956 & 1.76 & 18000 & \textbf{62048} & * & \textbf{62945} & 1.43 & \textbf{18000.00} \\ 
  vm1748 & 71885 & 72010 & 68979 & 4.04 & . & \textbf{18000.00} & 70958 & 1.29 & 225.29 & 71244 & 0.89 & 18000 & \textbf{71885} & * & \textbf{72010} & 0.17 & \textbf{18000.00} \\ 
  u1817 & 63639 & 67670 & 52186 & 18.00 & . & \textbf{18000.00} & \textbf{63639} & * & 1302.347 & 63016 & 0.98 & 18000 & 63618 & 0.03 & \textbf{67670} & 5.99 & \textbf{18000.00} \\ 
  rl1889 & 70065 & 71106 & 43374 & 38.09 & . & \textbf{18000.00} & 68422 & 2.34 & 244.973 & 68096 & 2.81 & 18000 & \textbf{70065} & * & \textbf{71106} & 1.46 & \textbf{18000.00} \\ 
  d2103 & 82787 & 82973 & 76035 & 8.16 & . & \textbf{18000.00} & 77333 & 6.59 & 1168.899 & 81081 & 2.06 & 18000 & \textbf{82787} & * & \textbf{82973} & 0.22 & \textbf{18000.00} \\ 
  u2152 & 74007 & 78066 & 52091 & 29.61 & . & \textbf{18000.00} & 73400 & 0.82 & 1619.609 & 72733 & 1.72 & 18000 & \textbf{74007} & * & \textbf{78066} & 5.20 & \textbf{18000.00} \\ 
  u2319 & 79351 & 81050 & \textbf{79351} & * & 81619 & \textbf{18000.00} & 78113 & 1.56 & 569.758 & 79130 & 0.28 & 18000 & 79343 & 0.01 & \textbf{81050} & 2.11 & \textbf{18000.00} \\ 
  pr2392 & 85409 & 90261 & 60225 & 29.49 & . & \textbf{18000.00} & 84094 & 1.54 & 422.734 & 85084 & 0.38 & 18000 & \textbf{85409} & * & \textbf{90261} & 5.38 & \textbf{18000.00} \\ 
  pcb3038 & 106928 & 112006 & 96356 & 9.89 & . & \textbf{18000.00} & 104667 & 2.11 & 917.386 & 105337 & 1.49 & 18000 & \textbf{106928} & * & \textbf{112006} & 4.53 & \textbf{18000.00} \\ 
  fl3795 & 97707 & 116792 & . & . & . & . & \textbf{97707} & * & 3158.887 & 95580 & 2.18 & 18000 & 89218 & 8.69 & \textbf{116792} & 23.61 & \textbf{18000.00} \\ 
  fnl4461 & 146995 & 152562 & . & . & . & . & - & - & - & - & - & - & \textbf{146995} & * & \textbf{152562} & 3.65 & \textbf{18000.00} \\ 
  rl5915 & 203695 & 217366 & . & . & . & . & 199336 & 2.14 & 5593.23 & 201814 & 0.92 & 18000 & \textbf{203695} & * & \textbf{217366} & 6.29 & \textbf{18000.00} \\ 
  rl5934 & 212021 & 229405 & . & . & . & . & 207385 & 2.19 & 5881.87 & 203667 & 3.94 & 18000 & \textbf{212021} & * & \textbf{229405} & 7.58 & \textbf{18000.00} \\ 
  pla7397 & 322285 & 334885 & . & . & . & . & 320744 & 0.48 & 18000 & 312645 & 2.99 & 18000 & \textbf{322285} & * & \textbf{334885} & 3.76 & \textbf{18000.00} \\ 
   \midrule\ average &  &  &  & 6.78 &  & \textbf{13749.78} &  & 1.80 & 1168.44 &  & 1.47 & 12837.26 &  & 0.34 &  & 2.24 & 14843.74 \\ 
   \bottomrule\end{tabular}

    \end{scriptsize}
  \end{table}

\end{landscape}

\end{document}